\documentclass[11pt,leqno]{amsart}
\usepackage{amsfonts}
\usepackage{color,xcolor}
\usepackage{graphicx}
\usepackage{manfnt}
\usepackage{enumerate}

\newtheorem{Thm}{Theorem}

\newtheorem{re}[Thm]{Remark}
\newtheorem{lem}[Thm]{Lemma}

 \numberwithin{equation}{section}
 \numberwithin{Thm}{section}

\newcommand{\be}{\begin{equation}}
\newcommand{\ee}{\end{equation}}
\newcommand\bes{\begin{eqnarray}}
\newcommand\ees{\end{eqnarray}}
\newcommand{\bess}{\begin{eqnarray*}}
\newcommand{\eess}{\end{eqnarray*}}

\begin{document}

\title[The Fisher-KPP equation over simple graphs]
{The Fisher-KPP equation over simple graphs: Varied persistence states in river networks}
\author[Y. Du, B. Lou, R. Peng and M. Zhou]{Yihong Du,\ \ Bendong Lou,\ \ Rui Peng\ \, and\ \, Maolin Zhou}

\thanks{{\bf Y. Du}: School of Science and Technology,
University of New England, Armidale, NSW 2341, Australia. Email: {\tt
ydu@une.edu.au}}

\thanks{{\bf B. Lou}: Mathematics and Science College, Shanghai Normal University,
Shanghai, 200234, China. Email: {\tt
lou@shnu.edu.cn}}

\thanks{{\bf R. Peng}: School of Mathematics and Statistics, Jiangsu Normal University,
Xuzhou, 221116, Jiangsu Province, China. Email: {\tt
pengrui\,$\b{}$\,seu@163.com}}

\thanks{{\bf M. Zhou}:  School of Science and Technology,
University of New England, Armidale, NSW 2341, Australia. Email: {\tt
zhouutokyo@gmail.com}}

\thanks{ We thank the two anonymous referees for their suggestions and comments which helped to improve the presentation of the paper.}
\thanks{This research  was partially supported by the Australian Research Council,
the NSF of China (No. 11671262, 11671175, 11571200), the Priority Academic Program Development
of Jiangsu Higher Education Institutions, Top-notch Academic Programs Project of Jiangsu Higher
Education Institutions (No. PPZY2015A013) and Qing Lan Project of Jiangsu Province.}

\date{\today}
\maketitle

\begin{abstract}
{ In this article, we study the dynamical behaviour of a new species spreading from a location in a river network where  two or three branches meet, based on the widely used Fisher-KPP advection-diffusion equation. This  local river system is represented by
 some simple graphs with every edge a half infinite line, meeting at a single vertex.  We obtain a rather complete description
of  the long-time  dynamical behaviour for every case under consideration, which can be classified into three different types (called a trichotomy), according to the water flow speeds in the river branches, which depend crucially on the topological structure of the graph representing  the local river system and on the cross section areas of the branches. The trichotomy includes two different kinds of persistence states, and the state called ``persistence below carrying capacity''  here appears new.

\bigskip


\smallskip

}
\end{abstract}

\setlength{\baselineskip}{16pt}{\setlength\arraycolsep{2pt}

\section{Introduction} \setcounter{equation}{0}

  The organisms living
in  a river system are subjected to the biased  flow in the downstream direction. How much stream
flow can be changed without damaging the
stream ecology,  and how stream-dwelling organisms can avoid
being washed out, are some of the key questions in
 stream ecology.  Partly motivated by these questions,
population models in rivers or
streams have gained increasing attention recently.  For example, in \cite{HL, HJL, LLL, LPL, SG}, the rivers and streams are treated as an interval on the real line, and questions on persistence and vanishing are examined via various advection-diffusion models over such an interval. As the real river systems usually also have rich topological structures,
it was argued in \cite{CGLF} that the topological structure of a river network may also greatly influence the population growth and spread of organisms living in it. Several recent papers (see, for example, \cite{JPS, R, SCA, SMA, V}) use suitable metric graphs to represent the topological structures of a river network, and study the persistence and vanishing problem by models of advection-diffusion equations over such graphs. The graphs in these works are all finite: They contain finitely many edges and vertices, and every edge has finite length. Naturally,  such finite graphs include  finite intervals as special cases.

These models over  finite graphs have several nice properties, making them very effective in analysing  the population dynamics in river systems. For instance, for a single species, with growth function  of Fisher-KPP type,  the models behave largely like the classical logistic equation, namely, the linearised eigenvalue problem at the trivial solution 0 has a principal eigenvalue, and when this eigenvalue is negative, the problem has a unique positive steady-state, which attracts all the time-dependent positive solutions, while 0 is the global attractor if this eigenvalue is nonnegative (see \cite{JPS}). Therefore, the persistence problem is reduced to the analysis of the sign of the principal eigenvalue (\cite{JPS, R, SCA, SMA}), and the long-time population profile is determined by the properties of the positive steady-state (\cite{LLL, LPL, V}).

If the finite graph is chosen to represent the entire river network, these models can be used to describe the evolution of a population over the global
network. On the other hand, if the population exists only in a local part of a complex river system, or one is only concerned with the population dynamics over such a local area, then the finite graph in the model can represent only a part of a more complex graph, if the boundary conditions are properly chosen; for example, it can be just a finite interval (\cite{LLL, LPL, SG}), or a finite $Y$-shaped graph (\cite{V}).

For a new or invasive species in a river system, it is of great interest to know how it invades the system. This kind of questions are better answered through
similar models but over unbounded spatial regions.
 Starting from the pioneering works of Fisher \cite{F} and Kolmogorove, Peterovsky and Piscunov \cite{KPP}, the spreading problem has been modelled successfully by reaction-diffusion equations over the entire Euclidean space $\mathbb R^N$. For example, if a new species spreads from the middle of a long single river branch of a possibly much bigger and complex river network, for a considerable period of time (depending on the spreading speed of the species),
 the dynamics of the species can be modelled by the following Cauchy problem:
\begin{equation}\label{tr-cauchy}
u_t-Du_{xx}+\beta u_x=f(u) \mbox{ for } x\in\mathbb R,\; t>0;\;\; u(0,x)=u_0(x) \mbox{ for } x\in\mathbb R,
\end{equation}
where $u(t,x)$ stands for the population density of the species, $D>0$ stands for its diffusion rate, and the drifting term $\beta u_x$ is caused by the river flow at a speed proportional to $\beta> 0$ in the increasing $x$ direction (hereafter we will call $\beta$ the water flow speed of the river branch, as this can be achieved by a simple rescaling), and $u_0(x)$ is a nonnegative function, not identically 0 and with compact support
(representing the initial population range),
 $f(u)$ is the growth function, typically of Fisher-KPP type. Here the entire real line $\mathbb R$ is used to represent the long  river branch in which the new species spreads initially. As we will explain in more detail below, the solution $u(t,x)$ of \eqref{tr-cauchy} quickly evolves into a wave shaped function moving with a certain speed. This kind of information about the spreading process of the species  is difficult to capture by  models over a finite spatial region.

One naturally wonders what would happen to \eqref{tr-cauchy} if the real line $\mathbb R$ is replaced by a graph, now with at least some edges of infinite length.
In this paper, we will address this question by examining the population dynamics of a new species spreading from a location in a river system where two or three long river branches meet, and so we will
consider similar advection-diffusion equations to \eqref{tr-cauchy}, but with $\mathbb R$ replaced by several simple graphs, with every edge having {\it infinite} length; see Figure 1.

\begin{figure}\label{riverbranches}
\begin{center}
\includegraphics[scale=0.70]{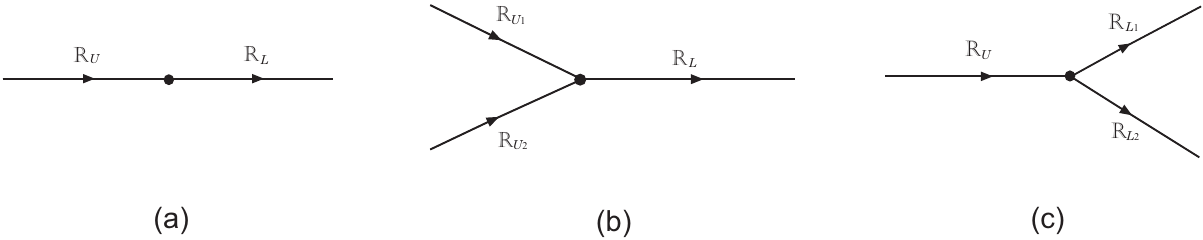}
\end{center}
\caption{\small \sf {\bf (a)} River system with lower-branch $\mathbb{R}_L$ and upper-branche $\mathbb{R}_U$.
{\bf (b)} River system with  lower-branch $\mathbb{R}_L$ and upper-branches $\mathbb{R}_{U_1}, \mathbb{R}_{U_2}$.\ \  {\bf (c)}  River system with upper-branch $\mathbb{R}_{U}$ and lower-branches $\mathbb{R}_{L_1}, \mathbb{R}_{L_2}$.}
\end{figure}

In this setting, the local river system becomes the entire world for the concerned species, and therefore, the dynamics determined by the model is only good
during the spreading phase of the species in the local system. Once the species has spread beyond the local system, into the bigger river network, its behaviour should be described by a different model; for example, by a finite graph model with the graph representing the entire river network.

Our results  reveal several new features of the population dynamics.
Firstly, for every case considered here, whenever a positive steady-state exists, there are infinitely many of them, which contrasts sharply with
the corresponding finite graph models and \eqref{tr-cauchy}. Secondly, only one of the infinitely many positive steady-states attracts the time-dependent positive solution
with any compactly supported initial population.
Thirdly, these globally attractive  steady-states can be classified into three types, completely determined by the water flow speeds in the concerned river branches.

These results indicate that the topological structure of the (local) river system, and the cross section area of every river branch, all contribute significantly to the dynamics of the population. Moreover, these contributions can be easily checked from the assumptions on the (local) river system.

In order to provide a more precise comparison, we now give some details on the dynamical behaviour of \eqref{tr-cauchy}.
 For simplicity of discussion, we will take $f(u)=u(1-u)$ below.

Under the above assumptions, the dynamical behavior of \eqref{tr-cauchy} is completely understood.
Indeed, if we define $v(t,x):=u(t, x+\beta t)$, where $u(t,x)$ is the unique solution of \eqref{tr-cauchy}, then
\begin{equation}\label{tr-v}
v_t-Dv_{xx}=f(v) \mbox{ for } x\in\mathbb R,\; t>0;\;\; v(0,x)=u_0(x) \mbox{ for } x\in\mathbb R.
\end{equation}
This is exactly the classical Fisher-KPP equation.
It is well known that
\[
\lim_{t\to\infty} v(t,x)=1 \mbox{ locally uniformly in } x\in\mathbb R.
\]
Moreover, if we denote $c_*:=2\sqrt{D}$, then the following ODE problem
\[
-D\phi_{xx}+c\phi_x=f(\phi) \mbox{ for } x\in \mathbb R,\; \phi(-\infty)=1,\; \phi(+\infty)=0,\; \phi(0)=1/2
\]
has a unique solution $\phi_c$ if $c\geq c_*$, and it has no solution if $c< c_*$.\footnote{It is also well known that \eqref{tr-cauchy} has two positive steady-states
$\phi_\beta$ and 1 if $\beta\geq c_*$, and 1 is the only positive steady-state if $0<\beta<c_*$. } \ Furthermore, there exists $C_{\pm}\in \mathbb R$ (depending on $u_0$)
such that
\begin{equation}\label{shift-v}
\lim_{t\to\infty}\left[\sup_{x\in\mathbb R_{\pm}}\Big|v(t,x)-\phi_{c_*}\Big(\pm(x-c_*t+\frac{3}{c^*}\log t)+C_{\pm}\Big)\Big|\right]=0.
\end{equation}
Therefore, roughly speaking, $v(t,x)$ behaves like the traveling wave $v_{c_*}(x-c_*t)$ with speed $c_*$ in the right direction, and like the traveling wave
$\phi_{c_*}(c_*t-x)$ in the left direction.

These well known facts may be found, for example, in \cite{AW, HNRR}.
Using \eqref{shift-v} and $u(t,x)=v(t, x-\beta t)$, we immediately see that
\[
\lim_{t\to\infty} u(t,x)=0 \mbox{ locally uniformly in } x\in\mathbb R \mbox{ if } \beta\geq c_*,
\]
and
\[
\lim_{t\to\infty} u(t,x)=1 \mbox{ locally uniformly in } x\in\mathbb R \mbox{ if }  \beta< c_*.
\]
In other words, to an observer whose position is fixed on any location of the river bank, the ultimate population density of the species
either

(I) vanishes (washed out by the waterflow) if  the waterflow speed  satisfies $\beta\geq c_*$, or

(II) stablizes at the (normalized) carrying capacity 1 if $\beta<c_*$.

\noindent
 Let us also recall from footnote 1 that \eqref{tr-cauchy} has one or two positive steady-states, depending on whether $\beta<c_*$ or $\beta\geq c_*$.

When $\mathbb R$ in \eqref{tr-cauchy} is replaced by one of the graphs in Figure 1, we will show that  the population may persist in a very different fashion to that exhibited by \eqref{tr-cauchy} above. As we already mentioned, now the problem has infinitely many positive steady-states whenever one exists, but only one of them can be the global attractor. We may now add that the global attractor in every case is a nonnegative steady-state, and these global attractors can be classified into three types, which indicates that the long-time dynamical behaviour of the problem is governed by a
 trichotomy, according to the water flow speeds of the river branches, namely,

 {\bf (i) washing out:}  if the water flow speed in every branch is no less than $c_*$, then the population will be washed out in every branch as in case (I) for \eqref{tr-cauchy},

  {\bf (ii) persistence at carrying capacity:}  if the water flow speed in every upper branch is smaller than $c_*$, then the population will persist at its carrying capacity 1 in each branch as in case (II) for \eqref{tr-cauchy},

  {\bf (iii) persistence below carrying capacity:} in all the remaining cases, the population will persist at a positive steady-state  strictly below the carrying capacity in every branch.

Let us now try to explain the above results from a biological point of view.  If the  volume of the water flow into the concerned local system per unit time is fixed, and if we assume that there is no gain or loss of water in the local system, then the water flow speed in each of the river branches in the local system is determined by the topological structure of the local system, as well as the cross section areas of the branches. For the cases in Figure 1, if every branch has a small enough cross section area, then the water flow speed in every branch will be greater than $c_*=2\sqrt{D}$ (which depends only on the diffusion rate $D$ of the species), and so  in such a case the population will be washed down stream and locally stabilise at 0 in every branch.  If the cross section areas of the upper branches are all big, then the water flow speed in every upper branch will be smaller than $c_*$, and the above result predicts that the population will locally stabilise at the carrying capacity 1 in every branch,
 regardless of the water flow speed in the lower branches. At first sight, this might be a little counter intuitive when the water flow speed is  fast in the lower branches.
However, the population in the upper branches constantly propagates into the lower branches through the junction point, and due to the infinite length of the upper
river branches, this population source is big enough to feed into the lower branches in a way that  the population stabilises locally at the carrying capacity in every branch.
This case can also be used to reveal the difference of the topological structure on the population dynamics; for example, in the extreme case that the cross section areas of the river branches are all the same,  then with the same water flow volume per unit time through the local system, the water flow speed in the upper river branches will be different in the different cases (a), (b) and (c) in Figure 1. This difference is clearly caused by the different topological structures of the river branches
in (a), (b), and (c), respectively, and has significant dynamical consequences for the population.

The remaining cases are characterised by  at least one of the branches having water flow speed below $c_*$ and at least one upper branch having water flow speed no less than $c_*$.  The former guarantees that the population can at least persist in that river branch while the latter is the reason that the population cannot stabilise at the carrying capacity locally.
 So although the fast flowing upper branch can benefit from population in slow flowing lower branches, the local population level in such a case cannot reach the carrying capacity. This contrasts to the situation in alternative (ii) above, where the local population level can reach the carrying capacity even if the water flow speed in a lower branch is above the threshold as along as the water flow speed in all upper branches is below the threshold. It is clear that the water flow direction has made the difference.

 A tricky question here is how the system selects the global attractor from infinitely many positive steady-states in the alternatives (ii) and (iii) above. The mechanism driving this selection is the compact support of the initial function. The positive steady state selected as the global attractor is always the one which decays to 0 the fastest at the proper end of the graph at infinity. This is similar in spirit to how the traveling wave with minimal speed is selected to attract the solution of the classical Fisher-KPP equation with compactly supported initial function. If the initial function does not have compact support, then in the classical Fisher-KPP problem, a traveling wave with a faster speed may be selected as the attracting profile depending on the decay rate of the initial function at infinity; similarly, if the initial function for our problem here does not have compact support, then another steady-state may be selected as the omega limit set. Indeed, if any one of the positive steady-state is chosen as the initial function, then the solution will not change with time, and so this steady-state is trivially selected as the long-time limit of the solution. Since our purpose is to understand the spreading and growth of a population starting from a bounded area in the local river system, only initial functions with compact support are of interest to us. 
\medskip

Next, we will describe the models and results more precisely, case by case. We would like to mention first  that the case of two river branches in (a) of Figure 1 is rather artificial. If there is a man made device separating a river branch into two, like a small dam or water gate, then the connecting conditions there should be different from the ones used here.  However, this is the simplest case to explain the model, and the mathematical treatment of this case already involves all the key techniques in this research. Therefore it could serve as a convenient guide for the reading of this paper. Secondly, due to the length of this paper, the question of spreading speed will not be discussed here, and is left for a future work.

\medskip

\subsection{Two river branches}

First, we consider the case that the local river system has two branches, and the waterflow in each branch has a different constant speed, but otherwise the environment is homogeneous to the concerned species. Let
$\mathbb{R}_L:=(0,+\infty)$ represent the lower river and $\mathbb{R}_U:=(-\infty,0)$
stand for the upper river. Let $w_L,\,w_U$
denote the density of the species in $\mathbb{R}_L$ and $\mathbb{R}_U$, respectively.
Then, the evolution of the species is governed by the following reaction-diffusion system:
\begin{equation}\label{two branches-a}
\begin{cases}
\partial_tw_L-D_L\partial_{xx}w_L+\beta_L\partial_xw_L=f_L(w_L), &\ x\in\mathbb{R}_L,t>0,\\
\partial_tw_U-D_U\partial_{xx}w_U+\beta_U\partial_xw_U=f_U(w_U), &\ x\in\mathbb{R}_U,t>0,\\
w_L(t,0)=w_U(t,0),&\ t>0,\\
D_La_L\partial_xw_L(t,0)=D_Ua_U\partial_xw_U(t,0),&\ t>0,\\
w_L(0,x)=w_0(x), &\ x\in \mathbb{R}_L,\\
w_U(0,x)=w_0(x),&\ x\in \mathbb{R}_U,
\end{cases}
\end{equation}
where the parameters $D_L,D_U,\beta_L,\beta_U,a_L,a_U$ are positive constants.
The constants $D_L,D_U$ are the random diffusion coefficients of the species,
$\beta_L,\beta_U$ are the advection coefficients (waterflow speeds), and $a_L,a_U$
account for the cross-section area of the river branches $\mathbb R_L$ and $\mathbb R_U$, respectively. The nonlinear reaction functions $f_L,f_U$
are assumed to be locally Lipschitz on $[0,\infty)$. The initial function $w_0$ belongs to $ C_{comp}(\mathbb{R})$,
where $C_{comp}(\mathbb{R})$ consists of continuous functions defined on $\mathbb{R}=(-\infty,\infty)$ with compact support.

Taking into account the fact that the volume of water flowing out of the upstream $\mathbb{R}_U$
is equal to that flowing into the downstream $\mathbb{R}_L$, by the conservation
of  flow at the junction point $x=0$ we have
 \begin{equation}\label{con1}
 a_L\beta_L=a_U\beta_U.
 \end{equation}
In \eqref{two branches-a}, the third line represents the natural continuity connection condition,
and the fourth line is the Kirchhoff law, which follows from  the continuity  connection condition and the conservation of flow \eqref{con1}  at the junction point $0$:
 $$
 a_L(D_L\partial_xw_L(t,0)-\beta_Lw_L(t,0))=a_U(D_U\partial_xw_U(t,0)-\beta_Uw_U(t,0)).
 $$

For the concerned species, it is reasonable to assume that its  diffusion rates and
growth rates are the same in the two river  branches. Without loss of generality,
we may assume $D_L=D_U=1$. Moreover, we assume
$f_L(w)=f_U(w)=w-w^2$ for simplicity.
Under these assumptions, problem \eqref{two branches-a} is simplified to the following one:
\begin{equation}\label{two branches}
\begin{cases}
\partial_tw_L-\partial_{xx}w_L+\beta_L\partial_xw_L=w_L-w_L^2, &\ x\in\mathbb{R}_L,t>0,\\
\partial_tw_U-\partial_{xx}w_U+\beta_U\partial_xw_U=w_U-w_U^2, &\ x\in\mathbb{R}_U,t>0,\\
w_L(t,0)=w_U(t,0),&\ t>0,\\
a_U\partial_xw_U(t,0)-a_L\partial_xw_L(t,0)=0,&\ t>0,\\
w_L(0,x)=w_0(x), &\ x\in \mathbb{R}_L,\\
w_U(0,x)=w_0(x),&\ x\in \mathbb{R}_U.
\end{cases}
\end{equation}

Given any nonnegative initial datum $w_0\in C_{comp}(\mathbb{R})$, it can be proved that
problem \eqref{two branches} admits a unique nonnegative classical solution (see Section 2 below). Our main interest is in
the long time behavior of the solution.

By a simple comparison consideration, it is easily seen that the stationary solutions $(\phi_L, \phi_U)$ of \eqref{two branches} which may determine its long-time behavior
satisfy
\begin{equation}\label{ss-0}
\left\{
\begin{array}{ll}
- \phi_L''+\beta_L  \phi_L' = \phi_L - \phi^2_L,\; 0\leq \phi_L\leq 1, & x\in \mathbb{R}_L ,\\
- \phi_U'' +\beta_U  \phi_U'= \phi_U - \phi^2_U,\; 0\leq \phi_U\leq 1, & x\in \mathbb{R}_U,\\
\phi_L (0) = \phi_U(0) =\alpha\in [0,1], & \\
a_U  \phi_U'(0) = a_L  \phi_L' (0). &
\end{array}
\right.
\end{equation}

By the maximum principle,
 $\alpha=0$ implies $(\phi_L,\phi_U)\equiv (0,0)$,
 $\alpha=1$ implies $(\phi_L,\phi_U)\equiv (1,1)$, and $\alpha\in (0, 1)$ implies
 \[
 0<\phi_L(x)<1 \mbox{ for } x\geq 0,\; 0<\phi_U(x)<1 \mbox{ for } x\leq 0.
 \]

So to have a complete understanding of \eqref{ss-0}, we only need to consider the problem
\begin{equation}\label{ss}
\left\{
\begin{array}{ll}
- \phi_L''+\beta_L  \phi_L' = \phi_L - \phi^2_L,\; 0<\phi_L<1, & x\in \mathbb{R}_L ,\\
- \phi_U'' +\beta_U  \phi_U'= \phi_U - \phi^2_U,\; 0<\phi_U<1, & x\in \mathbb{R}_U,\\
\phi_L (0) = \phi_U(0) =\alpha\in (0,1), & \\
a_U  \phi_U'(0) = a_L  \phi_L' (0). &
\end{array}
\right.
\end{equation}

The following result gives a complete description of the solutions to \eqref{ss} (for all possible cases of $\beta_L$, $\beta_U>0$).

\begin{Thm}\label{prop:2-bra}

\begin{itemize}
\item[\rm(i)] If $0< \beta_U <2$, then \eqref{ss} has no  solution for $\alpha\in (0,1)$.

\item[\rm(ii)] If $\beta_L, \beta_U \geq 2$, then for every $\alpha\in (0,1)$,
\eqref{ss} has a unique  solution.

\item[\rm(iii)]  If $\beta_U \geq  2 >\beta_L>0 $, then there exists $\alpha_0 \in (0,1)$ such that \eqref{ss} has a
unique  solution for each $\alpha\in [\alpha_0, 1)$, while \eqref{ss} has no  solution for $\alpha
\in (0,\alpha_0)$.
\item[\rm(iv)]  Whenever \eqref{ss} has a solution $(\phi_L, \phi_U)$, we have
\[
\phi'_L>0,\;\phi_U'>0,\; \phi_L(+\infty)=1,\; \phi_U(-\infty)=0.
\]
Moreover, in case {\rm(ii)} and in case {\rm(iii)} with $\alpha\in (\alpha_0, 1)$, as $x\to-\infty$, there exists some $c=c(\alpha)>0$ such that
  \begin{equation}\label{slow-u}
  \phi_U(x)=\left\{\begin{array}{ll}(c+o(1))e^{\frac{1}{2}(\beta_U-\sqrt{\beta_U^2-4}\,)x}& \mbox{ if  $\beta_U>2$,}\\
  (c+o(1))|x|e^{x}& \mbox{ if $\beta_U=2$;}
  \end{array}\right.
  \end{equation}
 while in case {\rm(iii)} with $\alpha=\alpha_0$, as $x\to-\infty$, there exists some $c>0$ such that
   \begin{equation}\label{fast-u}
 \phi_U(x)=(c+o(1))e^{\frac{1}{2}(\beta_U+\sqrt{\beta_U^2-4}\,)x}.
 \end{equation}
\end{itemize}
\end{Thm}

The long-time behavior of \eqref{two branches} is determined in the following theorem (for all possible cases of $\beta_L$, $\beta_U>0$).

\begin{Thm}\label{main:two}
Assume that  $w_0\in C_{comp}(\mathbb{R})$ is nonnegative and $w_0\not\equiv 0$.
Let $(w_L,w_U)$ be the solution of \eqref{two branches}. Then the following assertions hold.
\begin{enumerate}
\item[\rm(i)] If $0<\beta_U<2$, then $(w_L(t,\cdot),w_U(t,\cdot))\rightarrow(1,1)$ locally uniformly as $t\rightarrow \infty$.

\item[\rm(ii)] If $\beta_U, \beta_L\geq 2$, then $(w_L(t,\cdot),w_U(t,\cdot))\rightarrow(0,0)$
locally uniformly as $t\rightarrow \infty$. Moreover, $\|w_L(t,\cdot)\|_{L^\infty(\mathbb{R}_L)}\rightarrow 1$ and $\|w_U(t,\cdot)\|_{L^\infty(\mathbb{R}_U)}\rightarrow 0$ as $t\rightarrow \infty$.

\item[\rm(iii)] If $\beta_U\geq  2>\beta_L>0$, then $(w_L(t,\cdot),w_U(t,\cdot))\rightarrow(\phi_L(\cdot;\alpha_0),\phi_U(\cdot;\alpha_0))$
locally uniformly as $t\rightarrow \infty$, where $(\phi_L(\cdot;\alpha_0),\phi_U(\cdot;\alpha_0))$
is the unique solution of \eqref{ss}
 with $\alpha=\alpha_0$, determined in Theorem \ref{prop:2-bra}.

\end{enumerate}
\end{Thm}

Here, and in what follows, we say $(w_L, w_U)$ converges locally uniformly if $w_L$ converges locally uniformly in $\overline{\mathbb{R}}_L=[0,+\infty)$, and
$w_U$ converges locally uniformly in $\overline{\mathbb{R}}_U=(-\infty,0]$. The same convention will be used for similar three component functions below.

The conclusion in part (ii) of Theorem \ref{main:two} is rather natural and easy to understand: water flow speeds in both branches above the threshold level will result in the population being washed down the stream, so locally the population converges to 0. We should note, however, that the population in the upper branch satisfies $\|w_U(t,\cdot)\|_{L^\infty(\mathbb R_U)}\to 0$, while the population in the lower branch only converges to zero locally, as the $L^\infty$-norm of $w_L(t,\cdot)$ converges to 1, indicating that the population is washed down stream instead of being wiped out from the river network.

The conclusions in parts (i) and (ii) clearly indicate the importance of the upper river water flow speed over that of the lower river. Let us also note that in part (iii), the limiting stationary solution is increasing in $x$ in both branches, and the limit at $x\to-\infty$ is 0, while the limit as $x\to+\infty$ is 1. This steady-state is selected because it decays to 0  at $x=-\infty$ with the fastest rate compared to all the other positive steady-states in  this case; see part (iv) of Theorem \ref{prop:2-bra}.

\subsection{Two upper  branches and one lower  branch}
We next consider the case that the species starts in a local river system with two upper river branches  and one lower river branch. We use $\mathbb{R}_L:=(0,+\infty)$ to represent the lower branch,
and $\mathbb{R}_{U_1}:=(-\infty,0)$, $\mathbb{R}_{U_2}:=(-\infty,0)$ to stand for the two upper branches. Let $w_L,\,w_{U_1},\,w_{U_2}$
denote the density of the species in $\mathbb{R}_L,\, \mathbb{R}_{U_1}$ and $\mathbb{R}_{U_2}$, respectively.
Then, we are led to the following system:
\begin{equation}\label{main equation}
\begin{cases}
\partial_tw_{U_1}-\partial_{xx}w_{U_1}+\beta_{U_1}\partial_xw_{U_1}=w_{U_1}-w_{U_1}^2, &\ x\in\mathbb{R}_{U_1},t>0,\\
\partial_tw_{U_2}-\partial_{xx}w_{U_2}+\beta_{U_2}\partial_xw_{U_2}=w_{U_2}-w_{U_2}^2, &\ x\in\mathbb{R}_{U_2},t>0,\\
\partial_tw_L-\partial_{xx}w_L+\beta_L\partial_xw_L=w_L-w_L^2, &\ x\in\mathbb{R}_L,t>0,\\
w_L(t,0)=w_{U_1}(t,0)=w_{U_2}(t,0),&\ t>0,\\
a_L\partial_xw_L(t,0)=a_{U_1}\partial_xw_{U_1}(t,0)+a_{U_2}\partial_xw_{U_2}(t,0),&\ t>0,
\end{cases}
\end{equation}
where the parameters $\beta_L,\beta_{U_1},\beta_{U_2},a_L,a_{U_1},a_{U_2}$ are positive constants and
have the same biological interpretation as before. We also have the conservation
of the flow at the junction point $x=0$:
 \begin{equation}\label{con2}
 a_L\beta_L=a_{U_1}\beta_{U_1}+a_{U_2}\beta_{U_2}.
 \end{equation}
Regarding the initial conditions, we assume that
\begin{equation}\label{ini data}
\begin{cases}
w_i(0,x)=w_{i,0}(x) \mbox{ is nonnegative and continuous in } \mathbb{\bar R}_i,\ \ i=L, U_1, U_2,\\
w_{L,0}(0)=w_{U_1,0}(0)=w_{U_2,0}(0),\\
w_{i,0}(x)\equiv 0 \mbox{ for all large negative $x$ when $i\in\{U_1, U_2\}$},\\
 w_{L,0}(x)\equiv 0 \mbox{ for all large positive $x$}.
\end{cases}
\end{equation}

For the stationary solutions $(\phi_{U_1}, \phi_{U_2}, \phi_L)$ of \eqref{main equation}, again only the ones satisfying
$0\leq \phi_{U_1}\leq 1, 0\leq \phi_{U_2}\leq 1, 0\leq \phi_L\leq 1$ are relevant. Moreover, either
$(\phi_{U_1}, \phi_{U_2}, \phi_L)\equiv (0,0,0)$, or $(\phi_{U_1}, \phi_{U_2}, \phi_L)\equiv (1,1,1)$, or
it satisfies
\begin{equation}\label{ss-3-up}
\left\{
\begin{array}{ll}

- \phi_{U_1}'' +\beta_{U_1}  \phi_{U_1}'= \phi_{U_1} - \phi^2_{U_1},\;  0<\phi_{U_1}<1, & x\in \mathbb{R}_{U_1} := (-\infty,0),\\
- \phi_{U_2}'' +\beta_{U_2} \phi_{U_2}' = \phi_{U_2} - \phi^2_{U_2}, \;  0<\phi_{U_2}<1, & x\in \mathbb{R}_{U_2} := (-\infty,0),\\
- \phi_L'' +\beta_L  \phi_L' = \phi_L - \phi^2_L,\; \;\;\;\;\;\; \; 0<\phi_L<1, & x\in \mathbb{R}_L := (0,+\infty),\\
 \phi_{U_1}(0) = \phi_{U_2}(0) =\phi_L (0) = \alpha \in (0,1), & \\
a_{U_1} \phi_{U_1}'(0) + a_{U_2}  \phi_{U_2}' (0) = a_L \phi_L'(0). &
\end{array}
\right.
\end{equation}

A complete classification of the solutions to \eqref{ss-3-up} is given in the following theorem.

\begin{Thm}\label{uul}
\begin{itemize}
\item[(I)] If $\beta_{U_1}, \beta_{U_2} < 2$, then \eqref{ss-3-up} has no solution.
\item[(II)]
 If $ \beta_{U_1}, \beta_{U_2}, \beta_L \geq  2$, then the following hold:
\begin{itemize}
\item[\bf a.] For every $\alpha\in (0,1)$, \eqref{ss-3-up} has a continuum of solutions satisfying
 \begin{equation}\label{00}
 \phi_{U_1}'>0,\;\phi_{U_2}'>0,\; \phi_L'>0,\; \phi_{U_1}(-\infty)=\phi_{U_2}(-\infty)=0,\; \phi_L(+\infty)=1.
 \end{equation}
 \item[\bf b.] For $i=1,2$ and $j=3-i$, there exists  $\hat\alpha_{i} \in (0,1)$
 such that  for each $\alpha\in  [\hat \alpha_{i}, 1)$,  \eqref{ss-3-up} has a
unique solution satisfying
\begin{equation}\label{01}
  \phi_{U_i}'<0, \phi_{U_j}'>0,\; \phi_L'>0,\; \phi_{U_i}(-\infty)=1,\; \phi_{U_j}(-\infty)=0,\; \phi_L(+\infty)=1,
 \end{equation}
and has no such solution for $\alpha
\in (0, \hat\alpha_i)$.
\item[\bf c.] Any solution of \eqref{ss-3-up} with $\alpha\in (0,1)$ satisfies either \eqref{00} or \eqref{01}. Moreover, for any $\alpha\in (0,1)$,  there exist  $c_i=c_i(\alpha)>0$ for $i=1,2$, and a solution of  \eqref{ss-3-up}
  such that, as $x\to-\infty$, for both $i=1,2$,
   \begin{equation}\label{slow}
  \phi_{U_i}(x)=\left\{\begin{array}{ll}(c_i+o(1))e^{\frac{1}{2}(\beta_{U_i}-\sqrt{\beta_{U_i}^2-4}\,)x}& \mbox{ if  $\beta_{U_i}>2$,}\smallskip \\
  (c_i+o(1))|x|e^{x}& \mbox{ if $\beta_{U_i}=2$.}
  \end{array}\right.
  \end{equation}
\end{itemize}
\item[(III)]
 If $\beta_{U_1}, \beta_{U_2}\geq  2> \beta_L $, then the following hold:
\begin{itemize}
\item[\bf a.]
There exists $\alpha^* \in (0,1)$
 such that \eqref{ss-3-up} has a
continuum of solutions satisfying \eqref{00} for each $\alpha\in (\alpha^{*}, 1)$, has a unique solution  for $\alpha=\alpha^*$, and has no  solution for $\alpha
\in (0, \alpha^{*})$.
\item[\bf b.] For $i=1,2$ and $j=3-i$, there exist  $\hat\alpha^*_{i} \in (0,1)$ with  $\hat\alpha^*_{i}>\alpha^*$,
 such that  for each $\alpha\in  [\hat \alpha^*_{i}, 1)$,  \eqref{ss-3-up} has a
unique solution satisfying \eqref{01}, and has no such solution for $\alpha\in (0, \hat\alpha^*_{i})$.
\item[\bf c.] Any solution of \eqref{ss-3-up} with $\alpha\in (0,1)$ satisfies either \eqref{00} or \eqref{01}.

 \item[\bf d.] If $\alpha\in (\alpha^{*}, 1)$, then  for any solution of \eqref{ss-3-up} satisfying \eqref{00}, there exists $i\in\{1, 2\}$ and $c_i=c_i(\alpha)>0$ such that, as $x\to-\infty$, \eqref{slow} holds.
  \item[\bf e.]  If $\alpha=\alpha^*$, then the unique solution of \eqref{ss-3-up} has the following asymptotic expansion  as $x\to-\infty$,
 \begin{equation}\label{fast} \phi_{U_i}(x)=(c_i+o(1))e^{\frac{1}{2}(\beta_{U_i}+\sqrt{\beta_{U_i}^2-4}\,)x} \mbox{ for some } c_i=c_i(\alpha)>0,\; i=1,2.
 \end{equation}
  \end{itemize}
  \item[(IV)] If $\max\{\beta_{U_1},\beta_{U_2}\}\geq 2> \min\{\beta_{U_1},\beta_{U_2}\}$, then there exists $\alpha^{**}\in (0, 1)$
    such that
\eqref{ss-3-up} has a unique solution
 for each $\alpha\in  [\alpha^{**}, 1)$,   and has
no solution for $\alpha\in (0, \alpha^{**})$. Moreover, when $\alpha\in [\alpha^{**}, 1)$
and $\beta_{U_j}=\max\{\beta_{U_1}, \beta_{U_2}\}$, the solution $(\phi_{U_1}, \phi_{U_2},\phi_L)$ satisfies \eqref{01} with $i=3-j$,
and moreover, as $x\to-\infty$, \eqref{slow} holds for $\phi_{U_j}$ when $\alpha\in(\alpha^{**}, 1)$, and \eqref{fast} holds for $\phi_{U_j}$ when
 $\alpha=\alpha^{**}$.
 \end{itemize}
\end{Thm}

Although the set of stationary solutions of \eqref{main equation} is rich and rather complex as revealed in Theorem \ref{uul} above, the
 long time dynamics of \eqref{main equation} turns out to be relatively simple, which is given in the following theorem. (We note that in both Theorems 1.3 and 1.4, all the possible cases of $\beta_{U_1}, \beta_{U_2},\beta_L>0$ are included.) We remark that the behaviour at $x\to-\infty$ of the stationary solutions plays a crucial role in determining the long-time dynamics of \eqref{main equation}.

\begin{Thm}\label{main} Assume that the nonnegative initial data $(w_{U_1,0},w_{U_2,0}, w_{L,0})$
satisfy \eqref{ini data} and $w_{i,0}\not\equiv 0$ for some $i\in\{L,U_1,U_2\}$.
Let $(w_{U_1},w_{U_2}, w_L)$ be the solution of \eqref{main equation}. Then the following assertions hold true:

\begin{enumerate}

\item[\rm(i)] If $\beta_{U_1},\beta_{U_2}<2$, then $(w_{U_1}(t,\cdot),w_{U_2}(t,\cdot), w_L(t,\cdot))\rightarrow(1,1,1)$
locally uniformly as $t\rightarrow \infty$.

\item[\rm(ii)] If $\beta_L,\beta_{U_1},\beta_{U_2}\geq 2$, then $(w_{U_1}(t,\cdot),w_{U_2}(t,\cdot), w_L(t,\cdot))\rightarrow(0,0,0)$
locally uniformly as $t\rightarrow \infty$; moreover, $\|w_L(t,\cdot)\|_{L^\infty(\mathbb{R}_L)}\rightarrow 1$, $\|w_{U_1}(t,\cdot)\|_{L^\infty(\mathbb{R}_{U_1})}\rightarrow 0$ and  $\|w_{U_2}(t,\cdot)\|_{L^\infty(\mathbb{R}_{U_2})}\rightarrow 0$ as $t\rightarrow \infty$.

\item[\rm(iii)] If $\beta_{U_1},\beta_{U_2}\geq 2>\beta_L$, then
\[
(w_{U_1}(t,\cdot),w_{U_2}(t,\cdot), w_L(t,\cdot))\rightarrow(\phi_{U_1}(\cdot;\alpha^*),\phi_{U_2}(\cdot;\alpha^*), \phi_L(\cdot;\alpha^*))
\]
locally uniformly as $t\rightarrow \infty$, where $(\phi_{U_1}(\cdot;\alpha^*),\phi_{U_2}(\cdot;\alpha^*), \phi_L(\cdot;\alpha^*))$ is the unique  solution of
\eqref{ss-3-up}
with $\alpha=\alpha^*$.

\item[\rm(iv)] If $\beta_{U_1}\geq 2>\beta_{U_2}$, then
\[
(w_{U_1}(t,\cdot),w_{U_2}(t,\cdot), w_L(t,\cdot))\rightarrow(\psi_{U_1}(\cdot;\alpha^{**}),\psi_{U_2}(\cdot;\alpha^{**}), \psi_L(\cdot;\alpha^{**}))
\]
locally uniformly as $t\rightarrow \infty$, where $(\psi_{U_1}(\cdot;\alpha^{**}),\psi_{U_2}(\cdot;\alpha^{**}), \psi_L(\cdot;\alpha^{**}))$ is the unique  solution of \eqref{ss-3-up}
with $\alpha=\alpha^{**}$;\\
 If $\beta_{U_2}\geq 2>\beta_{U_1}$ a parallel conclusion holds $($with $U_1$ and $U_2$ interchanged in the above$)$.
\end{enumerate}
\end{Thm}

Let us note that the limiting stationary solutions in cases (iii) and (iv) have rather different behavior: In case (iii) it satisfies \eqref{00}, while \eqref{01} holds  in case (iv).  So when both the upper branches have water flow speed above the threshold level, but the lower branch water flow speed is below the threshold level, i.e., in case (iii), the population persists but stabilises at a function which is increasing in the water flow direction. However, in case (iv), when exactly one of the upper branches has water flow speed below the threshold level, then the population stabilises at a function that is increasing against the water flow direction in this branch, indicating that the population can spread up stream in that branch, while in the other river branches, the population increases in the water flow direction, regardless whether the water flow speed in the lower branch is below or above the threshold level.

\subsection{One upper branch and two lower branches}

Finally, we consider the case that the local river system consists of one upper branch $\mathbb R_U$ and two lower branches $\mathbb R_{L_1}$ and $\mathbb R_{L_2}$. In such a situation,
the problem under consideration reads as
\begin{equation}\label{split}
\begin{cases}
\partial_tw_U-\partial_{xx}w_U+\beta_U\partial_xw_U=w_U-w_U^2, &\ x\in\mathbb{R}_U:=(-\infty,0),t>0,\\
\partial_tw_{L_1}-\partial_{xx}w_{L_1}+\beta_{L_1}\partial_xw_{L_1}=w_{L_1}-w_{L_1}^2, &\ x\in\mathbb{R}_{L_1}:=(0,+\infty),t>0,\\
\partial_tw_{L_2}-\partial_{xx}w_{L_2}+\beta_{L_2}\partial_xw_{L_2}=w_{L_2}-w_{L_2}^2, &\ x\in\mathbb{R}_{L_2}:=(0,+\infty),t>0,\\
w_U(t,0)=w_{L_1}(t,0)=w_{L_2}(t,0),&\ t>0,\\
a_{L_1}\partial_xw_{L_1}(t,0)+a_{L_2}\partial_xw_{L_2}(t,0)=a_U\partial_xw_U(t,0),&\ t>0.\\

\end{cases}
\end{equation}

The corresponding initial conditions are
\begin{equation}\label{ini data-ull}
\begin{cases}
w_i(0,x)=w_{i,0}(x) \mbox{ is nonnegative and continuous in } \mathbb{\bar R}_i,\ \ i=U, L_1, L_2,\\
w_{U,0}(0)=w_{L_1,0}(0)=w_{L_2,0}(0),\\
w_{i,0}(x)\equiv 0 \mbox{ for all large positive $x$ when $i\in\{L_1, L_2\}$},\\
 w_{U,0}(x)\equiv 0 \mbox{ for all large negative $x$}.
\end{cases}
\end{equation}

Similar to the situation for \eqref{main equation}, the stationary solutions $(\phi_{U}, \phi_{L_2}, \phi_{L_2})$ of \eqref{split} that may play a role in the long-time behavor satisfy
$0\leq \phi_{L_1}\leq 1, 0\leq \phi_{L_2}\leq 1, 0\leq \phi_U\leq 1$. Moreover, either
$(\phi_{U}, \phi_{L_1}, \phi_{L_2})\equiv (0,0,0)$, or $(\phi_{U}, \phi_{L_1}, \phi_{L_2})\equiv (1,1,1)$, or
it satisfies
\begin{equation}\label{ss-1-2}
\left\{
\begin{array}{ll}
- \phi_U'' +\beta_U  \phi_U' = \phi_U - \phi^2_U,\; \;\;\;\;\;\; \; 0<\phi_U<1, & x\in \mathbb{R}_U := (-\infty, 0),\\
- \phi_{L_1}'' +\beta_{L_1}  \phi_{L_1}'= \phi_{L_1} - \phi^2_{L_1},\;  0<\phi_{L_1}<1, & x\in \mathbb{R}_{L_1} := (0, +\infty),\\
- \phi_{L_2}'' +\beta_{L_2} \phi_{L_2}' = \phi_{L_2} - \phi^2_{L_2}, \;  0<\phi_{L_2}<1, & x\in \mathbb{R}_{L_2} := (0, +\infty),\\
 \phi_{L_1}(0) = \phi_{L_2}(0) =\phi_U (0) = \alpha \in (0,1), & \\
a_{L_1} \phi_{L_1}'(0) + a_{L_2}  \phi_{L_2}' (0) = a_U \phi_U'(0). &
\end{array}
\right.
\end{equation}

\begin{Thm}\label{ull} The following results hold  for \eqref{ss-1-2}:
\begin{itemize}
\item[\rm(I)] If $\beta_{U} <2$, then \eqref{ss-1-2} has no solution.

\item[\rm(II)]  If $\beta_U, \beta_{L_1}, \beta_{L_2} \geq  2$, then   \eqref{ss-1-2} has a unique solution for every $\alpha\in (0, 1)$.

\item[\rm(III)] If $\beta_{U} \geq  2> \min\{\beta_{L_1},\beta_{L_2}\} $, then there exists $\alpha^* \in (0,1)$
 such that \eqref{ss-1-2} has a
unique solution for each $\alpha\in  [\alpha^{*}, 1)$,  and has no solution for $\alpha
\in (0, \alpha^{*})$.

 \item[\rm(IV)] Whenever \eqref{ss-1-2} has a solution $(\phi_{U}, \phi_{L_1},\phi_{L_2})$, we have
 \[
 \phi_{L_1}'>0,\;\phi_{L_2}'>0,\; \phi_U'>0,\; \phi_{L_1}(+\infty)=\phi_{L_2}(+\infty)=1,\; \phi_U(-\infty)=0.
 \]
 Moreover,  in case {\rm (III)}, as $x\to-\infty$, \eqref{slow-u} holds when $\alpha\in(\alpha^*, 1)$, and \eqref{fast-u} holds when $\alpha=\alpha^{*}$.

\end{itemize}
\end{Thm}

Our result for the long time dynamics of problem \eqref{split} (including all the possible cases of $\beta_U, \beta_{L_1}, \beta_{L_2}>0$) is stated as follows.

\begin{Thm}\label{main2}
Assume that the nonnegative initial data $(w_{U,0},w_{L_1,0},w_{L_2,0})$ satisfy \eqref{ini data-ull} and
$w_{i,0}\not\equiv 0$ for some $i\in\{U,L_1,L_2\}$. Let $(w_U,w_{L_1},w_{L_2})$ be the solution of \eqref{split}.
Then the following assertions hold:

\begin{enumerate}
\item[\rm(i)] If $\beta_U<2$, then $(w_U(t,\cdot),w_{L_1}(t,\cdot),w_{L_2}(t,\cdot))\rightarrow(1,1,1)$
locally uniformly as $t\rightarrow \infty$.

\item[\rm(ii)] If $\beta_U,\beta_{L_1},\beta_{L_2}\geq 2$, then $(w_U(t,\cdot),w_{L_1}(t,\cdot),w_{L_2}(t,\cdot))\rightarrow(0,0,0)$
locally uniformly as $t\rightarrow \infty$; moreover, $\|w_U(t,\cdot)\|_{L^\infty(\mathbb{R}_U)}\rightarrow 0$, $\|w_{L_1}(t,\cdot)\|_{L^\infty(\mathbb{R}_{L_1})}\rightarrow 1$ and  $\|w_{L_2}(t,\cdot)\|_{L^\infty(\mathbb{R}_{L_2})}\rightarrow 1$ as $t\rightarrow \infty$.

\item[\rm(iii)] If $\beta_U\geq 2>\min\{\beta_{L_1}, \beta_{L_2}\}$, then
$$
(w_U(t,\cdot), w_{L_1}(t,\cdot),w_{L_2}(t,\cdot))\rightarrow(\phi_U(\cdot;\alpha^*),\phi_{L_1}(\cdot;\alpha^*),\phi_{L_2}(\cdot;\alpha^*))
$$
locally uniformly as $t\rightarrow \infty$, where $(\phi_U(\cdot;\alpha^*),\phi_{L_1}(\cdot;\alpha^*),\phi_{L_2}(\cdot;\alpha^*))$ is the unique solution of \eqref{ss-1-2}
with $\alpha=\alpha^*$.

\end{enumerate}
\end{Thm}

We note that in this case, when the population persists below the carrying capacity in case (iii), it increases in the water flow direction, regardless of whether  one or both lower branches has water flow speed below the threshold level. In case (ii), the population in the upper river branch converges to 0 uniformly, while that in both lower branches only converges to 0 locally, indicating that the population is washed down stream instead of being wiped out from the river network.

\subsection{Remarks and comments}

We note that the number 2 plays a special role in all the main results here; this  is due to the fact that
 $2$ is the spreading speed for
 \begin{equation}\label{KPP}
 u_t-u_{xx}=u-u^2,
 \end{equation}
which is the equation governing the population growth and spread when the river system is reduced to the special case of one river branch
with 0 waterflow speed.

\begin{re}\label{re1}{\rm
From Theorems \ref{main:two}, \ref{main} and \ref{main2}, we see that the long-time dynamical behavior of the population can be described by a {\it trichotomy}:
\begin{itemize}

\item[(i)] {\bf (washing out)} if the waterflow speed in every branch of the local river system is no less than the critical speed 2, then the population is washed out  in every river branch of the local network;
\item[(ii)] {\bf (persistence at carrying capacity)} if the waterflow speed in every upper branch is  smaller than the critical speed 2, then the population in every branch of the local network goes to the normalized carrying capacity 1 as time goes to infinity;
\item[(iii)]  {\bf (persistence below carrying capacity)} in all the remaining cases, namely  at least one upper branch has waterflow speed no less than 2, and at least one other branch has waterflow speed less than 2,  the population in every river branch of the local network stablizes at a positive steady state strictly below the carrying capacity 1.
\end{itemize}
}
\end{re}
\begin{re}\label{re2}{\rm
The method  in this paper can be further developed to show that the above trichotomy remains valid for the more general situation that the local  river system has  $m$ upper  branches and $n$ lower branches meeting at a common junction point, where $m, n$ are arbitrary positive integers.
}
\end{re}

\begin{re}\label{re3}{\rm
It is possible to determine the spreading speed and profile of our solution along each river branch   by adapting the method of \cite{HNRR}. To keep the paper within a reasonable length, this is not pursued here.
Also, for simplicity, we have only considered the special growth function $f(u)=u-u^2$ in the models here. More general growth functions will be considered in a future work. }
\end{re}

\begin{re} {\rm The local river systems considered here are assumed to be homogeneous in space and in time. This is not realistic. We expect that the phenomena exhibited here persist in many heterogeneous situations, for example, when the environment is periodic in time. However, new techniques need to be developed to handle these more general and natural situations.}
\end{re}

Reaction diffusion equations over graphs arise from many other applications, and we only mention three related works here: In \cite{Y}, a general stability result is obtained for stationary solutions of such equations,  in \cite{CdRMR}, traveling wave solutions are obtained for diffusive equations over  graphs of the type mentioned in Remark \ref{re2} above, and in \cite{JM}, entire solutions on this kind of graphs are considered. We refer to the references in \cite{CdRMR, JPS, R, SCA, SMA, Y} for further works on this topic.

\subsection{Organization of the paper}
The rest of the paper is organized as follows. In section 2, we prepare some preliminary results, including
the comparison principles in the setting of river networks and the existence and uniqueness of solution of
problems \eqref{two branches}, \eqref{main equation} and \eqref{split}. In sections 3, we provide the proof of
Theorem \ref{main:two} and in section 4, we prove Theorems \ref{main} and \ref{main2}. The arguments in sections 3 and 4 are based on results for stationary solutions stated in Theorems \ref{prop:2-bra}, \ref{uul}, \ref{ull}, which are proved  in section 5 by a phase plane approach, except that
only a weaker version of Theorems \ref{uul}, \ref{ull} can be obtained by the phase plane method alone; the proof of these two theorems is completed by making use of an extra technique in section 4 (see Lemma \ref{*=*} and Remark \ref{uniq}). In section 6,  the results obtained here and their biological implications are further discussed.

\section{Preliminaries}

In this section, we will first establish the Phragm\`{e}n-Lindel\"{o}f type comparison principle
for parabolic problems in a river network. Then we derive the existence and uniqueness of solutions
to problems \eqref{two branches}, \eqref{main equation} and \eqref{split}.

We only formulate the results for \eqref{main equation}; the results for \eqref{two branches} and \eqref{split}
are parallel.

\begin{lem}\label{compprinciple} Assume that $c_i(t,x)$ is bounded on
 $[0,T]\times\mathbb{R}_i\,(i=L,U_1,U_2)$ for some $0<T<\infty$. Let $w_i\in C([0,T]\times\overline{\mathbb{R}}_i)\cap C^{1,2}((0,T]\times\overline{\mathbb{R}}_i)\,(i=L, U_1, U_2)$ satisfy
 \begin{equation}\nonumber
\begin{cases}
\partial_tw_i-\partial_{xx}w_i+\beta_i\partial_xw_i+c_i(x,t)w_i\leq0, &\ x\in\mathbb{R}_i,0<t<T,\\
w_L(t,0)=w_{U_1}(t,0)=w_{U_2}(t,0),&\ t>0,\\
a_{U_1}\partial_xw_{U_1}(t,0)+a_{U_2}\partial_xw_{U_2}(t,0)-a_L\partial_xw_L(t,0)\leq0,&\ 0<t<T,\\
w_i(x,0)\leq0,&\ 0<t<T
\end{cases}
\end{equation}
and
 \begin{equation}\label{*1}
 \liminf_{R\to\infty}\,e^{-cR^2}\Big[\max_{0\leq t\leq T,\, |x|=R}w_i(t,x)\Big]\leq0
 \end{equation}
for some positive constant $c$. Then we have
 $$
 w_i(t,x)\leq 0,\ \  \forall (t,x)\in[0,T]\times\mathbb{R}_i,\ i=L, U_1, U_2.
 $$
If additionally $w_j(0,\cdot)\leq,\not\equiv0$ for some $j\in\{L, U_1, U_2\}$, then
 $$
 w_i(t,x)<0,\ \  \forall (t,x)\in(0,T]\times\mathbb{R}_i,\ \ i=L, U_1, U_2.
 $$
\end{lem}

The proof of the first assertion of Lemma \ref{compprinciple} is the same as that of
\cite[Theorem 10, Chapter 3]{PW}, in which Theorem 7 there should be replaced by
\cite[Lemmas A.1, A2]{JPS}. The details are omitted here. Applying \cite[Lemmas A.1, A2]{JPS}
again, the second assertion of Lemma \ref{compprinciple} holds.

We next introduce the definition of supersolution and subsolution.
If $(\tilde w_L, \tilde w_{U_1}, \tilde w_{U_2})$ with $w_i\in C([0,T]\times\overline{\mathbb{R}}_i)\cap
C^{1,2}((0,T]\times\overline{\mathbb{R}}_i)\,(i=L, U_1, U_2)$ for some $T>0$ satisfies
\begin{equation}\label{supsub}
\begin{cases}
\partial_t \tilde{w}_i-\partial_{xx}\tilde{w}_i+\beta_i\partial_x\tilde{w}_i
\geq(\leq)f_i(\tilde{w}_i), &\ x\in\mathbb{R}_i,0<t<T,\\
\tilde{w}_L(t,0)=\tilde{w}_{U_1}(t,0)=\tilde{w}_{U_2}(t,0),&\ 0<t<T,\\
a_{U_1}\partial_x\tilde{w}_{U_1}(t,0)+a_{U_2}\partial_x\tilde{w}_{U_2}(t,0)-a_L\partial_x\tilde{w}_L(t,0)\geq(\leq)0,&\ 0<t<T,
\end{cases}
\end{equation}
we say $(\tilde w_L, \tilde w_{U_1}, \tilde w_{U_2})$ is a supersolution (or subsolution) of \eqref{supsub}.

Then using Lemma \ref{compprinciple}, we can conclude that

\begin{lem}\label{supsub-result}
Assume that $f_i(s)$ is locally Lipschitz $(i=L, U_1, U_2)$. Let $(\underline{w}_L, \underline{w}_{U_1}, \underline{w}_{U_2})$ and
$(\overline{w}_L, \overline{w}_{U_1}, \overline{w}_{U_2})$ be, respectively, a  bounded subsolution and
a bounded supersolution of \eqref{supsub} satisfying $\underline{w}_i(0,\cdot)\leq \overline {w}_i(0,\cdot)$
for $i=L, U_1, U_2$.
  Then, $\underline{w}_i\leq\overline w_i$ for
$i=L, U_1, U_2$. If additionally  $\underline{w}_i(0,\cdot)\leq,\not\equiv \overline{w}_i(0,\cdot)$
for some $i\in\{L, U_1, U_2\}$, then $\underline{w}_i<\overline{w}_i$ for
$i=L, U_1, U_2$ and $t\in (0, T]$.

\end{lem}

\begin{re}\label{r1}  Suppose that $\xi_1(t)$ and $\xi_2(t)$ are continuous functions of $t\in [0, T]$ with $\xi_1(t)<0<\xi_2(t)$ for $t\in [0, T]$. Then Lemma \ref{compprinciple} holds when $\mathbb R_i$ is replaced by $\mathbb R_i\cap (\xi_1(t), \xi_2(t))$ and \eqref{*1} is replaced by
\[
w_i(t, \xi_j(t))\leq 0 \mbox{ for $j=1,2$, \; $t\in [0, T]$}.
\]
Analogously, Lemma \ref{supsub-result} holds when $\mathbb R_i$ is replaced by $\mathbb R_i\cap (\xi_1(t), \xi_2(t))$ and we assume additionally
\[
\underline w_i(t,\xi_j(t))\leq\overline w_i(t,\xi_j(t)) \mbox{ for  $j=1,2$, \; $t\in [0, T]$ and $i\in\{L, U_1, U_2\}$.}
\]
\end{re}

Now we prove the main result of this section: existence and uniqueness of solution to
problem \eqref{main equation}.

\begin{Thm}\label{existence-1}
For any nonnegative initial data $(w_{L,0}, w_{U_1,0}, w_{U_2,0})$ satisfying \eqref{ini data}, problem
\eqref{main equation}  has a unique classical solution $(w_L, w_{U_1}, w_{U_2})$ which is defined
and is uniformly bounded for all $t>0$.
\end{Thm}

\begin{proof} We first show the existence and uniqueness of solution of \eqref{main equation}
by adopting the approach of \cite{vonB}. Following such an approach, we can transform
\eqref{main equation} to an equivalent half-line problem of the form (9.1)-(9.3) in \cite{vonB}
defined for $(t,x)\in(0,T)\times(0,\infty)$ with compactly supported initial data.
Then the standard theory guarantees that such an equivalent problem
(and so the original problem \eqref{main equation}) admits a unique classical solution,
 defined for all time $t>0$.

It remains to show the uniform boundedness of $w_i$, $i\in\{L, U_1, U_2\}$. For such a purpose,
given any $\ell>0$, we consider the following auxiliary problem
with zero Dirichlet boundary conditions:
\begin{equation}\label{Dirichlet}
\begin{cases}
\partial_tw_L-\partial_{xx}w_L+\beta_L\partial_xw_L=w_L(1-w_L), &\ x\in(0,\ell),\\
\partial_tw_i-\partial_{xx}w_i+\beta_i\partial_xw_i=w_i(1-w_i),\; i\in\{U_1, U_2\}, &\ x\in(-\ell,0),\\
w_L(t,0)=w_{U_1}(t,0)=w_{U_2}(t,0),&\ t>0,\\
w_L(t,\ell)=w_{U_1}(t,-\ell)=w_{U_2}(t,-\ell)=0,&\ t>0,\\
a_L\partial_xw_L(t,0)=a_{U_1}\partial_xw_{U_1}(t,0)+a_{U_2}\partial_xw_{U_2}(t,0),&\ t>0,\\
w_L(0,x)=w_{L,0}(x), &\ x\in (0,\ell),\\
w_{i}(0,x)=w_{i,0}(x),\; i\in\{U_1, U_2\},&\ x\in (-\ell,0).
\end{cases}
\end{equation}
By \cite[Lemma A.7]{JPS}, \eqref{Dirichlet} has a unique classical solution,
denoted by $(w_L^\ell, w_{U_1}^\ell, w_{U_2}^\ell)$, which is defined globally in $t$.
Thanks to \cite[Lemmas A.3, A.6]{JPS}, the  $(w_i^\ell)_{i=L, U_1, U_2}$ is
nondecreasing with respect to $\ell$, and for all $\ell>0$ it holds
 \bes
 \label{bound}
 0\leq w_i^\ell\leq 1+\sup_{\mathbb{R}_L}w_{L,0}+\sup_{\mathbb{R}_{U_1}}w_{U_1,0}+\sup_{\mathbb{R}_{U_2}}w_{U_2,0},\ i=L, U_1, U_2.
 \ees
Thus, the limit $\lim_{\ell\to\infty}(w_i^\ell(t,x))_{i=L, U_1, U_2}$ exists,
denoted by $(w_i^\infty(t,x))_{i=L, U_1, U_2}$.

On the other hand, following the same procedure as in \cite{vonB}, one can transform
\eqref{Dirichlet} to a well-stated initial-boundary value problem of
the form (9.1)-(9.3) in \cite{vonB}. Then, in light of \eqref{bound}, applying the standard
interior $L^p$ and Schauder estimates to such a parabolic system and then coming
back to the original problem \eqref{Dirichlet}, we can conclude that,
given constants $0<\epsilon_0<1$ and $\ell_0>0$,
 $$
 \|w_L^\ell\|_{C^{1+\frac{\alpha}{2},2+\alpha}([\epsilon_0,\frac{1}{\epsilon_0}]\times[0,\ell_0])},\
 \|w_{U_1}^\ell\|_{C^{1+\frac{\alpha}{2},2+\alpha}([\epsilon_0,\frac{1}{\epsilon_0}]\times[-\ell_0,0])},\
 \|w_{U_2}^\ell\|_{C^{1+\frac{\alpha}{2},2+\alpha}([\epsilon_0,\frac{1}{\epsilon_0}]\times[-\ell_0,0])}\leq C_0,
 $$
for some positive constants $\alpha\in(0,1)$ and $C_0$ with $C_0$ being independent of $\ell$ once $\ell>\ell_0$.
Therefore, through a standard diagonal process,  together with the compact embedding theorem,
we see that $w_i^\ell$ converges to $w_i^\infty$ locally in the usual $C^{1,2}$ norm ($i=L, U_1, U_2$),
and $(w_i^\infty)_{i=L, U_1, U_2}$ is a classical solution of \eqref{main equation} (the initial condition can be easily checked separately).

Thus, by uniqueness we must have $(w_i)_{i=L, U_1, U_2}=(w_i^\infty)_{i=L, U_1, U_2}$.
In view of \eqref{bound}, it is clear that
 \bes
 \nonumber
 0\leq w_i\leq 1+\sup_{\mathbb{R}_L}w_{L,0}+\sup_{\mathbb{R}_L}w_{L,0}+\sup_{\mathbb{R}_{U_2}}w_{U_2,0},\ \forall i\in\{L, U_1, U_2\}.
 \ees
That is, $(w_i)_{i=L, U_1, U_2}$ is uniformly bounded. The proof is thus complete.
\end{proof}

Similar arguments as above show that problems \eqref{two branches} and \eqref{split} admit
a unique classical and uniformly bounded solution for given nonnegative initial data.

\section{The two-branches problem}

 We prove Theorem \ref{main:two}} in this section by making use of Theorem \ref{prop:2-bra}. The proof of Theorem \ref{prop:2-bra}
 is given in section 5 by a phase plane approach, which is rather long and very different in nature to the techniques used here.

 According to the  behavior of the solution, we distinguish
three cases:

\smallskip

\ \ \ {\bf (i):}\ $2>\beta_U$;\ \ {\bf (ii):}\ $\beta_L, \beta_U\geq 2$;\ \ {\bf (iii):}\  $\beta_U\geq 2>\beta_L$.

\smallskip
\noindent
Clearly these three cases exhaust all the possible cases of the positive parameters $\beta_U$ and $\beta_L$.

\subsection{Case (i):\, $2>\beta_U$}   In this case, we have
\begin{Thm}\label{(i)}
Assume that $\beta_U<2$, and the initial function $w_0\in C_{comp}(\mathbb{R})$ is nonnegative and $w_0\not\equiv 0$.
Let $(w_L,w_U)$ be the solution of \eqref{two branches}. Then
 $$
 (w_L(t,\cdot),w_U(t,\cdot))\rightarrow(1,1)\ \ \ \mbox{locally uniformly as $t\rightarrow \infty$.}
  $$
\end{Thm}

\begin{proof}  First of all, following the proof of Theorem \ref{existence-1} we have
 \bes
 \label{bound-two}
 0\leq w_i\leq 1+\sup_{\mathbb{R}}w_{0},\ \ (i=L, U).
 \ees
As $\|w_0\|_{L^\infty(\mathbb{R})}>0$, we take
$\bar{u}(t)=1+\|w_0\|_{L^\infty(\mathbb{R})}e^{-t}$. Clearly,
$\bar{u}(0)=1+\|w_0\|_{L^\infty(\mathbb{R})}>w_0(x)$ and
$\partial_t\bar{u}(t)=1-\bar{u}(t)\geq\bar{u}(t)(1-\bar{u}(t))$. Thus,
$(\bar{u},\bar{u})$ is a supersolution of problem \eqref{two branches}.
So Lemma \ref{supsub-result}, together with \eqref{bound-two}, gives
 \bes
 \label{a1}
 1=\lim_{t\rightarrow \infty}\bar{u}(t)\geq
 \limsup_{t\rightarrow \infty}\|w_i(t,\cdot)\|_{L^\infty(\mathbb{R}_i)},\ \ (i=L, U).
 \ees

Since $w_0\not\equiv 0$ is nonnegative, $w_U(1,x)>0$ for $x\in\mathbb{R}_U$.
Because of $0<\beta_U<2$, by standard results on logistic equations,  we know that
there exists a unique constant $l_0>0$ such that the following problem
 \begin{equation}\label{wl}
 \begin{cases}
 -w''+\beta_Uw'=w(1-w) ,& x\in(-l,0),\\
 w(0)=w(-l)=0
 \end{cases}
 \end{equation}
 has a positive solution if and only if $l>l_0$, and the positive solution $w_l$ is unique and satisfies $\|w_l\|_\infty\to 0$ as $l\to l_0$.
 Therefore by fixing $l>l_0$ close to $l_0$, we can make sure that the unique solution $w_l$ of the above problem satisfies
 \[
 w_l(x)<w_U(1,x)\;  \mbox{ for }  x\in[-l,0].
 \]

We set $w_l^0=w_l$ on $[-l,0]$ and $w_l^0=0$ on $(-\infty,-l)$. Then let  $(\underline{w}_L,\underline{w}_U)$
be the solution of \eqref{two branches} with initial function $(0,w^0_l)$. Clearly, $\underline{w}_L(t,x)>0$ for $t>0, x\geq 0$, and $\underline{w}_U(t,x)>0$ for $t>0, x\leq 0$. Moreover,
one can use the standard parabolic comparison  principle to conclude that
$\underline{w}_U(t,x)\geq w_l^0(x)$ for all $(t,x)\in[0,\infty)\times[-l,0]$. Hence, we have
 $$
 (\underline{w}_L(t,\cdot),\underline{w}_U(t,\cdot))\geq(0,w^0_l)\ \
 \mbox{in}\ \mathbb{R}_L\times\mathbb{R}_U,\ \ \forall t>0.
 $$
Thus for any $\delta>0$,
 $$
 (\underline{w}_L(\delta,\cdot),\underline{w}_U(\delta,\cdot))\geq(0,w^0_l)\ \
 \mbox{in}\ \mathbb{R}_L\times\mathbb{R}_U.
 $$
It follows from Lemma \ref{supsub-result} that
 $$
 (\underline{w}_L(t+\delta,\cdot),\underline{w}_U(t+\delta,\cdot))\geq(\underline{w}_L(t,\cdot),\underline{w}_U(t,\cdot))\ \
 \mbox{in}\ \mathbb{R}_L\times\mathbb{R}_U,\ \ \forall t,\delta>0,
 $$
which implies that $(\underline{w}_L,\underline{w}_U)$ is nondecreasing with respect to $t$.

Denote $(\underline{w}_{L,\infty}(x),\underline{w}_{U,\infty}(x))
=(\lim_{t\rightarrow\infty}\underline{w}_L(t,x),\lim_{t\rightarrow\infty}\underline{w}_U(t,x))$. Then,
similarly as in the proof of Theorem \ref{existence-1}, a compactness argument allows us to conclude that
 $$
 (\underline{w}_L(t,\cdot),\underline{w}_U(t,\cdot))\to(\underline{w}_{L,\infty}(\cdot),\underline{w}_{U,\infty}(\cdot))\ \
 \mbox{locally uniformly as }\ t\to\infty,
 $$
and $(\underline{w}_{L,\infty}(x),\underline{w}_{U,\infty}(x))$ is a positive  stationary solution
 to \eqref{two branches} satisfying $0<\underline{w}_{L,\infty}\leq 1,\; 0<\underline{w}_{U,\infty}\leq 1$.
 By (i) of Theorem \ref{prop:2-bra}, $(1,1)$ is the unique positive stationary solution of equation \eqref{two branches}  in case (i). Therefore we necessarily have
 \bes
 \label{a2}
 (\underline{w}_L(t,\cdot),\underline{w}_U(t,\cdot))\to(1,1)\ \
 \mbox{locally uniformly as}\ t\to\infty.
 \ees

On the other hand, as $w^0_l(x)<w_U(1,x)$ for all $x\in(-\infty,0)$, it follows from Lemma \ref{supsub-result} that
 $$
 (\underline{w}_L(t,\cdot),\underline{w}_U(t,\cdot))\leq(w_L(t+1,\cdot),w_U(t+1,\cdot))\ \
 \mbox{in}\ \mathbb{R}_L\times\mathbb{R}_U,\ \ \forall t\geq0.
 $$
This and \eqref{a2} yield
 \bes
 \label{a3}
 (\liminf_{t\rightarrow \infty}{w}_L(t,\cdot),\liminf_{t\rightarrow \infty}{w}_U(t,\cdot))\geq(1,1)\ \
 \mbox{locally uniformly as}\ t\to\infty.
 \ees
Combining \eqref{a1} and \eqref{a3}, we obtain the desired result, and the proof is complete.
\end{proof}

%
%

\subsection{Case (ii):\, $\beta_L,\beta_U\geq 2$}
By Theorem \ref{prop:2-bra}(iii),
 for any given $\alpha\in(0,1)$, \eqref{ss-3-up}
has a unique  solution $(\phi_L(\cdot,\alpha),\phi_U(\cdot,\alpha))$, and both
$\phi_L(x,\alpha)$ and $\phi_U(x,\alpha)$ are increasing in $x$.
We will use $(\phi_L,\phi_U)$ to construct a suitable supersolution to
establish the desired asymptotic behavior of $(w_L,w_U)$.

\begin{Thm}\label{(iii)} Assume that $\beta_U, \beta_L\geq 2$, and the initial datum $w_0\in C_{comp}(\mathbb{R})$
is nonnegative and $w_0\not\equiv 0$.  Let $(w_L,w_U)$ be the solution of \eqref{two branches}.
Then
 $$
 (w_L(t,\cdot),w_U(t,\cdot))\rightarrow(0,0)\ \ \ \mbox{locally uniformly as $t\rightarrow \infty$.}
  $$
Moreover, $\|w_L(t,\cdot)\|_{L^\infty(\mathbb{R}_L)}\rightarrow 1$ and $\|w_U(t,\cdot)\|_{L^\infty(\mathbb{R}_U)}\rightarrow 0$ as $t\rightarrow \infty$.
\end{Thm}

\begin{proof}
Fix $\alpha \in(0,1)$, we construct a supersolution in the following manner:
$$\begin{cases}
\bar{w}_L(t,x):=\phi_L(x;\alpha)+Me^{-\lambda t}e^{\frac{\beta_L}{2}x},& x\in\mathbb{R}_L,t\geq0,\\
\bar{w}_U(t,x):=\phi_U(x;\alpha)+Me^{-\lambda t}e^{\frac{\beta_U}{2}x},& x\in\mathbb{R}_U,t\geq0,\\
\end{cases}$$
where $M$ and $\lambda$ will be determined later, and the region for $x$ will also be suitably further restricted. Clearly,
 $$
 a_L\partial_x\bar{w}_L(t,0)=a_U\partial_x\bar{w}_U(t,0),\ \ \forall t>0.
 $$

 To check that $\bar{w}_L$ satisfies the supersolution conditions, we  calculate
 \begin{eqnarray*}
 &&\ \ \ \partial_t\bar{w}_L-\partial_{xx}\bar{w}_L+\beta_L\partial_x\bar{w}_L-\bar{w}_L+\bar{w}_L^2\\
 &&= -\phi''_L+\beta_L\phi'_L-\phi_L+\phi_L^2
 +\Big(2\phi_L-\lambda-1+\frac{\beta_L^2}{4}\Big)Me^{-\lambda t}e^{\frac{\beta_L}{2}x}
 +\Big(Me^{-\lambda t}e^{\frac{\beta_L}{2}x}\Big)^2\\
 &&\geq \Big[2\phi_L(0;\alpha)+\frac{\beta_L^2}{4}-1-\lambda\Big]Me^{-\lambda t}e^{\frac{\beta_L}{2}x} \;\mbox{ for } x\in (0,\infty),\; t>0.
 \end{eqnarray*}
 Here, we have used the monotonicity of $\phi_L(x;\alpha)$ in $x$. Similarly we have, for any $l>0$,
  \begin{eqnarray*}
 &&\ \ \ \partial_t\bar{w}_U-\partial_{xx}\bar{w}_U+\beta_U\partial_x\bar{w}_U-\bar{w}_U+\bar{w}_U^2\\
 &&\geq \Big(2\phi_U+\frac{\beta_U^2}{4}-1-\lambda\Big)Me^{-\lambda t}e^{\frac{\beta_U}{2}x}\\
 &&\geq \Big(2\phi_U(l;\alpha)+\frac{\beta_U^2}{4}-1-\lambda\Big)Me^{-\lambda t}e^{\frac{\beta_U}{2}x} \mbox{ for } x\in[-l, 0),\; t>0.
  \end{eqnarray*}

  Define
  \[ v(x):=M e^{kx} \mbox{ with } k=\frac 12 {\beta_U^2+\sqrt{\beta_U^2-4}}.
  \]
  Clearly
  \[
-v''+\beta_U v'=v\geq v-v^2.
\]
By Theorem  \ref{prop:2-bra}(iv), as $x\to-\infty$, $\phi_U(x;\alpha)$ satisfies \eqref{slow-u}.
It follows that there exists $l=l_\alpha>0$ such that
\[
\phi_U(x;\alpha)\geq v(x) \mbox{ for } x\leq - l.
\]

We now fix $l=l_\alpha$.
Because $\beta_U, \beta_L\geq 2$, $\phi_L(0;\alpha)=\alpha >0$ and $2\phi_U(-l;\alpha)>0$, by our earlier calculations, we can always find
a sufficiently small $\lambda>0$ (depending on $l$ but independent of $M$) such that
\[
 \partial_t\bar{w}_L-\partial_{xx}\bar{w}_L+\beta_L\partial_x\bar{w}_L-\bar{w}_L+\bar{w}_L^2\geq 0
\mbox{ for $x\in (0,+\infty)$, $t>0$,}
\]
 and
\[
\partial_t\bar{w}_U-\partial_{xx}\bar{w}_U+\beta_U\partial_x\bar{w}_U-\bar{w}_U+\bar{w}_U^2\geq 0 \mbox{ for } x\in [-l, 0), t>0.
\]

 Since $w_0\in C_{comp}(\mathbb{R})$, we can choose large $M>\max\{1, \|w_0\|_\infty\}$ such that $v(x)>w_0(x)$ in $\mathbb R$ and
  $$\bar{w}_i(0,x)>w_i(0,x)\ \ \mbox{  for  } x\in \mathbb{R}_i,\ \ i=L, U.$$
 Since $v(0)=M>\max\{1, \|w_0\|_\infty\}\geq w_U(t,0)$ for all $t\geq 0$, by the standard comparison principle we deduce
\[
w_U(t,x)\leq v(x) \mbox{ for } x\leq 0, t>0.
\]
It follows in particular that
\[
\bar w_U(t,-l)>\phi_U(-l;\alpha)\geq v(-l)\geq w_U(t,-l) \mbox{ for all } t\geq 0.
\]

Let $\displaystyle \xi(t):=\frac{2\lambda}{\beta_L}t$. Then
\[
\bar w_L(t, \xi(t))>M\geq w_L(t,\xi(t)) \mbox{ for all $t>0$}.
\]
Thus with $\lambda$ and $M$ fixed as above,  $(\bar{w}_L,\bar{w}_U)$ is a supersolution of \eqref{two branches} over the region $x\in[-l, \xi(t)]$ and $t\geq 0$. It follows from Remark 2.3   that
\[
w_L(t,x)\leq \bar w_L(t,x) \mbox{ for } x\in [0, \xi(t)], t>0;\; w_U(t,x)\leq \bar w_U(t,x) \mbox{ for } x\in [-l, 0],\; t>0.
\]
Since
\[
w_U(t,x)\leq v(x)\leq \phi_U(x;\alpha)<\bar w_U(t,x) \mbox{ for } x\leq -l,\; t>0,
\]
we thus have
\[
w_U(t,x)\leq \bar w_U(t,x) \mbox{ for } x\leq 0,\; t>0.
 \]
 Therefore
   $$
 \limsup_{t\rightarrow \infty}w_i(t,x)\leq \lim_{t\rightarrow\infty}\bar{w}_i(t,x)=\phi_i(x;\alpha)
 $$
locally uniformly for $x\in \mathbb{R}_i$, $i=L, U$. We further notice that $\lim_{\alpha\rightarrow0}\phi_i(x;\alpha)=0$  locally uniformly for $x\in\mathbb{R}_i$ and $i=L, U$. Then due to the arbitrariness of $\alpha$, we infer that
  $$
  \lim_{t\rightarrow \infty}w_i(t,x)=0\ \ \mbox{locally uniformly  for  } x\in \overline{\mathbb{R}}_i,\ \ i=L, U.
  $$

\medskip

Since $\bar w_U$ is increasing in $x$, it is easily seen from the above proof  that $\|w_U(t,\cdot)\|_{L^\infty(\mathbb R_U)}\to 0$ as $t\to+\infty$.
  To determine the limit of $\|w_L(t,\cdot)\|_{L^\infty(\mathbb{R}_L)}$, let us consider the auxiliary problem
 $$
 \begin{cases}
 \partial_t\underline{u}-\partial_{xx}\underline{u}+\beta_L\partial_x\underline{u}=\underline{u}-\underline{u}^2, & x\in\mathbb{R}_L,t>0,\\
 \underline{u}(t,0)=0, & t>0,\\
 \underline{u}(0,x)=\mbox{min}\{x,w_0(x)\}\geq,\not\equiv0, &x\in \mathbb{R}_L.
 \end{cases}$$
Obviously, $\underline{u}$ is a subsolution of the equation satisfied by $w_L$. Thus, $\underline{u}\leq w_L$ in $[0,\infty)\times\mathbb{R}_L$. From the equation satisfied by $\underline u(t, x+\beta_L t)$,  it is easily seen that $\|\underline{u}(t,\cdot)\|_{L^\infty(\mathbb{R}_L)}\rightarrow 1$ as $t\rightarrow \infty$, so in turn, $
\liminf_{t\to\infty}\|w_L(t,\cdot)\|_{L^\infty(\mathbb{R}_L)}\geq  1$. A simple comparison argument (involving an ODE solution) shows that $\limsup_{t\to\infty}\|w_L(t,\cdot)\|_{L^\infty(\mathbb{R}_L)}\leq  1$. We thus obtain $\lim_{t\to\infty}\|w_L(t,\cdot)\|_{L^\infty(\mathbb{R}_L)}=  1$.
\end{proof}

\subsection{Case (iii):\ $\beta_U\geq 2>\beta_L$} In this case, by Theorem \ref{prop:2-bra}(iii) there exists a constant $\alpha_0\in(0,1)$ such that equation \eqref{ss} has a unique solution  $(\phi_L(x;\alpha),\phi_U(x;\alpha))$
when $\alpha\in[\alpha_0,1)$ and has no solution when $\alpha\in(0,\alpha_0)$.

\begin{Thm}\label{(iv)}
Assume that $\beta_U\geq 2>\beta_L$ and the initial datum $w_0\in C_{comp}(\mathbb{R})$ is nonnegative and $w_0\neq 0$.  Let $(w_L,w_U)$ be the solution of \eqref{two branches}. Then we have
 \begin{equation}\label{UL}
 (w_L(t,\cdot),w_U(t,\cdot))\rightarrow (\phi_L(\cdot;\alpha_0),\phi_U(\cdot;\alpha_0))\ \ \mbox{locally uniformly as $t\rightarrow \infty$}.
 \end{equation}
\end{Thm}
\begin{proof}
For $M>1$, set
\[
k:=\frac{\beta_U+\sqrt{\beta_U^2-4}}{2}, \;\; \bar \phi_{U, 0}(x):=Me^{kx},\;\; \bar \phi_{L, 0}(x):=M.
\]
Since $w_0$ has compact support, we can fix $M>1$ large enough such that
\[
\bar \phi_{U, 0}(x)=Me^{kx}>w_0(x) \mbox{ for } x\leq 0,\; \bar \phi_{L, 0}(x)=M>w_0(x) \mbox{ for } x\geq 0.
\]
Clearly
\[
-\bar \phi_{U,0}''+\beta_U\bar\phi'_{U,0}=\bar \phi_{U,0}>\bar \phi_{U, 0}-\bar \phi_{U, 0}^2 \mbox{ for } x<0,
\]
\[
-\bar \phi_{L,0}''+\beta_L\bar\phi'_{L,0}=0> M(1-M)=\bar \phi_{L, 0}-\bar \phi_{L, 0}^2 \mbox{ for } x>0.
\]
Moreover,
\[
a_U\bar \phi_{U,0}'(0)-a_L\bar \phi_{L,0}'(0)=Ma_U k>0.
\]
Therefore $(\bar \phi_{U,0}, \bar \phi_{L,0})$ is a super solution of the corresponding elliptic problem of  \eqref{two branches}.
It follows that the unique solution $(\bar \phi_{U}(t,x), \bar \phi_{L}(t,x))$ of  \eqref{two branches} with initial function
$(\bar \phi_{U,0}, \bar \phi_{L,0})$ is nonincreasing in $t$, and as $t\to+\infty$,
\[
(\bar \phi_{U}(t,\cdot), \bar \phi_{L}(t,\cdot))\to (\hat\phi_U(\cdot),\hat \phi_L(\cdot))
\]
in $C^2_{loc}(\overline{\mathbb R}_U)\times C^2_{loc}(\overline{\mathbb R}_L)$, and $(\hat\phi_U, \hat\phi_L)$ is a nonnegative stationary solution of  \eqref{two branches}.
Clearly
\[
\hat\phi_U(x)\leq \bar\phi_{U, 0}(x)=Me^{kx} \mbox{ for } x\leq 0,\; \hat \phi_L(x)\leq  \bar\phi_{U, 0}(x)=M \mbox{ for } x\geq 0.
\]
By a simple comparison consideration involving an ODE, we also easily see that
\[
0\leq \hat\phi_U\leq 1,\; 0\leq \hat \phi_L\leq 1.
\]

Since $\beta_L<2$, there exists a unique $l_0>0$ such that  the problem
\[
-\phi''+\beta_L\phi'=\phi-\phi^2 \mbox{ in } (0, l),\; \phi(0)=\phi(l)=0
\]
has a unique positive solution $\phi_l(x)$ if and only if $l>l_0$, and $\|\phi_l\|_\infty\to 0$ as $l\to l_0$.
Therefore we can fix $l>l_0$ such that $\phi_l(x)<\min\{1, \phi_L(1,x)\}$ for $x\in [0, l]$.
Define
\[
\underline \phi_{U, 0}(x)\equiv 0 \mbox{ for } x\leq 0,\; \underline \phi_{L, 0}(x)=\phi_l(x) \mbox{ for } x\in [0, l],\;
\underline \phi_{L, 0}(x)=0 \mbox{ for } x\geq l.
\]
Then it is easily checked that $(\underline \phi_{U, 0}, \underline \phi_{L, 0})$ is a subsolution of the corresponding elliptic problem of \eqref{two branches}. It follows that the unique solution $(\underline \phi_{U}(t,x), \underline \phi_{L}(t,x))$ of  \eqref{two branches} with initial function
$(\underline \phi_{U,0}, \underline \phi_{L,0})$ is nondecreasing in $t$, and as $t\to+\infty$,
\[
(\underline \phi_{U}(t,\cdot), \underline \phi_{L}(t,\cdot))\to (\tilde \phi_U(\cdot), \tilde \phi_L(\cdot))
\]
in $C^2_{loc}(\overline{\mathbb R}_U)\times C^2_{loc}(\overline{\mathbb R}_L)$, and $(\tilde \phi_U, \tilde \phi_L)$ is a positive stationary solution of  \eqref{two branches}.

Since
\[
\underline \phi_{U,0}< \bar \phi_{U,0},\; \underline \phi_{L,0}<\bar \phi_{L,0},
\]
we have
\[
0<\tilde \phi_U(x)\leq \hat\phi_U(x)\leq M e^{kx} \mbox{ for } x\leq 0, \;\;0<\tilde \phi_L(y)\leq \hat \phi_L(y) \mbox{ for } y\geq 0.
\]
It follows that  $(\hat\phi_U(x), \hat\phi_L(y))$ and $(\tilde \phi_U(x), \tilde \phi_L(y))$ must be solutions of \eqref{ss} with some $\alpha=\hat \alpha, \tilde \alpha\in (0, 1)$, respectively, and $\hat \alpha\geq \tilde \alpha$.

We claim that $\hat \alpha=\tilde\alpha=\alpha_0$, with $\alpha_0\in (0, 1)$ given by Theorem \ref{prop:2-bra}. Indeed, by part (iv) of this theorem, if $\hat \alpha>\alpha_0$,
then there exists some $\hat c>0$ such that as $x\to-\infty$,
\[
\hat \phi_U(x)=\left\{\begin{array}{ll} (\hat c+o(1))|x| e^{x} &\mbox{  if  } \beta_U=2,\\
(\hat c+o(1))e^{\frac 12 (\beta_U-\sqrt{\beta_U^2-4}\,)x} & \mbox{  if  } \beta_U>2,
\end{array}\right.
\]
which is a contradiction to $\hat \phi_U(x)\leq M e^{\frac 12 (\beta_U+\sqrt{\beta_U^2-4}\,)x}$ for all $x\leq 0$. Therefore $\hat \alpha\leq \alpha_0$. Since \eqref{ss} has no solution for $\alpha\in (0, \alpha_0)$, we
necessarily have $\hat \alpha\geq \tilde\alpha\geq \alpha_0$. Hence $\hat \alpha=\tilde\alpha=\alpha_0$ and
 $(\hat \phi_U, \hat \phi_L)=(\tilde \phi_U, \tilde \phi_L)$ coincides with the unique solution of \eqref{ss} with $\alpha=\alpha_0$, i.e., $(\phi_U(\cdot; \alpha_0), \phi_L(\cdot;\alpha_0))$.

Using
\[
\bar \phi_U(0, x)> w_0(x) \mbox{ for } x\leq 0,\; \bar \phi_L(0, x)> w_0(x) \mbox{ for } x\geq 0,
\]
we also obtain
\[
\bar \phi_U(t,x)\geq w_U(t,x) \mbox{ for } x\leq 0,\; \bar \phi_L(t, x)\geq w_L(t,x) \mbox{ for } x\geq 0.
\]
It follows that
\[
\limsup_{t\to+\infty} w_U(t,x)\leq \phi_U(x;\alpha_0) \mbox{ locally uniformly for } x\leq 0,
\]
\[
\limsup_{t\to+\infty} w_L(t,x)\leq \phi_L(x;\alpha_0) \mbox{ locally uniformly for } x\geq 0.
\]

Similarly, using
\[
\underline \phi_U(0,x)=0< w_U(1,x) \mbox{ for } x\leq 0,\; \underline \phi_L(0, x)< w_L(1,x) \mbox{ for } x\geq 0
\]
we obtain
\[
\underline \phi_U(t,x)\leq w_U(t+1,x) \mbox{ for } x\leq 0,\; \underline \phi_L(t, x)\leq w_L(t+1,x) \mbox{ for } x\geq 0,
\]
and hence
\[
\liminf_{t\to+\infty} w_U(t,x)\geq \phi_U(x;\alpha_0) \mbox{ locally uniformly for } x\leq 0,
\]
\[
\liminf_{t\to+\infty} w_L(t,x)\geq \phi_L(x;\alpha_0) \mbox{ locally uniformly for } x\geq 0.
\]
Therefore \eqref{UL} holds.
\end{proof}

\vskip8pt
Clearly Theorem \ref{main:two} is a consequence of the above theorems in this section.

\section{The three-branches problems}

In this section, we prove
 Theorems \ref{main} and \ref{main2} based on Lemmas \ref{I}, \ref{II}, \ref{III}, \ref{IV} and \ref{prop:1-2-weak} proved in section 5 by a phase plane approach. These latter results form a weaker version of Theorems \ref{uul} and \ref{ull}, and we will  use a new technique here to improve these results from section 5 to complete the proof of Theorems \ref{uul} and \ref{ull}; see Lemma \ref{*=*} and Remark \ref{uniq} below.

\subsection{Proof of Theorem \ref{main}}\ This theorem has four conclusions,  corresponding to  the following  four cases:
\[
\mbox{ {\bf (i):}\ $\beta_{U_1}, \beta_{U_2}<2$,\ \ {\bf (ii):}\ $\beta_L, \beta_{U_1},\beta_{U_2}\geq 2$, \;  {\bf (iii):}\ $\beta_{U_1}, \beta_{U_2}\geq 2>\beta_L$,}
\] \vspace{-0.6cm}
\[\mbox{

{\bf (iv):}\ $\max\{\beta_{U_1}, \beta_{U_2}\}\geq 2>\min\{\beta_{U_1}, \beta_{U_2}\}$.}
\]
Clearly these exhaust all the possible cases of the positive parameters $\beta_{U_1}, \beta_{U_2}$ and $\beta_L$.

\subsubsection{Proof of {\bf (i)}}
We borrow  the ideas in the proof  of Theorem \ref{(i)}.
Since $\beta_{U_1},\beta_{U_2}<2$, for $i=1,2$, we can similarly find $l_{i, 0}>0$ such that \eqref{wl} has a positive solution if and only if $l>l_{i,0}$. Then fix $l_i>l_{i,0}$
but close to $l_{i,0}$ so that
\[
w_{l_i}<w_{U_i}(1,x) \mbox{ for } x\in [-l_i, 0].
\]
Then define, for $i=1,2$,
\[
w_{l_i}^0=w_{l_i} \mbox{ on } [-l_i, 0],\; w_{l_i}^0=0 \mbox{ on } (-\infty, -l_i).
\]
Let $(\underline w_{U_1}, \underline w_{U_2}, \underline w_{L})$ be the solution of \eqref{main equation} with initial function $(w_{l_1}^0, w_{l_2}^0, 0)$.
Then a similar comparison argument shows that $(\underline w_{U_1}, \underline w_{U_2}, \underline w_{L})$ is nondecreasing in $t$, and as $t\to+\infty$,
\[
(\underline w_{U_1}(t,\cdot), \underline w_{U_2}(t,\cdot), \underline w_{L}(t,\cdot))\to (\underline w_{U_1, \infty}(\cdot), \underline w_{U_2,\infty}(\cdot), \underline w_{L,\infty}(\cdot))
\]
in $L^\infty_{loc}(\mathbb R_{U_1})\times L^\infty_{loc}(\mathbb R_{U_2})\times L^\infty_{loc}(\mathbb R_{L})$. Moreover,
$(\underline w_{U_1, \infty}, \underline w_{U_2,\infty}, \underline w_{L,\infty})$ is a positive stationary solution of \eqref{main equation} satisfying
\[
0<\underline w_{L, \infty}\leq 1, \; 0<\underline w_{U_i,\infty}\leq 1,\; i=1,2.
\]
By Theorem \ref{uul} (I), we easily see that $(\underline w_{U_1, \infty}, \underline w_{U_2,\infty}, \underline w_{L,\infty})\equiv (1,1,1)$.

By the choice of the initial function of $(\underline w_{U_1}, \underline w_{U_2}, \underline w_{L})$ and the comparison principle, we have
\[
(\underline w_{U_1}(t,\cdot), \underline w_{U_2}(t,\cdot), \underline w_{L}(t,\cdot))\leq ( w_{U_1}(t+1,\cdot),  w_{U_2}(t+1,\cdot),  w_{L}(t+1,\cdot)) \mbox{ for } t\geq 0.
\]
It follows that
\[
\liminf_{t\to+\infty}( w_{U_1}(t,\cdot),  w_{U_2}(t,\cdot),  w_{L}(t,\cdot))\geq (1,1,1)\; \mbox{ locally uniformly}.
\]

If we define $\bar u$ as in the proof of Theorem \ref{(i)} then $(\bar u, \bar u, \bar u)$ is a super solution of \eqref{main equation}, from which we deduce
\[
\limsup_{t\to+\infty}( w_{U_1}(t,\cdot),  w_{U_2}(t,\cdot),  w_{L}(t,\cdot))\leq (1,1,1)\; \mbox{ locally uniformly}.
\]
We thus obtain
\[
\lim_{t\to+\infty}( w_{U_1}(t,\cdot),  w_{U_2}(t,\cdot),  w_{L}(t,\cdot))= (1,1,1) \; \mbox{ locally uniformly}.
\]
This completes the proof of conclusion (i). \hfill $\Box$

\subsubsection{Proof of {\bf (ii)}}
 We use the ideas  in the proof of Theorem \ref{(iii)}. By Lemma  \ref{II}, there exists $\epsilon>0$ sufficiently small such that
 for each $\alpha\in (0, \epsilon)$, \eqref{ss-3-up} has a solution $(\phi_{U_1}(x;\alpha), \phi_{U_2}(x;\alpha), \phi_{L}(x;\alpha))$ satisfying
 \eqref{slow}.

   We define
$$\begin{cases}
\bar{w}_L(t,x):=\phi_L(x;\alpha)+Me^{-\lambda t}e^{\frac{\beta_L}{2}x},& x\in\mathbb{R}_L,t\geq0,\\
\bar{w}_{U_i}(t,x):=\phi_{U_i}(x;\alpha)+Me^{-\lambda t}e^{\frac{\beta_{U_i}}{2}x},& x\in\mathbb{R}_{U_i},t\geq0,\; i=1,2,\\
\end{cases}$$
with $M$ and $\lambda$ positive constants to be determined later.
 It is easily checked that
 $$
 a_L\partial_x\bar{w}_L(t,0)=a_{U_1}\partial_x\bar{w}_{U_1}(t,0)+a_{U_2}\partial_x\bar{w}_{U_2}(t,0),\ \ \forall t>0.
 $$
   For $i=1,2$, define
  \[ v_i(x):=M e^{k_ix} \mbox{ with } k_i=\frac 12 {\beta_{U_i}^2+\sqrt{\beta_{U_i}^2-4}}.
  \]
  Clearly
  \[
-v_i''+\beta_{U_i} v_i'=v_i\geq v_i-v_i^2.
\]
It follows from \eqref{slow} that there exists $l_i=l_i(\alpha)>0$ such that
\[
\phi_{U_i}(x;\alpha)\geq v_i(x) \mbox{ for } x\leq -l_i.
\]

By the calculations in the proof of Theorem \ref{(iii)}, we can see that there exists $\lambda>0$ small, independent of $M$, such that for such $\lambda$ and $i\in\{1, 2\}$,
\[
 \partial_t\bar{w}_{U_i}-\partial_{xx}\bar{w}_{U_i}+\beta_{U_i}\partial_x\bar{w}_{U_i}-\bar{w}_{U_i}+\bar{w}_{U_i}^2\geq 0 \mbox{ for } x\in [-l_i, 0),\; t>0,
\]
and
 \[
 \partial_t\bar{w}_L-\partial_{xx}\bar{w}_L+\beta_L\partial_x\bar{w}_L-\bar{w}_L+\bar{w}_L^2\geq 0 \mbox{ for } x>0, t>0.
 \]

 We now choose $M>0$ large enough such that $M>\max\{1, \|w_{i,0}\|_\infty, i=U_1, U_2, L\}$ , $v_i(x)>w_{U_i,0}(x)$ in $\mathbb R_{U_i}$ for $i=1,2$, and
  $$\bar{w}_j(0,x)>w_j(0,x)\ \ \mbox{  for  } x\in \mathbb{R}_j,\ \ j=L, U_1, U_2.$$
 Then
 \[
 \mbox{ $v_i(0)=M>\max\{1, \|w_{j,0}\|_\infty, j=U_1, U_2, L\}\geq \max\{\|w_i\|_\infty: i=L, U_1, U_2\}\geq w_{U_i}(t,0)$ for all $t\geq 0$,}
 \]
 and by the standard comparison principle we deduce
\[
w_{U_i}(t,x)\leq v_i(x) \mbox{ for } x\leq 0, t>0,\; i=1,2.
\]
It follows in particular that
\[
\bar w_{U_i}(t,-l_i)>\phi_{U_i}(-l_i;\alpha)\geq v_i(-l_i)\geq w_{U_i}(t,-l_i) \mbox{ for all } t\geq 0, \; i=1,2.
\]
Let $\displaystyle \xi(t):=\frac{2\lambda}{\beta_L}t$. Then
\[
\overline w_L(t,\xi(t))>M\geq w_{L}(t,\xi(t)) \mbox{ for } t>0,
\]
and so
  $(\bar{w}_L(t,x),\bar{w}_{U_1}(t,y), \bar w_{U_2}(t,z))$ is a supersolution of \eqref{main equation} over the region $
(x,y,z)\in [0,\xi(t)] \times [-l_1, 0]\times [-l_2, 0]$ and $t\geq 0$. It follows that
\[
w_L(t,x)\leq \bar w_L(t,x) \mbox{ for } x\in [0,\xi(t)], t>0;\; w_{U_i}(t,x)\leq \bar w_{U_i}(t,x) \mbox{ for } x\in [-l_i, 0],\; t>0,\; i=1,2.
\]
Since
\[
w_{U_i}(t,x)\leq v_i(x)\leq \phi_{U_i}(x;\alpha)<\bar w_{U_i}(t,x) \mbox{ for } x\leq -l_i,\; t>0,,\; i=1,2,
\]
we thus have
\[
w_{U_i}(t,x)\leq \bar w_{U_i}(t,x) \mbox{ for } x\leq 0,\; t>0,\; i=1,2.
 \]
 Therefore
   $$
 \limsup_{t\rightarrow \infty}w_i(t,x)\leq \lim_{t\rightarrow\infty}\bar{w}_i(t,x)=\phi_i(x;\alpha)
 $$
locally uniformly for $x\in \overline{\mathbb{R}}_i$, $i=L, U_1, U_2$. We further notice that $\lim_{\alpha\rightarrow0}\phi_i(x;\alpha)=0$  locally uniformly for $x\in\mathbb{R}_i$ and $i=L, U_1, U_2$. Then due to the arbitrariness of $\alpha$, we infer that
  $$
  \lim_{t\rightarrow \infty}w_i(t,x)=0\ \ \mbox{locally uniformly  for  } x\in \overline{\mathbb{R}}_i,\ \ i=L, U_1, U_2.
  $$

The conclusions on the large time behavior of $\|w_j(t,\cdot)\|_{L^\infty(\mathbb R_j)}$, $j\in\{U_1, U_2, L\}$, can be proved in the same way as in the proof of Theorem \ref{(iii)}.
This completes the proof for {\bf (ii)}. \hfill $\Box$

\medskip

\subsubsection{Proof of {\bf (iii)}} We first improve the conclusion in Lemma \ref{III} by showing that $\alpha_1^*=\alpha_2^*$.
\begin{lem}\label{*=*}
In Lemma \ref{III}, $\alpha_1^*=\alpha_2^*$.
\end{lem}

\begin{proof}
Arguing indirectly we assume that $\alpha_1^*<\alpha_2^*$, and denote the unique solution of \eqref{ss-3-up} with $\alpha=\alpha_1^*$ and $\alpha=\alpha_2^*$  by $(\phi_1,\phi_2,\phi_3)$ and $(\psi_1,\psi_2,\psi_3)$, respectively. By \eqref{fast}, there exists positive constants
$c_1, c_2, \hat c_1,\hat c_2$ such that, as $x\to-\infty$,
\begin{equation}\label{expansion-infty}
\phi_i(x)=(c_i+o(1)) e^{k_i x},\; \psi_i(x)=(\hat c_i+o(1)) e^{k_i x},\;\; i=1,2,
\end{equation}
where $k_i=\frac{\beta_{U_i}+\sqrt{\beta_{U_i}^2-4}}{2}$ for $i=1,2$.
Recall that we also have
\[
\phi_3(+\infty)=\psi_3(+\infty)=1.
\]
These facts imply the existence of a positive constant $M>1$ such that
\[
M(\phi_1,\phi_2,\phi_3)\geq (\psi_1,\psi_2,\psi_3),
\]
by which we mean
\[
M\phi_i(x)\geq \psi_i(x) \mbox{ for } x\leq 0,\;i=1,2,\;\;  M\phi_3(x)\geq \psi_3(x) \mbox{ for } x\geq 0.
\]
It follows that
\[
M\alpha_1^*\geq \alpha_2^* \mbox{ and hence } M\geq \frac{\alpha_2^*}{\alpha_1^*}>1.
\]
Define
\[
M_*:=\inf\{M>0: M(\phi_1,\phi_2,\phi_3)\geq (\psi_1,\psi_2,\psi_3)\}.
\]
Then
\[
M_*(\phi_1,\phi_2,\phi_3)\geq (\psi_1,\psi_2,\psi_3),\; M_*\geq \frac{\alpha_2^*}{\alpha_1^*}>1.
\]
Moreover, it is easily checked that $M_*(\phi_1,\phi_2,\phi_3)$ is a super solution to \eqref{ss-3-up} with $\alpha=\alpha^*_2$.
Since $M_*\phi_3(+\infty)-\psi_3(+\infty)=M_*-1>0$, by the comparison principle we have
\[
M_*\phi_i(x)> \psi_i(x) \mbox{ for } x\leq 0,\;i=1,2,\;\;  M_*\phi_3(x)> \psi_3(x) \mbox{ for } x\geq 0.
\]
Set
\[
\theta_i(x):=M_*\phi_i(x)-\psi_i(x),\; i=1,2,3.
\]
A simple calculation yields
\[
-\theta_i''+\beta_{U_i}\theta_i'=(1-M_*\phi_i-\psi_i)\theta_i+(M_*^2-M_*)\phi_i^2\geq (1-M_*\phi_i-\psi_i)\theta_i,\; i=1,2.
\]
Denote
\[
\epsilon_i(x)=M_*\phi_i(x)+\psi_i(x),\; i=1,2,
\]
and let $\bar \theta_i(t,x)$ be the unique solution of
\[
\begin{cases}
\partial_t\bar \theta_i-\partial_{xx}\bar\theta_i+\beta_{U_i}\partial_x\bar\theta_i=(1-\epsilon_i(x))\bar\theta_i,& x<0,\; t>0,\\
\bar\theta_i(t,0)=\theta_i(0),\bar\theta_i(t,-\infty)=0, & t>0,\\
\bar\theta_i(0,x)=\theta_i(x), & x\leq 0.
\end{cases}
\]
Then $\bar\theta_i(t,x)$ is nonincreasing in $t$ and as $t\to+\infty$, $\bar\theta_i(t,x)\to \hat\theta_i(x)$ and $\hat\theta_i(x)$ satisfies
\[
\begin{cases}
-\hat\theta_i''+\beta_{U_i}\hat\theta_i'=(1-\epsilon_i(x))\hat\theta_i,& x<0,\\
\hat\theta_i(0)=\theta_i(0), \hat\theta_i(-\infty)=0,&\\
0\leq \hat\theta_i(x)\leq \theta_i(x), & x\leq 0.
\end{cases}
\]
Since $\theta_i(0)>0$, by the maximum principle we have $\hat\theta_i(x)>0$ for $x\leq 0$.
By \eqref{expansion-infty} we have, as $x\to-\infty$,
\begin{equation}\label{hat-theta}
0<\hat\theta_i(x)\leq \theta_i(x)=(M_*c_i-\hat c_i+o(1))e^{k_i x},\; \epsilon_i(x)=(M_*c_i+\hat c_i+o(1))e^{k_i x}.
\end{equation}
By standard ODE theory, there exist fundamental solutions $\Theta_i^1(x)$ and $\Theta_i^2(x)$ of the linear ODE
\[
-\theta''+\beta_{U_i}\theta'=(1-\epsilon_i(x))\theta
\]
such that, as $x\to-\infty$,
\[
\Theta_i^1(x)=(1+o(1))e^{k_i x},\; \Theta_i^2(x)=\left\{\begin{array}{ll}(1+o(1))e^{k^-_ix} &\mbox{ if } \beta_{U_i}>2,\\
(1+o(1))|x|e^{ x} & \mbox{ if } \beta_{U_i}=2,
\end{array}\right.
\]
where
\[
k_i^-:=\frac{\beta_{U_i}-\sqrt{\beta_{U_i}^2-4}}{2}, \; i=1,2.
\]
It follows that
\[
\hat\theta_i=a_i \Theta_i^1+b_i\Theta_i^2,\;\; i=1,2.
\]
for some constants $a_i$ and $b_i$. By \eqref{hat-theta}, we necessarily have $b_i=0$ and $a_i>0$.
We thus obtain, as $x\to-\infty$,
\[
\hat\theta_i(x)=(a_i+o(1))e^{k_i x} \mbox{ for } i=1,2.
\]
It follows from this and \eqref{expansion-infty} that, for some $L_1>0$ large and $\tilde \epsilon_1>0$ small,
\[
\theta_i(x)\geq \hat\theta_i(x)\geq \tilde\epsilon_1 \psi_i(x) \mbox{ for } x\leq -L_1, i=1,2.
\]
Due to $\theta_i(x)>0$ on $[-L_1, 0]$ and the  continuity of $\theta_i$, we can find $\tilde\epsilon_2>0$ small so that
\[
\theta_i(x)\geq \tilde\epsilon_2 \psi_i(x) \mbox{ for } x\in[ -L_1, 0],  i=1,2.
\]
We thus have
\[
\theta_i(x)\geq \epsilon \psi_i(x) \mbox{ for } x\in (-\infty,0],  i=1,2,\; \epsilon:=\min\{\tilde\epsilon_1,\tilde\epsilon_2\}.
\]
Since $\theta_3(x):=M_*\phi_3(x)-\psi_3(x)\to M_*-1>0$ as $x\to+\infty$, by shrinking $\epsilon>0$ further if needed, we have
\[
\theta_3(x)\geq  \epsilon \psi_3(x) \mbox{ for } x\geq 0.
\]
It follows that
\[
(M_*-\epsilon)(\phi_1,\phi_2,\phi_3)\geq (\psi_1,\psi_2,\psi_3),
\]
a contradiction to the definition of $M_*$. This proves $\alpha_1^*=\alpha_2^*$.
\end{proof}
\begin{re}\label{uniq}
The method in the above proof of Lemma \ref{*=*} can be easily extended to show that
 $\hat \alpha_{i,1}=\hat \alpha_{i,2}$ in Lemma \ref{II},
 $\hat \alpha_{i,1}^*=\hat \alpha_{i,2}^*$ $(i=1,2)$ in Lemma \ref{III},  $\alpha_1^{**}=\alpha_2^{**}$ in Lemma \ref{IV}, and $\alpha_1^*=\alpha_2^*$ in Lemma \ref{prop:1-2-weak}.
Thus Theorems \ref{uul} and \ref{ull} are consequences of the combination of Lemma \ref{*=*} and the method used in its proof, and Lemmas \ref{II}, \ref{III}, \ref{IV}, \ref{prop:1-2-weak}.
\end{re}

Denoting $\alpha^*:=\alpha_1^*=\alpha_2^*$, we can now adapt the approach in the proof of Theorem \ref{(iv)} to complete the proof of conclusion (iii).

For $M>1$ and $i\in\{1,2\}$, set
\[
k_i:=\frac{\beta_{U_i}+\sqrt{\beta_{U_i}^2-4}}{2}, \;\; \bar \phi_{U_i, 0}(x):=Me^{k_ix},\;\; \bar \phi_{L, 0}(x):=M.
\]
In view of \eqref{ini data}, we can fix $M>1$ large enough such that
\[
\bar \phi_{U_i, 0}(x)=Me^{k_ix}>w_{U_i,0}(x) \mbox{ for } x\leq 0,\; \bar \phi_{L, 0}(x)=M>w_{L, 0}(x) \mbox{ for } x\geq 0.
\]
Let $(\phi_{U_1}(\cdot;\alpha^*), \phi_{U_2}(\cdot;\alpha^*),\phi_{L}(\cdot;\alpha^*))$ denote the unique solution of \eqref{ss-3-up}
with $\alpha=\alpha^*$. In view of \eqref{fast}, by enlarging $M$ further if needed, we may also assume that
\[
\bar \phi_{U_i, 0}(x)=Me^{k_ix}\geq \phi_{U_i}(x;\alpha^*) \mbox{ for } x\leq 0,\; i=1,2,\;\;
\bar \phi_{L, 0}(x)=M>\phi_{L, 0}(x: \alpha^*) \mbox{ for } x\geq 0.
\]

Clearly
\[
-\bar \phi_{U_i,0}''+\beta_{U_i}\bar\phi'_{U_i,0}=\bar \phi_{U_i,0}>\bar \phi_{U_i, 0}-\bar \phi_{U_i, 0}^2 \mbox{ for } x<0,
\]
\[
-\bar \phi_{L,0}''+\beta_L\bar\phi'_{L,0}=0> M(1-M)=\bar \phi_{L, 0}-\bar \phi_{L, 0}^2 \mbox{ for } x>0.
\]
Moreover,
\[
\xi_1a_{U_1}\bar \phi_{U_1,0}'(0)+\xi_2a_{U_2}\bar \phi_{U_2,0}'(0) -a_L\bar \phi_{L,0}'(0)=M(\xi_1a_{U_1}k_1+\xi_2a_{U_2}k_2)>0.
\]
Therefore $(\bar \phi_{U_1,0},\bar \phi_{U_2,0}, \bar \phi_{L,0})$ is a super solution of the corresponding elliptic problem of  \eqref{main equation}.
It follows that the unique solution $(\bar \phi_{U_1}(t,x), \bar \phi_{U_2}(t,x), \bar \phi_{L}(t,x))$ of  \eqref{main equation} with initial data
$(\bar \phi_{U_1,0}, \bar \phi_{U_2,0}, \bar \phi_{L,0})$ is nonincreasing in $t$, and as $t\to+\infty$,
\[
(\bar \phi_{U_1}(t,\cdot), \bar \phi_{U_1}(t,\cdot), \bar \phi_{L}(t,\cdot))\to (\hat\phi_{U_1}(\cdot), \hat\phi_{U_2}(\cdot), \hat \phi_L(\cdot))
\]
in $C^2_{loc}(\overline{\mathbb R}_{U_1})\times C^2_{loc}(\overline{\mathbb R}_{U_2})\times C^2_{loc}(\overline{\mathbb R}_L)$, and $(\hat\phi_{U_1},\hat\phi_{U_2}, \hat\phi_L)$ is a nonnegative stationary solution of  \eqref{main equation} satisfying
\[
(\hat\phi_{U_1},\hat\phi_{U_2}, \hat\phi_L)\geq (\phi_{U_1}(\cdot;\alpha^*), \phi_{U_2}(\cdot;\alpha^*),\phi_{L}(\cdot;\alpha^*)).
\]
Clearly
\begin{equation}\label{hat-i}
\hat\phi_{U_i}(x)\leq \bar\phi_{U_i, 0}(x)=Me^{k_ix} \mbox{ for } x\leq 0,\; i=1,2,\;\;  \hat \phi_L(x)\leq  \bar\phi_{L, 0}(x)=M \mbox{ for } x\geq 0.
\end{equation}
By a simple comparison consideration involving an ODE, we also easily see that
\[
0\leq \hat\phi_{U_i}\leq 1,\; 0\leq \hat \phi_L\leq 1.
\]
These facts imply that $(\hat\phi_{U_1},\hat\phi_{U_2}, \hat\phi_L)$ is a solution of \eqref{ss-3-up} with some $\alpha=\hat\alpha\in [\alpha^*, 1)$.
Due to \eqref{hat-i}, this solution is not of type \eqref{01}, and so it is of type \eqref{00}. If $\hat\alpha>\alpha^*$, by \eqref{slow} we obtain a contradiction to \eqref{hat-i}. Thus we necessarily have $\hat\alpha=\alpha^*$ and
\[
(\hat\phi_{U_1},\hat\phi_{U_2}, \hat\phi_L)=(\phi_{U_1}(\cdot;\alpha^*), \phi_{U_2}(\cdot;\alpha^*),\phi_{L}(\cdot;\alpha^*)).
\]

Since $\beta_L<2$, there exists a unique $l_0>0$ such that  the problem
\[
-\phi''+\beta_L\phi'=\phi-\phi^2 \mbox{ in } (0, l),\; \phi(0)=\phi(l)=0
\]
has a unique positive solution $\phi_l(x)$ if and only if $l>l_0$, and $\|\phi_l\|_\infty\to 0$ as $l\to l_0$.
Therefore we can fix $l>l_0$ such that $\phi_l(x)<\min\{1, \phi_L(1,x)\}$ for $x\in [0, l]$.
Define
\[
\underline \phi_{U_i, 0}(x)\equiv 0 \mbox{ for } x\leq 0,\; i=1,2,\;\;  \underline \phi_{L, 0}(x)=\phi_l(x) \mbox{ for } x\in [0, l],\;
\underline \phi_{L, 0}(x)=0 \mbox{ for } x\geq l.
\]
Then it is easily checked that $(\underline \phi_{U_1, 0}, \underline \phi_{U_2, 0}, \underline \phi_{L, 0})$ is a subsolution of the corresponding elliptic problem of \eqref{main equation}. It follows that the unique solution $(\underline \phi_{U_1}(t,x), \underline \phi_{U_1}(t,x),\underline \phi_{L}(t,x))$ of  \eqref{main equation} with initial function
$(\underline \phi_{U_1,0}, \underline \phi_{U_2,0},\underline \phi_{L,0})$ is nondecreasing in $t$, and as $t\to+\infty$,
\[
(\underline \phi_{U_1}(t,\cdot), \underline \phi_{U_2}(t,\cdot), \underline \phi_{L}(t,\cdot))\to (\tilde \phi_{U_1}(\cdot), \tilde \phi_{U_2}(\cdot), \tilde \phi_L(\cdot))
\]
in $C^2_{loc}(\overline{\mathbb R}_{U_1})\times C^2_{loc}(\overline{\mathbb R}_{U_2})\times C^2_{loc}(\overline{\mathbb R}_L)$, and $(\tilde \phi_{U_1}, \tilde \phi_{U_2}, \tilde \phi_L)$ is a positive stationary solution of  \eqref{main equation}.

Since
\[
\underline \phi_{U_i,0}< \bar \phi_{U_i,0} \mbox{ for } i=1,2,\; \underline \phi_{L,0}<\bar \phi_{L,0},
\]
we also have
\[
0<\tilde \phi_{U_i}(x)\leq \hat\phi_{U_i}(x)\leq M e^{k_i x} \mbox{ for } x\leq 0,\; i=1,2,\;\;  \;\;0<\tilde \phi_L(y)\leq \hat \phi_L(y) \mbox{ for } y\geq 0.
\]
It follows that  $(\tilde\phi_{U_1}(x), \tilde\phi_{U_2}(x), \tilde\phi_L(y))$ is a solution of \eqref{ss-3-up} with some $ \tilde \alpha\in (0, \alpha^*]$.
By Theorem \ref{uul}, we necessarily have  $\tilde\alpha=\alpha^*$, and hence
\[
(\tilde \phi_U, \tilde \phi_L)=(\phi_{U_1}(\cdot; \alpha^*),  \phi_{U_2}(\cdot; \alpha^*), \phi_L(\cdot;\alpha^*)).
\]

Using
\[
\bar \phi_{U_i}(0, x)> w_{U_i,0}(x) \mbox{ for } x\leq 0,\; i=1,2,\;\; \bar \phi_L(0, x)> w_{L, 0}(x) \mbox{ for } x\geq 0,
\]
we  obtain
\[
\bar \phi_{U_i}(t,x)\geq w_{U_i}(t,x) \mbox{ for } x\leq 0,\; i=1,2,\;\; \bar \phi_L(t, x)\geq w_L(t,x) \mbox{ for } x\geq 0.
\]
It follows that
\[
\limsup_{t\to+\infty} w_{U_i}(t,x)\leq \phi_{U_i}(x;\alpha^*) \mbox{ locally uniformly for } x\leq 0, i=1,2,
\]
and
\[
\limsup_{t\to+\infty} w_L(t,x)\leq \phi_L(x;\alpha^*) \mbox{ locally uniformly for } x\geq 0.
\]

Similarly, using
\[
\underline \phi_{U_i}(0,x)=0< w_{U_i}(1,x) \mbox{ for } x\leq 0,\; i=1,2,\; \underline \phi_L(0, x)< w_L(1,x) \mbox{ for } x\geq 0
\]
we obtain
\[
\underline \phi_{U_i}(t,x)\leq w_{U_i}(t+1,x) \mbox{ for } x\leq 0,\; i=1,2,\; \underline \phi_L(t, x)\leq w_L(t+1,x) \mbox{ for } x\geq 0,
\]
and hence
\[
\liminf_{t\to+\infty} w_{U_i}(t,x)\geq \phi_{U_i}(x;\alpha^*) \mbox{ locally uniformly for } x\leq 0,\; i=1,2,
\]
\[
\liminf_{t\to+\infty} w_L(t,x)\geq \phi_L(x;\alpha^*) \mbox{ locally uniformly for } x\geq 0.
\]
The desired asymptotic behavior of $(w_{U_1}, w_{U_2}, w_L)$ thus follows.
\hfill$\Box$

\subsubsection{Proof of {\bf (iv)}} This is similar to the proof of (iii) above, and we only sketch the main steps. We only consider the case $\beta_{U_1}\geq 2>\beta_{U_2}$.

Firstly we prove $\alpha_1^{**}=\alpha_2^{**}$ in Lemma \ref{IV} by the method of Lemma \ref{*=*} as indicated in Remark \ref{uniq}, where
 the behavior of $\phi_{U_1}(x)$ as $x\to-\infty$ is crucial to deduce a contradiction.

Next we prove the convergence result by constructing super and sub-solutions similarly, where we
treat   $\phi_{U_1}$ as  $\phi_{U_i}$ in the proof of  (iii),  and treat
$\phi_{L}$ and $\phi_{U_2}$ as $\phi_L$  in the proof of  (iii).

Since the modifications required in the arguments are obvious,
the details are omitted. \hfill $\Box$

\subsection
{Proof of Theorem \ref{main2}.}\ We may prove $\alpha_1^*=\alpha_2^*$ in Lemma \ref{prop:1-2-weak} by the method of Lemma \ref{*=*} as indicated in Remark \ref{uniq}, where  the behavior of $\phi_{U}(x)$ as $x\to-\infty$ is crucial to deduce a contradiction.

The rest of the proof is parallel to that of Theorem \ref{main:two}. More precisely,
the proof of conclusion (i) is similar to that of Theorem \ref{(i)}, where in the construction of the subsolution, we treat $w_{L_i}$ the same way as $w_L$ there, and treat $w_U$ the same. Here instead of using Theorem \ref{prop:2-bra} (i) we use Theorem \ref{ull} (I).

The proof of (ii) is similar to that of Theorem \ref{(iii)}, where $w_{L_1}$ and $w_{L_2}$ are treated in the same way as $w_L$ there, and Theorem \ref{ull} (II) is used instead of Theorem \ref{prop:2-bra} (ii).

The proof of (iii) is similar to that of Theorem \ref{(iv)} with $w_{L_1}$ and $w_{L_2}$  treated in the same way as $w_L$ there,
and Theorem \ref{ull} (III)  used instead of Theorem \ref{prop:2-bra} (iii).
\hfill$\Box$


\section{Stationary Solutions}

In this section we study the stationary problems \eqref{ss}, \eqref{ss-3-up} and \eqref{ss-1-2} by a phase plane approach. We start with \eqref{ss}.
For our analysis below, it is convenient to use the following change of variables to reduce \eqref{ss} to an equivalent system which
can be conveniently treated by a phase plane argument. We define
\[
\Phi_L(x)=\phi_L(\beta_L^{-1}x),\; \Phi_U(x)=\phi_U(\beta_U^{-1}x).
\]
Then by a simple calculation, using $a_U\beta_U=a_L\beta_L$, \eqref{ss} is reduced to
\begin{equation}\label{SS}
\left\{
\begin{array}{ll}
-\Phi_L'' + \Phi_L' = \beta_L^{-2}(\Phi_L - \Phi^2_L),\; 0<\Phi_L<1, & x\in \mathbb{R}_L,\\
-\Phi_U'' + \Phi_U' = \beta_U^{-2}(\Phi_U - \Phi^2_U),\; 0<\Phi_U<1, & x\in \mathbb{R}_U,\\
\Phi_L (0) = \Phi_U(0) =\alpha\in (0, 1), & \\
\Phi_U'(0) = \Phi_L'(0). &
\end{array}
\right.
\end{equation}

We will use a phase plane argument to solve \eqref{SS}. For $\mu>0$,  consider the equation
\begin{equation}\label{phi-eq}
-\phi''+\phi'=\mu(\phi-\phi^2).
\end{equation}
If $\phi_1(x)$ is a positive solution of \eqref{phi-eq} for $x\geq 0$ with $\mu=\beta_L^{-2}$ and $0<\phi_1<1$, and $\phi_2$  is
a positive solution of \eqref{phi-eq} for $x\leq 0$ with $\mu=\beta_U^{-2}$ and $0<\phi_2<1$, then clearly $(\Phi_L, \Phi_U):=(\phi_1, \phi_2)$ will be a solution
of \eqref{SS} provided that $\phi_1(0)=\phi_2(0)=\alpha$ and $\phi_1'(0)=\phi_2'(0)$. The converse is also true.

In the next subsection, we will describe the phase plane analysis of \eqref{phi-eq}, which will be used to solve \eqref{SS} in subsection 5.2. This method will be extended to solve the  stationary problems \eqref{ss-3-up} and \eqref{ss-1-2} in subsections 5.3 and 5.4, respectively.


\subsection{Phase plane analysis of \eqref{phi-eq}} This is a standard KPP equation, and the phase plane analysis for such equations has been done in several works; see, for example, \cite{AW, DL}. We recall the main features below for convenience of later reference and discussion.

Denote $ f(\phi) := \phi -\phi^2$. Then \eqref{phi-eq} is equivalent to the system
\begin{equation}\label{phi-psi}
\left\{
\begin{array}{l}
 \phi' (x)= \psi,\\  \psi'(x)=  \psi -\mu  f(\phi).
\end{array}
\right.
\end{equation}
A solution $(\phi(x), \psi(x))$ of \eqref{phi-psi} produces a trajectory on the $\phi\psi$-phase plane,
whose  slope is given by
\begin{equation}\label{Psiphi}
\frac{\mathrm{d}\psi }{\mathrm{d} \phi}=  1  - \mu\frac{f(\phi)}{\psi},
\end{equation}
whenever $\psi \neq0$. It is easily seen that $(0,0)$ and $(1,0)$ are the only singular points of \eqref{phi-psi}.
The eigenvalues of the Jacobian matrix at these points are
$$
\lambda^{\pm}_0=\frac{1}{2}\Big({1 \pm\sqrt{1 -4\mu}}\Big)\ \ \mbox{at } (0,0),\quad
\mbox{and} \quad
\lambda^{\pm}_1 =\frac{1}{2}\Big({1 \pm\sqrt{1+4\mu}}\Big)\ \ \mbox{at } (1,0).
$$
 Hence, $(1,0)$ is always a saddle point, with its unstable manifold tangent to $\ell_1^+:$  $\psi=\lambda_1^+(\phi-1)$ at $(1,0)$, and its stable manifold tangent to $\ell_1^-:$ $\psi=\lambda_1^-(\phi-1)$ at $(1,0)$

  The singular point $(0,0)$ is an unstable spiral if $\mu>1/4$; if $\mu=1/4$, it is an unstable node so that every trajectory approaching $(0,0)$ as $x\to-\infty$ must be tangent to $\ell$: $\psi=\frac 12 \phi$ there; if $0<\mu<1/4$,  then $(0,0)$ is again an unstable node, but this time, there exists one trajectory which approaches $(0,0)$ as $x\to-\infty$ and is tangent to $\ell^+$: $\psi=\lambda_0^+\phi$ at $(0,0)$, all other trajectories which approach $(0,0)$ as $x\to-\infty$ are tangent to $\ell^-$: $\psi=\lambda_0^-\phi$ at $(0,0)$. See Figure 2 for an illustration of these facts.
  \vskip15pt

\begin{figure}[!htbp]\label{fig:00}
\begin{center}
\includegraphics[scale=0.6]{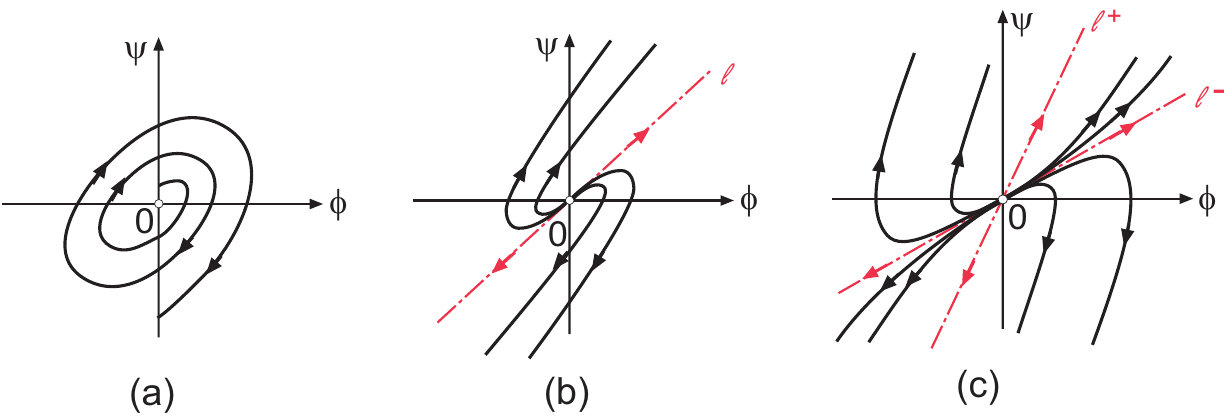}
\end{center}
\caption{\small{{\bf (a)}  $\mu>1/4$, and $(0,0)$ is an unstable spiral;
{\bf (b)} $\mu=1/4$,  and $(0,0)$ is an unstale node, with the straight line $\ell: \psi =\frac 12 \phi$  tangent to all trajectories at (0,0);
{\bf (c)} $\mu\in(0, 1/4)$, and $(0,0)$ is an unstable node, with the straight line $\ell^{-}: \psi =\lambda_0^{-} \phi$ tangent to all trajectories at (0,0) except one, which is tangent to $\ell^{+}: \psi =\lambda_0^{+} \phi$.}}
\end{figure}

The trajectories of \eqref{phi-psi} in the region
$D:=\{(\phi,\psi): 0\leq \phi \leq 1\}$ of the phase plane are described below, according to whether $\mu>1/4,\; =1/4$ or $<1/4$.
\smallskip

\noindent
{\bf Case (a):} $\mu>1/4$ (see Figure 3 (a)).  There is a unique trajectory $\Gamma^+$ in $D$, which approaches $(1,0)$ as $x\to+\infty$
and intersects the positive $\psi$ axis at some finite $x$ value, say $x=0$. $\Gamma^+$ lies above the $\phi$-axis, and is  the stable manifold of $(1,0)$ in $D$. The unstable manifold of $(1,0)$ in $D$ gives rise to a unique trajectory $\Gamma^-$, which approaches $(1,0)$ as $x\to-\infty$, and intersects the negative $\psi$-axis at some finite $x$ value. $\Gamma^-$ lies below the
$\phi$-axis.

Let  $D_{(\uparrow\,\Gamma^+)}$ denote the interior of the region  lying above $\Gamma^+$ in $D$,
$D_{(\Gamma^+\Gamma^-)}$ denote the interior of the region  lying between  $\Gamma^+$ and
$\Gamma^-$ in $D$,  and $D_{(\Gamma^-\downarrow)}$ denote the interior of the region lying below $\Gamma^-$ in $D$;
then the following hold:
\begin{itemize}
\item[(a1)]
For any $(\phi_0,\psi_0)\in D_{(\uparrow\,\Gamma^+)}$, the unique trajectory of \eqref{phi-psi} passing through $(\phi_0,\psi_0)$ intersects
the line $\phi=1$ in the positive direction, and intersects the positive $\psi$-axises in the negative direction; it remains in $D_{(\uparrow\,\Gamma^+)}$ between these two intersection points.
\item[(a2)] For any $(\phi_0,\psi_0)\in D_{(\Gamma^+\Gamma^-)}$,
the unique trajectory of \eqref{phi-psi} passing through $(\phi_0,\psi_0)$  intersects the negative $\psi$ axis in the positive direction, and it intersects the positive $\psi$-axis in the negative direction; it remains in $D_{(\Gamma^+\Gamma^-)}$ between these two intersection points, and  crosses  the positive $\phi$-axis exactly once.
$\Gamma_1$ and $\Gamma_2$ in Figure 3 (a) are two examples of such trajectories.
\item[(a3)]
 For any $(\phi_0,\psi_0)\in D_{(\Gamma^-\downarrow)}$,
the unique trajectory of \eqref{phi-psi} passing through $(\phi_0,\psi_0)$  intersects the negative $\psi$ axis in the positive direction, and it intersects the line $\phi=1$ in the negative direction; it remains in $D_{(\Gamma^-\downarrow)}$ between these two intersection points.
\end{itemize}

\medskip

\noindent
{\bf Case (c): $\mu\in(0, 1/4)$} (see Figure 3 (c)). For convenience of presentation, we consider Case (c) before Case (b). The stable manifold of $(1,0)$ in $D$ now is the trajectory $H$ which approaches $(1,0)$ as $x\to+\infty$, and approaches $(0,0)$ as $x\to-\infty$. $H$ lies above the $\phi$-axis. The unstable manifold of $(1,0)$ in $D$ is the trajectory
$\Gamma^-$, which approaches $(1,0)$ as $x\to-\infty$ and intersects the negative $\psi$-axis at some finite $x$ value. $\Gamma^-$ lies below the $\phi$-axis.

There is another special trajectory $\Gamma^*$, which approaches $(0,0)$ as $x\to-\infty$ and intersects the line $\phi=1$ at some point $(1,\gamma^*)$.
Moreover, $\Gamma^*$ lies above the line $\ell^+$: $\psi=\lambda_0^+ \phi$, and $H$ lies below the line $\ell^-$: $\psi=\lambda_0^-\phi$ (these can be easily checked by observing that the trajactories passing through a point on $\ell^{\pm}\cap\{0<\phi<1\}$ always move from below to above the line in its positive direction).

  Furthermore, if we denote by $D_{(\uparrow\,\Gamma^*)}$ the
interior of the region that is above $\Gamma^*$ in $D$, by $D_{(\Gamma^*H)}$ the interior of the region that lies between $\Gamma^*$ and $H$ in $D$,  by $D_{(H\Gamma^-)}$ the interior of the region that lies between $H$ and $\Gamma^-$ in $D$,
and by $D_{(\Gamma^-\downarrow)} $ the interior of the region lying below $\Gamma^-$ in $D$, then the following hold:
\begin{itemize}
\item[(c1)] If $(\phi_0,\psi_0)\in D_{(\uparrow\,\Gamma^*)}$, then the trajectory passing through $(\phi_0,\psi_0)$ intersects the line $\phi=1$ at
some point  $(1,\gamma)$ with $\gamma>\gamma^*$ in the positive direction, and it intersects the positive $\psi$-axis in the negative direction, and remains in
$D_{(\uparrow\,\Gamma^*)}$ between these two intersection points. 

\item[(c2)] If $(\phi_0,\psi_0)\in D_{(\Gamma^*H)}$, then the trajectory passing through $(\phi_0,\psi_0)$
 intersects the line $\phi=1$ at a point  $(1,\gamma)$ with $\gamma\in (0, \gamma^*)$  in the positive direction, say at some $x=x_0$, and it approaches $(0,0)$ as $x\to-\infty$. It remains in $D_{(\Gamma^*H)}$ for $x\in (-\infty, x_0)$. 
 All these trajectories, as well as $H$, are tangent to $\ell^-$ at $(0,0)$. In contrast, $\Gamma^*$ is tagent to $\ell^+$ at $(0,0)$.

 \item[(c3)] If $(\phi_0,\psi_0)\in D_{(H\Gamma^-)}$, then the trajectory passing through $(\phi_0,\psi_0)$
  intersects the negative $\psi$-axis in the positive direction, say at some $x=x_0$, and it approaches $(0,0)$ as $x\to-\infty$. It remains in $D_{(H\Gamma^-)}$ for $x\in (-\infty, x_0)$, and crosses the positive $\phi$-axis exactly once. 
 These trajectories are tangent to $\ell^-$ at $(0,0)$.
 \item[(c4)]  If $(\phi_0,\psi_0)\in D_{(\Gamma^-\downarrow)}$,
the unique trajectory of \eqref{phi-psi} passing through $(\phi_0,\psi_0)$  intersects the negative $\psi$ axis in the positive direction, and it intersects the line $\phi=1$ in the negative direction; it remains in $D_{(\Gamma^-\downarrow)}$ between these two intersection points.
\end{itemize}

\smallskip

\noindent
{\bf Case (b): $\mu= 1/4$} (see Figure 3 (b)). In this case we have $H$, $\Gamma^-$ and $\Gamma^*$ as in Case (c), and all the descriptions of the
trajectories in Case (c) remain valid; the only difference is that now $\ell^+$ and $\ell^-$ collapse into $\ell$.

\begin{figure}[!htbp]\label{fig:single}
\begin{center}
\includegraphics[scale=0.6]{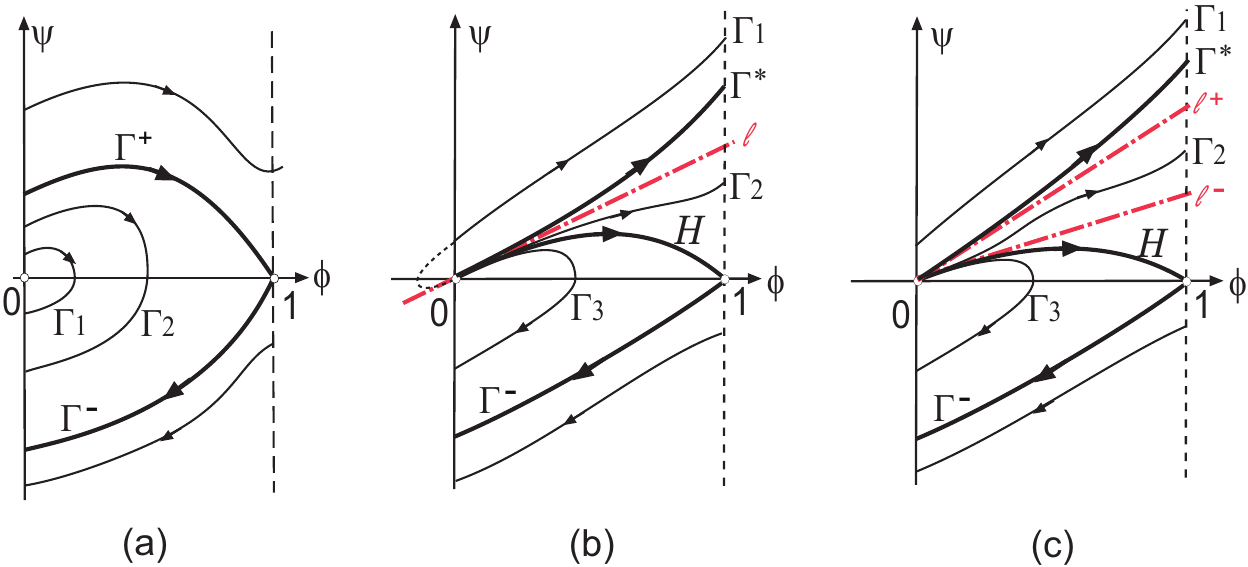}
\end{center}
\caption{\small{Trajectories of \eqref{phi-psi} in $D$. {\bf (a)} $\mu>1/4$;  {\bf (b)} $\mu=1/4$;
{\bf (c)} $0<\mu<1/4$.}}
\end{figure}

\subsection{Solutions of \eqref{ss}}

We prove Theorem \ref{prop:2-bra} by changing \eqref{ss} to its equivalent form \eqref{SS}, and using the phase plane analysis on \eqref{phi-eq}.

\begin{proof}[Proof of  Theorem \ref{prop:2-bra}] We will work with the equivalent problem \eqref{SS}.

{\bf Case (i).} If \eqref{SS} has a positive solution $(\Phi_L, \Phi_U)$ for some $\alpha\in (0,1)$, then $\Phi_U$ is a positive solution to \eqref{phi-eq}
with $\mu=\beta_U^{-2}>1/4$ for $x\in (-\infty, 0]$. Thus $\Gamma_U:=\{(\Phi_U(x), \Phi_U'(x)): x\leq 0\}$ is a trajectory of \eqref{phi-psi}
with $\mu>1/4$ that stays in the interior of $D$ in its negative direction. By the phase plane result in the previous subsection for case (a), this is possible only if $\Gamma_U$ is part of $\Gamma^-$, and thus $(\Phi_U(0), \Phi_U'(0))=:(\alpha, \psi_0)\in \Gamma^-$. In particular, $\psi_0<0$.

Now $\Gamma_L:=\{(\Phi_L(x), \Phi_L'(x)): x\geq 0\}$ is a trajectory of \eqref{phi-psi} with $\mu=\beta_L^{-2}$ starting from $(\Phi_L(0), \Phi_L'(0))=(\alpha, \psi_0)$
moving in the positive direction as $x$ increases from $x=0$.
Since $\psi_0<0$, by the phase plane anaylsis for cases (a)-(c) in the previous subsection, $\Gamma_L$ intersects the negative $\psi$-axis at a finite time $x>0$,
i.e., $\Phi_L(x)=0$ for some finite $x>0$, a contradiction to the assumption that $\Phi_L(x)>0$ for all $x>0$.
 This contradiction completes the proof.

{\bf Case (ii).} Clearly $(\Phi_L, \Phi_U)$ solves \eqref{SS} with $\alpha\in (0,1)$ if and only if
\begin{itemize}
\item[(ii)$_1$.] $\phi_1:=\Phi_L$ is a positive solution of
\eqref{phi-eq} in $[0, +\infty)$ with $\mu=\mu_1:=\beta_L^{-2}\in (0, 1/4]$, and $0<\phi_1<1$,
\item[(ii)$_2$.] $\phi_2:=\Phi_U$ is a positive solution of
\eqref{phi-eq} in $(-\infty, 0]$ with $\mu=\mu_2:=\beta_U^{-2}\in (0, 1/4]$, and $0<\phi_2<1$,
\item[(ii)$_3$.] $\phi_1(0)=\phi_2(0)=\alpha$, $\phi_1'(0)=\phi_2'(0)$.
\end{itemize}

We now use the phase plane trajectories of \eqref{phi-psi} to obtain a unique pair $(\phi_1, \phi_2)$ having the above properties.
Since $\mu_1, \mu_2\in (0, 1/4]$, we are in cases (b) or (c) described in the previous subsection.
For $\mu=\mu_i$, $i=1,2$, we denote the special trajectories $H, \Gamma^*$ and $\Gamma^-$ by $H_i,\; \Gamma^*_i$ and $\Gamma^-_i$,
respectively. Similarly, $\ell$, $\ell^+$ and $\ell^-$ are denoted by $\ell_i$, $\ell^+_i$ and $\ell^-_i$, respectively.

Given $\alpha\in (0, 1)$, the line $\phi=\alpha$ intersects $H_1$ at a point $(\alpha, \psi_0)$. Since $H_1$ lies below $\ell_1$ in case (b), and below $\ell^-_1$ in case (c), we have $\psi_0<\frac12 \alpha$ in case (b), $\psi_0<\frac12 (1-\sqrt{1-4\mu_1})\alpha< \frac12 \alpha$ in case (c).
Thus we always have $\psi_0<\frac12\alpha$. Clearly the trajectory passing through $(\alpha, \psi_0)$ at $x=0$ gives rise to a solution $\phi_1(x)$ $(x\geq 0)$ satisfying (ii)$_1$ above.  Moreover, $\phi_1'>0$ and $\phi_1(+\infty)=1$.

We next find $\phi_2$ by considering the trajectories of \eqref{phi-psi} with $\mu=\mu_2$. To satisfy (ii)$_3$,
 $\phi_2$ must be generated by the unique trajectory $\Gamma$ passing through $(\alpha, \psi_0)$ at $x=0$, in its negative direction, i.e., for $x<0$. We already know that $0<\psi_0<\frac12 \alpha\leq\frac12(1+\sqrt{1-4\mu_2})\alpha$. This implies that $(\alpha, \psi_0)$ belongs to
$D_{(\Gamma_2^*H_2)}\cup H_2\cup D_{(H_2\Gamma_1^-)}$, and by our analysis for cases (b) and (c) in the previous subsection,
we know that $\Gamma$ approches $(0,0)$ as $x\to-\infty$. Moreover, it is also easily seen that $\Gamma$ stays above the $\phi$-axis
for $x\in (-\infty, 0]$. Therefore $\phi_2'>0$  and  $\phi_2(-\infty)=0$, $\phi_2(0)=\alpha$,
$\phi_2'(0)=\psi_0=\phi_1'(0)$.

We have thus proved that in Case (ii), for any $\alpha\in (0, 1)$, \eqref{SS} and hence \eqref{ss} has a  solution. Moreover,
\[
\phi_L'>0, \; \phi_U'>0, \; \phi_L(+\infty)=1,\; \phi_U(-\infty)=0.
\]

To show uniqueness, suppose  $(\Phi_L, \Phi_U)$ is an arbitrary solution of \eqref{SS} with $\alpha\in (0,1)$. Then $\Gamma_L:=\{(\Phi_L(x), \Phi_L'(x)): x\geq 0\}$ is a trajectory of \eqref{phi-psi} with $\mu=\mu_1=\beta_L^{-2}\in (0, 1/4]$, which stays inside $D$ for all positive $x$. By the phase plane result for case (b) and (c), this is possible only if
$\Gamma_L$ is part of $H_1$, and hence $(\Phi_L(0), \Phi_L'(0))=(\alpha, \psi_0)$ and $\Phi_L\equiv \phi_1$. It follows in turn that $\Phi_U\equiv \phi_2$.
That is $(\Phi_L, \Phi_U)$ coincides with the above constructed $(\phi_1, \phi_2)$. This proves the uniqueness conclusion.

{\bf Case (iii).}  In this case, $(\Phi_L, \Phi_U)$ solves \eqref{SS} with $\alpha\in (0,1)$ if and only if
\begin{itemize}
\item[(iii)$_1$.] $\phi_1:=\Phi_L$ is a positive solution of
\eqref{phi-eq} in $[0, +\infty)$ with $\mu=\mu_1:=\beta_L^{-2}>1/4$, and $0<\phi_1<1$,
\item[(iii)$_2$.] $\phi_2:=\Phi_U$ is a positive solution of
\eqref{phi-eq} in $(-\infty, 0]$ with $\mu=\mu_2:=\beta_U^{-2}\in(0, 1/4]$, and $0<\phi_2<1$,
\item[(iii)$_3$.] $\phi_1(0)=\phi_2(0)=\alpha$, $\phi_1'(0)=\phi_2'(0)$.
\end{itemize}

In the $\phi\psi$-plane we now consider the two curves $\Gamma_1^+$ and $\Gamma^*_2$, where for $i=1,2$, the subscript $i$ is used as in Case (ii) above to indicate the trajectory in the phase plane of \eqref{phi-psi} with $\mu=\mu_i$. We claim that
\[
\mbox{ $\Gamma_1^+$ and $\Gamma^*_2$ intersects at exactly one point $(\alpha_0,\psi_0)$.}
\]
Since these are continuous curves, with $\Gamma_1^+$ above $\Gamma^*_2$ near $\phi=0$, and $\Gamma_1^+$ below $\Gamma^*_2$ near $\phi=1$, clearly they have at least one intersection point at some $\phi\in (0, 1)$. For clarity let us denote by
$\psi=\Psi_1(\phi)$ the equation for $\Gamma_1^+$, and $\psi=\Psi_2(\phi)$ the equation for $\Gamma_2^*$. Then
\[
\Psi_i'(\phi)=1-\mu_i\frac{f(\phi)}{\Psi_i(\phi)} \;\; \mbox{ for } \phi\in (0,1),\; i=1,2.
\]
Suppose the two curves intersect at $\phi=\phi_0\in (0,1)$. Then $\Phi_1(\phi_0)=\Phi_2(\phi_0)$ and due to $\mu_1>\mu_2$,
we obtain
\[
\Psi_1'(\phi_0)=1-\mu_1\frac{f(\phi_0)}{\Psi_1(\phi_0)}<1-\mu_2\frac{f(\phi_0)}{\Psi_2(\phi_0)}=\Psi_2'(\phi_0).
\]
This indicates that at any intersection point of the two curves, the slope of $\Gamma_1^+$ is always bigger than that of $\Gamma_2^*$.
This fact clearly implies that there can be no more than one intersection point. This proves our claim.

Let us also observe that in the range $\phi\in(0, \alpha_0)$, $\Gamma^+_1$ is above $\Gamma^*_2$, and for $\phi\in (\alpha_0, 1)$,
$\Gamma^+_1$ is below $\Gamma^*_2$; that is
\[
\Psi_1(\phi)>\Psi_2(\phi)\;  \mbox{ for } \phi\in (0, \alpha_0), \;\;\; \Psi_1(\phi)<\Psi_2(\phi)\;  \mbox{ for } \phi\in (\alpha_0, 1).
\]

We show next that \eqref{SS} has no solution for $\alpha\in (0, \alpha_0)$. Suppose on the contrary that it has a solution for some
$\alpha\in (0, \alpha_0)$; then we obtain a pair $(\phi_1,\phi_2)$ satisfying (iii)$_1$-(iii)$_3$ above. By our phase plane analysis for case (a), necesarily $\phi_1$ is generated by the trajectory $\Gamma^+_1: \psi=\Psi_1(\phi)$ in $\phi\geq \alpha$, and $\phi_2$ is generated by the trajectory of \eqref{phi-psi}
with $\mu=\mu_2$ passing through $(\alpha, \Psi_1(\alpha))$ in the negative  direction.

On the other hand, $\alpha<\alpha_0$ implies $\Psi_2(\alpha)<\Psi_1(\alpha)$ and hence $(\alpha, \Psi_1(\alpha))$ lies above
$\Gamma^*_2$. Thus $(\alpha, \Psi_1(\alpha))\in D_{(\uparrow\,\Gamma^*_2)}$, and by the phase plane analysis for cases (b) and (c),
we know that the trajectory passing through $(\alpha, \Psi_1(\alpha))$ intersects the positive $\psi$-axis in the negative direction. This implies that
$\phi_2(x_0)=0$ for some $x_0<0$, a contradiction to $0<\phi_2(x)<1$ for $x<0$. This proves the non-existence conclusion for $\alpha\in (0, \alpha_0)$.

Suppose now $\alpha\in [\alpha_0, 1)$. Then $\Psi_2(\alpha)\geq \Psi_1(\alpha)$ and hence $(\alpha, \Psi_1(\alpha))$ lies on or below
$\Gamma^*_2$ and above the $\phi$-axis. By our phase plane analysis for cases (b) and (c), we know that the  trajectory of \eqref{phi-psi}
with $\mu=\mu_2$ passing through $(\alpha, \Psi_1(\alpha))$ approaches $(0,0)$ as $x\to-\infty$, and it is also easily seen that it stays
above the $\phi$-axis in the negative direction from $(\alpha, \Psi_1(\alpha))$. Hence it generates a $\phi_2$ satisfying (iii)$_2$ and
\begin{equation}\label{phi2}
 \phi_2'(x)>0 \mbox{ for } x<0,\; \phi_2(-\infty)=0,\; \phi_2(0)=\alpha,\; \phi_2'(0)=\Psi_1(\alpha).
\end{equation}
Clearly $\Gamma_1^+$ in the positive direction from  $(\alpha, \Psi_1(\alpha))$ generates a $\phi_1$ satisfying (III)$_1$ and
\begin{equation}\label{phi1}
 \phi_1'(x)>0 \mbox{ for } x>0,\; \phi_1(+\infty)=1,\; \phi_1(0)=\alpha,\; \phi_1'(0)=\Psi_1(\alpha)=\phi_2'(0).
\end{equation}
We thus obtain a pair $(\phi_1, \phi_2)$ satisfying the above properties (iii)$_1$-(iii)$_3$, as required. Moreover, they also satisfy \eqref{phi2} and \eqref{phi1}.
The uniqueness follows from a similar reasoning as in Case (ii).

To complete the proof of the theorem, it remains to prove the conclusion in part {\bf (iv)} of the theorem.
We have $\mu_2\in(0, 1/4]$ and so $\phi_2$ is generated by the trajectory $\Gamma^*_2$ in the negative direction when $\alpha=\alpha_0$ in case (iii). For $\alpha\in (\alpha_0, 1)$ in case (iii), or $\alpha\in (0, 1)$ in case (ii), $\phi_2$ it is generated by the trajectory $\Gamma$  which lies below $\Gamma_2^*$ and above the $\phi$-axis.
 In the former case, the asymptotic behavior of $\phi_2$ is determined by the tangent line of $\Gamma^*_2$ at $(0,0)$, which is $\ell^+_2$, while in the latter cases, the asymptotic behavior of $\phi_2$ is determined by the tangent line of the trajectory $\Gamma$ at $(0,0)$, which is $\ell^-_2$. These facts imply in particular that $\phi_U(x)=\Phi_U(\beta_U x)\to 0$ exponentially as $x\to-\infty$. We next determine its exact behavior as $x\to-\infty$.

To simplify notations, we write $\phi=\phi_U$ and hence
\[
-\phi''+\beta_U\phi'=\phi-\phi^2,\; \phi>0\;  \mbox{ for } x<0,\; \phi(-\infty)=\phi'(-\infty)=0.
\]
To determine the exact behavior of $\phi(x)$ as $x\to-\infty$, we view $\phi$ as the solution of the following linear equation
\begin{equation}
\label{pert}
-\phi''+\beta_U\phi'=[1+\epsilon(x)]\phi,\; \phi>0 \mbox{ for } x<0,\; \phi(-\infty)=\phi'(-\infty)=0,
\end{equation}
where $\epsilon(x):=-\phi(x)\to 0$ exponentially as $x\to-\infty$. Denote
\[
k^{\pm}=\frac12 \left(\beta_U\pm \sqrt{\beta_U^2-4}\,\right);
\]
then for $\beta_U>2$, it is easily checked that
\[
\phi_+(x):=e^{k^+x},\; \phi_-(x):=e^{k^- x}
\]
are the fundamental solutions of the linear equation
\[
-\phi''+\beta_U\phi'=\phi.
\]
By standard ODE theory one sees that the perturbed problem \eqref{pert} has two fundamental solutions of the form
\[
\tilde\phi_+(x)=(1+o(1))\phi_+(x),\; \tilde\phi_-(x)=(1+o(1))\phi_-(x) \;\mbox{ as } x\to-\infty,
\]
and
\[
\tilde\phi'_+(x)=(1+o(1))\phi'_+(x),\; \tilde\phi'_-(x)=(1+o(1))\phi'_-(x) \;\mbox{ as } x\to-\infty.
\]
It follows that
\[
\phi_U(x)=a\tilde \phi_+(x)+b\tilde\phi_-(x)=a(1+o(1))e^{k^+x}+b(1+o(1))e^{k^-x} \mbox{ as } x\to-\infty,
\]
\[
\phi'_U(x)=a(1+o(1))k^+e^{k^+x}+b(1+o(1))k^-e^{k^-x} \mbox{ as } x\to-\infty.
\]
Using these and our above description of the tangent lines of the trajectories, we necessarily have $a>0=0$ when $\alpha=\alpha_0$, and $b>0$ when $\alpha>\alpha_0$. The desired behavior for $\phi_U(x)$ as $x\to-\infty$ is a simple consequence of this fact.

The case $\beta_U=2$ is more difficult to treat. In this case, $k^+=k^-=1$ and the fundamental solutions of the linear equation are
\[
\phi_1(x):=e^{x},\; \phi_2(x):=-x e^{x}.
\]
By standard  ODE theory the perturbed problem \eqref{pert} has two fundamental solutions of the form
\[
\tilde\phi_1(x)=(1+o(1))\phi_1(x),\; \tilde\phi_2(x)=(1+o(1))\phi_2(x) \;\mbox{ as } x\to-\infty,
\]
and
\[
 \tilde\phi_1'(x))=(1+o(1))\phi_1'(x)),\;  \tilde \phi_2'(x)=(1+o(1))\phi_2'(x) \;\mbox{ as } x\to-\infty.
 \]
It follows that
\[
\phi_U(x)=a\tilde \phi_1(x)+b\tilde\phi_2(x),\;
\phi'_U(x)=a \tilde \phi_1'(x)+b \tilde \phi_2'(x).
\]
By some tedious calculations, it can be shown that if the trajectory
\[
\Gamma_U:=\{(\Phi_U(x), \Phi_U'(x)): x\leq 0\}\equiv \{(\phi_U(\beta_U^{-1}x), \phi_U'(\beta_U^{-1}x)): x\leq 0\}
\]
 lies on $\Gamma_2^*$, then $a>0=b$, and if
$\Gamma_U$ lies below $\Gamma_2^*$, then $b>0$. Thus $a>0=b$ when $\alpha=\alpha_0$, and $b>0$ when $\alpha>\alpha_0$.
The desired behavior for $\phi_U(x)$ as $x\to-\infty$ thus follows.
\end{proof}


\subsection{Solutions of \eqref{ss-3-up}}

In this subsection, we prove a slightly weaker version of Theorem \ref{uul}.
As before, if we set
\[
\Phi_L(x):=\phi_L(\beta_L^{-1}x),\; \Phi_{U_1}(x):=\phi_{U_1}(\beta_{U_1}^{-1}x),\; \Phi_{U_2}(x):=\phi_{U_2}(\beta_{U_2}x),
\]
then \eqref{ss-3-up} is reduced to
\begin{equation}\label{SS-3-up}
\left\{
\begin{array}{ll}
- \Phi_{U_1}'' +\Phi_{U_1}'= \beta_{U_1}^{-2}(\Phi_{U_1} - \Phi^2_{U_1}),\;  0<\Phi_{U_1}<1, & x\in \mathbb{R}_{U_1} := (-\infty,0),\\
- \Phi_{U_2}'' +\Phi_{U_2}' = \beta_{U_2}^{-2}(\Phi_{U_2} - \Phi^2_{U_2}), \;  0<\Phi_{U_2}<1, & x\in \mathbb{R}_{U_2} := (-\infty,0),\\
- \Phi_L'' +\Phi_L' = \beta_{L}^{-2}(\Phi_L - \Phi^2_L),\; \;\;\;\;\; 0<\Phi_L<1, & x\in \mathbb{R}_L := (0,+\infty),\\
 \Phi_{U_1}(0) = \Phi_{U_2}(0) = \Phi_L (0) =\alpha\in (0,1), & \\
\xi_1\Phi_{U_1}'(0) +\xi_2 \Phi_{U_2}' (0) = \Phi_L'(0), &
\end{array}
\right.
\end{equation}
where
\[
\xi_1=\frac{a_{U_1}\beta_{U_1}}{a_L\beta_L},\; \xi_2=\frac{a_{U_2}\beta_{U_2}}{a_L\beta_L}\; \mbox{ and hence } \xi_1+\xi_2=1.
\]

We now extend the phase plane approach in the previous subsection to treat
 \eqref{SS-3-up}. Recall that, according to the behavior of the solution we have four cases which cover all the possibilities of the parameters
 $\beta_{U_1},\beta_{U_2}$ and $\beta_L$:
 \begin{itemize}
\item[\rm(I)]  $\beta_{U_1}, \beta_{U_2} <2$.
\item[\rm(II)]  $ \beta_{U_1}, \beta_{U_2}, \beta_L \geq  2$.
\item[\rm(III)]  $\beta_{U_1}, \beta_{U_2} \geq  2> \beta_L. $
\item[\rm(IV)]  $\max\{\beta_{U_1},\beta_{U_2}\}\geq 2>\min \{\beta_{U_1},\beta_{U_2}\}$.
\end{itemize}

Our results in the following four lemmas are slightly weaker than that in Theorem \ref{uul} (see comments in Remark 5.5 below).

\begin{lem}\label{I} {\rm(Case (I))} If $\max\{\beta_{U_1}, \beta_{U_2}\} \in (0, 2)$, then \eqref{ss-3-up} has no solution.
\end{lem}
\begin{lem}\label{II}{\rm(Case (II))}
 If $ \beta_{U_1}, \beta_{U_2}, \beta_L \geq  2$, then the following hold:
\begin{itemize}
\item[\bf a.] For every $\alpha\in (0,1)$, \eqref{ss-3-up} has a continuum of solutions satisfying \eqref{00}.
 \item[\bf b.] For $i=1,2$ and $j=3-i$, there exist  $\hat\alpha_{i,1}, \hat\alpha_{i,2} \in (0,1)$ with $\hat\alpha_{i,1}\leq \hat\alpha_{i,2}$
 such that  for each $\alpha\in \{\hat\alpha_{i,1}\}\cup [\hat \alpha_{i,2}, 1)$,  \eqref{ss-3-up} has a
unique solution satisfying \eqref{01}
and has no such solution for $\alpha
\in (0, \hat\alpha_{i,1})$.
\item[\bf c.] Any solution of \eqref{ss-3-up} with $\alpha\in (0,1)$ satisfies either \eqref{00} or \eqref{01}. Moreover, for any $\alpha\in (0,1)$,  there exist  $c_i=c_i(\alpha)>0$ for $i=1,2$, and a solution of  \eqref{ss-3-up},
  such that, as $x\to-\infty$, for both $i=1,2$, \eqref{slow} holds.
\end{itemize}
\end{lem}
\begin{lem}\label{III} {\rm(Case (III))}
 If $\beta_{U_1}, \beta_{U_2} \geq  2> \beta_L>0 $, then the following hold:
\begin{itemize}
\item[\bf a.]
There exist $\alpha^*_1, \alpha_2^* \in (0,1)$ with $\alpha^*_1\leq \alpha_2^*$
 such that \eqref{ss-3-up} has a
continuum of solutions satisfying \eqref{00} for each $\alpha\in (\alpha_2^{*}, 1)$, has a unique solution satisfying \eqref{01} for $\alpha\in\{\alpha_1^*, \alpha_2^*\}$, and has no such solution for $\alpha
\in (0, \alpha_1^{*})$.
\item[\bf b.] For $i=1,2$ and $j=3-i$, there exist  $\hat\alpha^*_{i,1}, \hat\alpha^*_{i,2} \in (0,1)$ with $\hat\alpha^*_{i,1}\leq \hat\alpha^*_{i,2}$ and $\hat\alpha^*_{i,1}>\alpha^*_1,
\;\hat\alpha^*_{i,2}>\alpha^*_2$,
 such that  for each $\alpha\in \{\hat\alpha^*_{i,1}\}\cup [\hat \alpha^*_{i,2}, 1)$,  \eqref{ss-3-up} has a
unique solution satisfying \eqref{01}, and has no such solution for $\alpha\in (0, \hat\alpha^*_{i,1})$.
\item[\bf c.] Any solution of \eqref{ss-3-up} with $\alpha\in (0,1)$ satisfies either \eqref{00} or \eqref{01}.

 \item[\bf d.] If $\alpha\in (\alpha_2^{*}, 1)$, then  for any solution of \eqref{ss-3-up} satisfying \eqref{00}, there exists $i\in\{1, 2\}$ and $c_i=c_i(\alpha)>0$ such that, as $x\to-\infty$, \eqref{slow} holds.
  \item[\bf e.]  If $\alpha\in \{\alpha^*_1,\alpha_2^{*}\}$, then the unique solution of \eqref{ss-3-up}  satisfies \eqref{fast}   as $x\to-\infty$.
  \end{itemize}
  \end{lem}
  \begin{lem}\label{IV}
    {\rm(Case (IV))} If $\min\{\beta_{U_1}, \beta_{U_2}\}<2\leq \max\{\beta_{U_1},\beta_{U_2}\}$, then there exist $\alpha_1^{**},\, \alpha_2^{**}\in (0, 1)$
    with $\alpha_1^{**}\leq \alpha_2^{**}$  such that
\eqref{ss-3-up} has a unique solution
 for each $\alpha\in \{\alpha_1^{**}\}\cup [\alpha_2^{**}, 1)$,   and has
no solution for $\alpha\in (0, \alpha_1^{**})$. Moreover, when $\alpha\in\{\alpha_1^{**}\}\cup[ \alpha_2^{**}, 1)$
and $\beta_{U_j}=\min\{\beta_{U_1}, \beta_{U_2}\}$, the solution $(\phi_{U_1}, \phi_{U_2},\phi_L)$ satisfies \eqref{01} with $i=3-j$,
and as $x\to-\infty$, \eqref{slow} holds for $\phi_{U_j}$ when $\alpha\in(\alpha_2^{**}, 1)$, and \eqref{fast} holds for $\phi_{U_j}$ when
 $\alpha\in\{\alpha_1^{**},\; \alpha_2^{**}\}$.
\end{lem}

\begin{re}
It can be shown that $\alpha_1^*=\alpha_2^*$,
 $\hat \alpha_{i,1}=\hat \alpha_{i,2}$,
 $\hat \alpha_{i,1}^*=\hat \alpha_{i,2}^*$ and $\alpha_1^{**}=\alpha_2^{**}$ in the above theorems $($which clearly gives all the conclusions in Theorem \ref{uul} $)$. Since the proof of these facts does not use the phase plane method, it is given in Section 4; see Lemma \ref{*=*} and Remark \ref{uniq}.
\end{re}
\noindent
{\bf Proof of Lemma \ref{I}:} We will work with the equivalent problem \eqref{SS-3-up}.
 Suppose by way of contradiction that \eqref{SS-3-up} has a solution $(\Phi_{U_1},\Phi_{U_2}, \Phi_L)$ for some $\alpha\in (0, 1)$. Since $\beta_{U_i}<2$ for $i=1,2$, the same consideration as in case (i) of Theorem \ref{prop:2-bra} shows that
\[
\Gamma_{U_i}:=\{(\Phi_{U_i}(x), \Phi_{U_i}'(x): x\leq 0\},\; i=1,2
\]
is part of the trajectory $\Gamma^-_i$ of \eqref{phi-psi} with $\mu=\beta_{U_i}^{-2}$, where the subscript $i$ in $\Gamma^-_i$ indicates the special trajectory
$\Gamma^-$ for \eqref{phi-psi} with $\mu=\beta_{U_i}^{-2}$. Therefore $\Phi'_{U_i}(0)<0$. It follows that
\[
\Phi_L'(0)=\xi_1\Phi_{U_1}'(0)+\xi_2\Phi_{U_2}'(0)<0.
\]
Now $\Gamma_L:=\{(\Phi_L(x), \Phi'_L(x): x\geq 0\}$ is the trejectory of \eqref{phi-psi} with $\mu=\beta_L^{-2}$ starting from $(\alpha, \Phi_L'(0))$ moving in the positive direction as $x$ increases from 0. As $(\alpha, \Phi'_L(0))$ is below the positive $\phi$-axis, by the phase plane result in subsection 5.1 for cases (a)-(c), $\Gamma_L$ must intersects the line $\phi=0$ at some finite $x>0$, a contradiction to $\Phi_L(x)>0$ for all $x>0$.
 This contradiction completes the proof.$\hfill\Box$
\medskip

\noindent
{\bf Proof of Lemma \ref{II}.} Clearly $(\Phi_{U_1}, \Phi_{U_2}, \Phi_L)$ solves \eqref{SS-3-up} with $\alpha\in (0,1)$ if and only if
\begin{itemize}
\item[(II)$_1$.] $\phi_1:=\Phi_{U_1}$ is a positive solution of
\eqref{phi-eq} in $(-\infty, 0]$ with $\mu=\mu_1:=\beta_{U_1}^{-2}\in (0, 1/4]$, and $0<\phi_1<1$,
\item[(II)$_2$.] $\phi_2:=\Phi_{U_2}$ is a positive solution of
\eqref{phi-eq} in $(-\infty, 0]$ with $\mu=\mu_2:=\beta_{U_2}^{-2}\in (0, 1/4]$, and $0<\phi_2<1$,
\item[(II)$_3$.] $\phi_3:=\Phi_L$ is a positive solution of
\eqref{phi-eq} in $[0,+\infty)$ with $\mu=\mu_3:=\beta_{L}^{-2}\in (0, 1/4]$, and $0<\phi_3<1$,
\item[(II)$_4$.] $\phi_1(0)=\phi_2(0)=\phi_3(0)=\alpha$, $\xi_1\phi_1'(0)+\xi_2\phi_2'(0)=\phi'_3(0)$.
\end{itemize}

For $i=1,2,3$, let $\Gamma^*_i,\; H_i,\; \Gamma_i^-$ denote, respectively, the special trajectories $\Gamma^*, H, \Gamma^-$ of \eqref{phi-psi} with $\mu=\mu_i$; similarly $\ell$, $\ell^+$ and $\ell^-$ will be denoted by $\ell_i$, $\ell^+_i$ and $\ell^-_i$, respectively.

Fix $\alpha\in (0, 1)$. Then the line $\phi=\alpha$ intersects $H_3$ in the $\phi\psi$-plane at a unique point $(\alpha,\psi_0)$. The part of the trajectory of $H_3$ starting from $(\alpha, \psi_0)$ in its positive direction clearly generates a $\phi_3$ satisfying (II)$_3$ with
\[
\mbox{ $\phi_3(0)=\alpha,\; \phi_3'(0)=\psi_0,\;  \phi_3(+\infty)=1$ and $\phi'_3>0$.}
\]
 Moreover,
$(\alpha, \psi_0)$ is the only point on the line $\phi=\alpha$ such that the trajectory of \eqref{phi-psi} with $\mu=\mu_3$ passing through it
 generates a solution $\phi$ of \eqref{phi-eq} satisfying $\phi(x)\in (0, 1)$ for $x\geq 0$.

 Since $H_3$ lies below $\ell^-_3$ (which coincides with $\ell_3$ when $\mu_3=1/4$), we have $\psi_0<\frac 12 \alpha$.
 Let $(\alpha, \psi^*_i)$ denotes the intersection point of $\phi=\alpha$ with $\Gamma^*_i$, $i=1,2$. Since $\Gamma^*_i$ lies above $\ell^+_i$ for $i=1,2$, we have
 \[
 \psi^*_i>\frac 12 (1+\sqrt{1-4\mu_i}\,)\alpha\geq \frac 12 \alpha>\psi_0,\; i=1,2.
 \]
 It follows that
 \[
 \xi_1\psi_1^*+\xi_2\psi_2^*>\psi_0.
 \]
 Therefore there is a continumm of positive constants $\psi_1^0,\psi_2^0$ satisfying
 \begin{equation}\label{phi12}
 \xi_1\psi^0_1+\xi_2\psi^0_2=\psi_0,\;0<\psi_i^0\leq \psi_i^*,\; i=1,2.
 \end{equation}

 Fix any such $\psi_1^0$ and $\psi_2^0$; for $i=1,2$, the point $(\alpha, \psi^0_i)$ lies between $\Gamma_i^*$ and the positive $\phi$-axis in the $\phi\psi$-plane, and hence
  $(\alpha, \psi^0_i)\in \Gamma_i^*\cup D_{(\Gamma_i^*\Gamma_i^-)}$, $i=1,2$. By the phase plane analysis in subsection 5.1 for cases (b) and (c), we know that, for $i=1,2$,  the trajectory $\Gamma_{\psi^0_i}$ of \eqref{phi-psi} with $\mu=\mu_i$ starting from $(\alpha, \psi_i^0)$ in the negative direction approaches $(0,0)$ as $x\to-\infty$, and lies above the $\phi$-axis. It thus generates a $\phi_i$ satisfying (II)$_i$ and
 \[
 0<\phi_i<1,\; \phi_i'>0,\; \phi_i(0)=\alpha,\; \phi_i'(0)=\psi_i^0,\;  \phi_i(-\infty)=0.
 \]

 Now $(\phi_1, \phi_2, \phi_3)$ clearly satisfies (II)$_1$-(II)$_4$, and
 \begin{equation}\label{001}
 0<\phi_i<1,\; \phi_i'>0 \mbox{ for } i=1,2,3,\; \phi_1(-\infty)=\phi_2(-\infty)=0,\; \phi_3(+\infty)=1.
 \end{equation}
 Thus \eqref{ss-3-up} has a solution satisfying \eqref{001} for every $\alpha\in (0, 1)$. Moreover, since there is a continuum of positive constants $\psi_1^0$ and $\psi^0_2$ satisfying
 \eqref{phi12}, and each such pair gives rise to a pair $\phi_1$ and $\phi_2$ as described above, we see that \eqref{ss-3-up} has a continuum of solutions satisfying \eqref{001} for each $\alpha\in (0,1)$. Let us note from the above argument that the component $\phi_L$ in $(\phi_{U_1}, \phi_{U_2}, \phi_L)$ is uniquely determined for each $\alpha\in (0, 1)$.

To prove \eqref{slow}, we note that  there is a continuum of pairs $\psi_1^0$ and $\psi_2^0$ satisfying \eqref{phi12} and for $i=1,2$, $\psi_i^0<\psi_i^*$. Thus
 $(\alpha, \psi_i^0)$ lies below $\Gamma_i^*$, and
 the trajectory $\Gamma_{\psi_i^0}$ is tangent to $\ell^-_i$ at $(0, 0)$, which yields
 the asymptotic behavior for $\phi_{U_i}(x)=\Phi_{U_i}(\beta_{U_i} x)$ described in \eqref{slow}, which can be proved by the argument used in the proof of conclusion (iv) in Theorem \ref{prop:2-bra}.

The above constructed solutions of \eqref{ss-3-up} do not exhaust all the possible solutions. There are solutions of a different kind, which we now construct. For $i=1,2$, let
 \[
 \psi=\Psi_i^*(\phi),\; \psi=\Psi_i^-(\phi),\; \psi=\Psi_3(\phi)
 \]
 be the equations of the trajectories $\Gamma^*_i, \Gamma_i^-$ and $H_3$, respectively. Then define
 \[
 \tilde \Psi_1^*(\phi):=\xi_1\Psi^*_1(\phi)+\xi_2\Psi_2^-(\phi),\; \tilde \Psi_2^*(\phi):=\xi_1\Psi_1^-(\phi)+\xi_2\Psi_2^*(\phi).
 \]
 Clearly
 \[
 \tilde\Psi^*_i(0)<0=\Psi_3(0),\; \tilde \Psi^*_i(1)>0=\Psi_3(1),\; i=1,2.
 \]
 Therefore, for $i=1,2$,
 \[
 \Sigma_i:=\{\alpha\in (0, 1): \tilde\Psi^*(\alpha)\geq \Psi_3(\alpha)\}\not=\emptyset,
 \]
 and there exist $\alpha^*_{i,1}, \alpha^*_{i,2}\in (0,1)$ such that
 $
 \alpha^*_{i,1}\leq \alpha^*_{i,2}$ and
 \[
 \tilde \Psi^*_i(\alpha^*_{i,1})=\Psi_3(\alpha^*_{i,1}),\; \tilde\Psi^*_i(\phi)<\Psi_3(\phi) \mbox{ for } \phi\in (0, \alpha^*_{i,1}),
 \]
 \[
 \tilde \Psi^*_i(\alpha^*_{i,2})=\Psi_3(\alpha^*_{i,2}),\; \tilde\Psi^*_i(\phi)>\Psi_3(\phi) \mbox{ for } \phi\in (\alpha^*_{i,2}, 1).
 \]
 Then clearly
  \[
   (0, \alpha^*_{i,1})\cap\Sigma_i=\emptyset,\; (\alpha^*_{i,2}, 1)\subset\Sigma_i.
 \]

 Fix $i\in \{1,2\}$ and $\alpha\in \Sigma_i$. The trajectory $\Gamma_i^-$ starting from $(\alpha, \Psi_i^-(\alpha))$ in its negative direction generates a $\phi_i(x)$ satisfying (II)$_i$ and
 \[
 \phi_i'<0,\; \phi_i(-\infty)=1,\; \phi_i(0)=\alpha, \;\phi_i'(0)=\Psi^-_i(\alpha).
 \]
 The trajectory $H_3$ starting from $(\alpha, \Psi_3(\alpha))$ in its positive direction generates a $\phi_3(x)$ satisfying (II)$_3$ and
 \[
 \phi_3'>0,\; \phi_3(+\infty)=1,\; \phi_3(0)=\alpha, \;\phi_3'(0)=\Psi_3(\alpha).
 \]

 Denote $j=3-i$ and set
 \[
 \psi_j^0:=\xi_j^{-1}[\Psi_3(\alpha)-\xi_i\Psi_i^-(\alpha)].
 \]
 Using $\alpha\in\Sigma_i$ we easily deduce $0<\psi_j^0\leq \Psi_j^*(\alpha)$. Let $\Gamma_{\psi_j^0}$ denote the trajectory of \eqref{phi-psi} with $\mu=\mu_j$
 passing through $(\alpha, \psi_j^0)$. Let $\phi_j(x)$ be generated by $\Gamma_{\psi_j^0}$ in its negative direction starting from $(\alpha, \psi_j^0)$; then the phase plane result indicates that $\phi_j$ satisfies (II)$_j$ and
 \[
 \phi_j'>0,\; \phi_j(-\infty)=0,\; \phi_j(0)=\alpha, \;\phi_j'(0)=\psi_j^0.
 \]
 We thus have
 \[
 \phi_3'(0)=\xi_i\phi_i'(0)+\xi_j\phi_j'(0),
 \]
 and thus $(\phi_1, \phi_2,\phi_3)$ solves \eqref{SS-3-up} with $\alpha\in\Sigma_i$, and
 \begin{equation}\label{011}
  \phi_i'<0,\; \phi_i(-\infty)=1,\; \phi_j'>0,\; \phi_j(-\infty)=0,\; \phi_3'>0,\; \phi_3(+\infty)=1.
  \end{equation}

  Conversely, suppose that $(\Phi_{U_1},\Phi_{U_2}, \Phi_L)=(\phi_1, \phi_2,\phi_3)$ is a solution of \eqref{SS-3-up}.
  Then by the phase plane result, it is easily seen that $\phi_3$ has to be generated by $\Gamma_3^+$. If one of $\phi_1'(0)$ and $\phi_2'(0)$ is negative, say $\phi_i'(0)<0$, then by the phase plane result,  $\phi_i$ has to be generated by $\Gamma_i^-$,
  and $\phi_j'(0)$ must be no bigger than $\Phi_j^*(\alpha)$. Therefore necessarily
   $\alpha\in \Sigma_i$, and $(\phi_1, \phi_2,\phi_3)$ coincides with the solution constructed above. If both $\phi_1'(0)$ and $\phi_2'(0)$ are positive, then we are back to the situation considered earlier and $(\phi_1, \phi_2,\phi_3)$ coincides with a solution satisfying \eqref{001} constructed there.

We have now proved all the conclusions in the lemma. $\hfill\Box$

\medskip

\noindent
{\bf Proof of Lemma \ref{III}.} In this case, $(\Phi_{U_1}, \Phi_{U_2}, \Phi_L)$ solves \eqref{SS-3-up} with $\alpha\in (0,1)$ if and only if
\begin{itemize}
\item[(III)$_1$.] $\phi_1:=\Phi_{U_1}$ is a positive solution of
\eqref{phi-eq} in $(-\infty, 0]$ with $\mu=\mu_1:=\beta_{U_1}^{-2}\in (0, 1/4]$, and $0<\phi_1<1$,
\item[(III)$_2$.] $\phi_2:=\Phi_{U_2}$ is a positive solution of
\eqref{phi-eq} in $(-\infty, 0]$ with $\mu=\mu_2:=\beta_{U_2}^{-2}\in (0, 1/4]$, and $0<\phi_2<1$,
\item[(III)$_3$.] $\phi_3:=\Phi_L$ is a positive solution of
\eqref{phi-eq} in $[0,+\infty)$ with $\mu=\mu_3:=\beta_{L}^{-2}>1/4$, and $0<\phi_3<1$,
\item[(III)$_4$.] $\phi_1(0)=\phi_2(0)=\phi_3(0)=\alpha$, \; $\xi_1\phi_1'(0)+\xi_2\phi_2'(0)=\phi'_3(0)$.
\end{itemize}

We first note that if we replace the trajectory $H_3$ by $\Gamma_3^+$ in the above proof for Lemma \ref{II} {\bf b}, we immediately obtain the conclusions in  Lemma \ref{III} {\bf b}, which gives all the solutions of \eqref{ss-3-up} satisfying \eqref{01}.

Next we find all the other solutions of \eqref{ss-3-up}.
In the $\phi\psi$-plane, we consider the three trajectories $\Gamma^*_1,\Gamma^*_2$ and $\Gamma^+_3$, where as in the proof of Lemma \ref{II}, for $i=1,2,3$, we use
the subscript $i$ to denote the corresponding trajectory in the phase plane of \eqref{phi-psi} with $\mu=\mu_i$. Let
\[
\psi=\Psi_1^*(\phi),\; \psi=\Psi_2^*(\phi),\; \psi=\Psi_3(\phi)
\]
be the equations of $\Gamma_1^*$, $\Gamma_2^*$ and $\Gamma^+_3$, respectively. Define
\[
\Psi_*(\phi):=\xi_1\Psi^*_1(\phi)+\xi_2\Psi^*_2(\phi).
\]
Clearly
\[
\Psi_*(0)=0<\Psi_3(0),\;\; \Psi_*(1)>0=\Psi_3(1).
\]
Therefore
\[
\Sigma^*:=\{\alpha\in (0,1): \Psi_*(\alpha)>\Psi_3(\alpha)\}\not=\emptyset,
 \]
 and
  there exist $\alpha_1^*, \alpha^*_2\in (0, 1)$ satisfying $\alpha_1^*\leq \alpha^*_2$ and
 \[
 \Psi_*(\alpha_1^*)=\Psi_3(\alpha_1^*),\; \Psi_*(\alpha)<\Psi_3(\alpha) \mbox{ for } \phi\in (0, \alpha_1^*),
 \]
  \[
 \Psi_*(\alpha_2^*)=\Psi_3(\alpha_2^*),\; \Psi_*(\alpha)>\Psi_3(\alpha) \mbox{ for } \phi\in (\alpha_2^*, 1).
 \]
 Clearly
 \[
  (0, \alpha^*_{1})\cap\Sigma^*=\emptyset,\; (\alpha^*_{2}, 1)\subset\Sigma^*.
 \]
 Since $\Gamma_1^-$ and $\Gamma_2^-$ are below the $\phi$-axis and $\Gamma_1^*$ and $\Gamma^*_2$ are above the $\phi$-axis in the $\phi\psi$-plane,
 we easily see that
 \begin{equation}\label{*<*}
 \hat\alpha^*_{i,1}>\alpha^*_1,
\;\hat\alpha^*_{i,2}>\alpha^*_2 \mbox{ for } i=1,2.
 \end{equation}

 For $\alpha=\alpha_j^*$, $ j=1,2$, clearly the trajectory $\Gamma_3^+$ starting from $(\alpha_j^*, \Psi_3(\alpha_j^*))$ in its positive direction generates a $\phi_3$ satisfying
 (III)$_3$ and
 \[
 \phi_3'>0,\; \phi_3(0)=\alpha_j^*,\; \phi_3'(0)=\Psi_3(\alpha_j^*),\;  \phi_3(+\infty)=1.
 \]

 For $i=1,2$, starting from $(\alpha_j^*, \Psi_i(\alpha_j^*))$ in the negative direction, the trajectory $\Gamma_i^*$ generates a $\phi_i$ satisfying (III)$_i$ and
 \[
 \phi_i'>0,\; \phi_i(0)=\alpha_j^*,\; \phi_i(-\infty)=0,\; \xi_1\phi_1'(0)+\xi_2\phi'_2(0)=\xi_1\Psi_1^*(\alpha_j^*)+\xi_2\Psi^*_2(\alpha_j^*)=\Psi_3(\alpha_j^*)=\phi_3'(0).
 \]
 Thus $\phi_1, \phi_2,\phi_3$ satisfy (III)$_1$-(III)$_4$ and \eqref{001}.
 This proves that for $\alpha=\alpha_j^*$, \eqref{ss-3-up} has a solution satisfying \eqref{001}. Moreover, since $\Gamma_i^*$ is tangent to $\ell^+$ at $(0,0)$,
 we can prove \eqref{fast} the same way as in the proof of conclusion (iv) in Theorem \ref{prop:2-bra}.

 To show the solution is unique,
 suppose  $(\Phi_{U_1}, \Phi_{U_2}, \Phi_L)$ is an arbitrary solution of \eqref{SS-3-up} with
  $\alpha=\alpha_j^*$. Then $(\Phi_L(x), \Phi_L'(x)): x\geq 0\}$ is a trajectory of \eqref{phi-psi} with $\mu=\mu_3>1/4$
  which does not intersects the lines $\phi=0$ and $\phi=1$ in any finite time $x>0$.
  By the phase plane analysis in subsection 5.1 for case (a), necessarily this trajectory is part of $\Gamma_3^+$, and thus
  $(\Phi_L(0), \Phi'_L(0))=
  (\alpha_j^*, \Psi_3(\alpha_j^*))$. This and \eqref{*<*} imply that $(\Phi_{U_1}, \Phi_{U_2}, \Phi_L)$ coincides with $(\phi_1, \phi_2,\phi_3)$ constructed above. The uniqueness conclusion is thus proved.

 We next consider the  case $\alpha\in (0, \alpha_1^*)$.
 Then
 \[
 \xi_1\Psi^*_1(\alpha)+\xi_2\Psi^*_2(\alpha)=\Psi_*(\alpha)<\Psi_3(\alpha).
 \]
 Suppose for contradiction that \eqref{SS-3-up} has a solution $(\Phi_{U_1}, \Phi_{U_2}, \Phi_L)$ for some
  $\alpha\in (0, \alpha_1^*)$. Then $\Phi_L(x)$ generates a trajectory of \eqref{phi-psi} with $\mu=\mu_3>1/4$ which does not intersects the lines $\phi=0$ and $\phi=1$ in any finite time $x>0$. By the phase plane analysis in subsection 5.1 for case (a), necessarily this trajectory is part of $\Gamma_3^+$, and so $(\Phi_L(0), \Phi'_L(0))=
  (\alpha, \Psi_3(\alpha))$. It follows that, for $i=1,2$,  $\Phi_{U_i}$ is generated by the trajectory of \eqref{phi-psi} with $\mu=\mu_i\in (0, 1/4]$ in the negative direction starting from some point $(\alpha, \psi_i^0)$ with $\psi_i^0>0$ and
 \[
 \xi_1\psi_1^0+\xi_2\psi_2^0=\Psi_3(\alpha)
 >\xi_1\Psi^*_1(\alpha)+\xi_2\Psi^*_2(\alpha).
 \]
 Therefore either $\psi_1^0>\Psi^*_1(\alpha)$ or $\psi_2^0>\Psi^*_2(\alpha)$ holds. For definiteness, we assume the former holds. Then the point
 $(\alpha, \psi_1^0)$ is above the trajectory $\Gamma_1^*$ in the $\phi\psi$-plane, and by the phase plane result in subsection 5.1 for cases (b) and (c), the trajectory
 of \eqref{phi-psi} with $\mu=\mu_1$, starting from $(\alpha, \psi_1^0)$  in its negative direction, intersects the line $\phi=0$ at some finite time $x$, which implies that $\Phi_{U_1}(x)=0$ for some finite $x<0$, a contradiction to $\Phi_{U_1}(x)>0$ for $x<0$.
 Therefore \eqref{ss-3-up} has no solution for $\alpha\in (0, \alpha_1^*)$.

We next consider the general case $\alpha\in \Sigma^*$. If $\Psi_*(\alpha)=\Psi_3(\alpha)$ then we can obtain a unique solution of \eqref{ss-3-up}
 satisfying \eqref{001} in the same way as for the case $\alpha\in\{\alpha^*_1, \alpha^*_2\}$. Suppose next
 \[
 \Psi_*(\alpha)>\Psi_3(\alpha).
 \]
This is the case when, for example, $\alpha\in (\alpha_2^*, 1)$. The trajectory $\Gamma^+_3$ starting from the point $(\alpha, \Psi_3(\alpha))$ in its positive direction clearly generates a $\phi_3$ satisying (III)$_3$ with
\[\mbox{$\phi_3(0)=\alpha,\; \phi_3'(0)=\Psi_3(\alpha),\; \phi_3(+\infty)=1$ and $\phi'_3>0$.}
\]
 Moreover,
$(\alpha, \Psi_3(\alpha))$ is the only point on the line $\phi=\alpha$ such that the trajectory of \eqref{phi-psi} with $\mu=\mu_3$ passing through it
 generates a solution $\phi$ of \eqref{phi-eq} satisfying $\phi(x)\in (0, 1)$ for $x\geq 0$.

 By the definition of $\alpha^*_2$ we now have
 \[
 \xi_1\Psi^*_1(\alpha)+\xi_2\Psi^*_2(\alpha)=\Psi_*(\alpha)>\Psi_3(\alpha).
 \]
 Therefore there is a continuum of positive constants $\psi_1^0$ and $\psi^0_2$ satisfying
 \[
 \xi_1\psi^0_1+\xi_2\psi^0_2=\Psi_3(\alpha),\; \psi^0_1\leq \Psi^*_1(\alpha),\; \psi^0_2\leq \Psi^*_2(\alpha).
 \]
 For each such pair $\psi_1^0, \psi_2^0$, the trajectory $\Gamma_{\psi^0_i}$ of \eqref{phi-psi} with $\mu=\mu_i$ starting from $(\alpha, \psi_i^0)$ in the negative direction approaches $(0,0)$ as $x\to-\infty$, and lies above the $\phi$-axis, for $i=1,2$. It thus generates a $\phi_i$ satisfying (III)$_i$ and
 \[
 0<\phi_i<1,\; \phi_i'>0,\; \phi_i(0)=\alpha,\; \phi_i'(0)=\psi_i^0,\; \phi_i(-\infty)=0,\; i=1,2.
 \]
 Now $(\phi_1, \phi_2, \phi_3)$ clearly satisfies (III)$_1$-(III)$_4$ and \eqref{001}.

 Since there is a continuum of positive constants $\psi_1^0$ and $\psi^0_2$ for use in the above process to generate $\phi_1$ and $\phi_2$, we see that \eqref{ss-3-up} has a continuum of solutions satisfying \eqref{001} for each $\alpha\in (\alpha_2^*,1)$. Let us note from the above argument that the
 component $\phi_L$ in $(\phi_{U_1}, \phi_{U_2}, \phi_L)$ is uniquely determined for each $\alpha$. Let us also note that due to
 $\xi_1\Psi^*_1(\alpha)+\xi_2\Psi^*_2(\alpha)>\Psi_3(\alpha)$, in the choice of $\psi_1^0$ and $\psi_2^0$,  at least one of the two inequalities in
 \[
 \psi^0_1\leq \Psi^*_1(\alpha),\; \psi^0_2\leq \Psi^*_2(\alpha)
 \]
 must be a strict inequality, which implies that either $(\alpha, \psi_1^0)$ lies below $\Gamma_1^*$ or $(\alpha, \psi_2^0)$ lies below $\Gamma_2^*$, and hence
 either the trajectory $\Gamma_{\psi_1^0}$ is tangent to $\ell^-_1$ at $(0, 0)$, or the trajectory $\Gamma_{\psi_2^0}$ is tangent to $\ell_2^-$ at $(0, 0)$, which yields
 the asymptotic behavior for $\phi_{U_1}$ and $\phi_{U_2}$ described in \eqref{slow}, which can be proved by the argument used in the proof of conclusion (iv) in Theorem \ref{prop:2-bra}.

 Finally we note that
 if $(\Phi_{U_1}, \Phi_{U_2}, \Phi_L)$ is a solution of \eqref{SS-3-up} with $\alpha\in (0, 1)$, then the phase plane results indicate that  $\Phi_L$ must agree with $\phi_3$ above, and $\Phi_{U_i}$ has to be generated as $\phi_i$ above for $i=1,2$, which implies $\alpha\in\Sigma^*$, except that we alow one of $\psi_1^0$ and $\psi_2^0$ negative; say $\psi_1^0<0$.
 Then necessarily $\psi_i^0=\Psi_3^-(\alpha)$ and thus we are back to the situation described in part {\bf b} above.

All the conclusions in the lemma are now proved. $\hfill\Box$

\medskip

\noindent
{\bf Proof of Lemma \ref{IV}.} Without loss of generality, we may assume that $0<\beta_{U_1}<2\leq \beta_{U_2}$.

We first consider the case $\beta_L<2$.
Define $\mu_i=\beta_{U_i}^{-2}$ for $i=1,2$, and $\mu_3=\beta_L^{-2}$. Clearly $\mu_1, \mu_3>1/4$ and $\mu_2\in (0, 1/4]$.  The important trajectories in our analysis of this case are $\Gamma_1^-,\; \Gamma_2^*$ and $\Gamma_3^+$. Let
\[
\psi=\Psi_1(\phi),\; \psi=\Psi_2(\phi),\;\psi=\Psi_3(\phi)
\]
be the equations of $\Gamma_1^-,\; \Gamma_2^*$ and $\Gamma_3^+$ in the $\phi\psi$-plane, respectively.

Define $\Psi(\phi):=\xi_1\Psi_1(\phi)+\xi_2\Psi_2(\phi)$. Clearly
\[
\Psi(0)<0<\Psi_3(0),\; \Psi(1)>0=\Psi_3(0).
\]
Therefore there exist $\alpha_1^{**}, \alpha^{**}_2\in (0, 1)$ satisfying $\alpha_1^{**}\leq \alpha^{**}_2$ and
 \[
 \Psi(\alpha_1^{**})=\Psi_3(\alpha_1^{**}),\; \Psi(\alpha)<\Psi_3(\alpha) \mbox{ for } \phi\in (0, \alpha_1^{**}),
 \]
  \[
 \Psi(\alpha_2^{**})=\Psi_3(\alpha_2^{**}),\; \Psi(\alpha)>\Psi_3(\alpha) \mbox{ for } \phi\in (\alpha_2^{**}, 1).
 \]

 We show next that \eqref{SS-3-up} has no solution for $\alpha\in (0,\alpha_1^{**})$.
 Suppose for contradiction that $(\Phi_{U_1},\Phi_{U_2}, \Phi_L)=(\phi_1, \phi_2,\phi_3)$ is
 a solution of \eqref{SS-3-up} for some  $\alpha\in (0,\alpha_1^{**})$. Then $\{(\phi_1(x),\phi_1'(x)): x\leq 0\}$
 is a trajectory of \eqref{phi-psi} with $\mu=\mu_1>1/4$ which stays in $D$ for all $x<0$.
 By the phase plane result for case (a) necessarily it is part of $\Gamma_1^-$. Hence
 $\phi_1'(0)=\Psi_1(\alpha)$.

 Next we look at $\{(\phi_3(x), \phi_3'(x)): x\geq 0\}$, which forms a trajectory of \eqref{phi-psi} with $\mu=\mu_3>1/4$ staying in $D$ for all $x\geq 0$.
 By the phase plane result for case (a) necessarily it is part of $\Gamma_3^+$. Hence
 $\phi_1'(0)=\Psi_3(\alpha)$.

 Using $\phi_3'(0)=\xi_1\phi_1'(0)+\xi_2\phi_2'(0)$  and $\Psi_3(\alpha)>\Psi(\alpha)$, we thus obtain
 \[
 \xi_2\phi_2'(0)=\Psi_3(\alpha)-\xi_1\Psi_1(\alpha)>\xi_2\Psi_2(\alpha).
 \]
 This implies that $(\alpha, \phi_2'(0))$ lies above $\Gamma_2^*$ in the $\phi\psi$-plane. Therefore $\{(\phi_2(s), \phi_2'(x)): x\leq 0\}$ is a trajectory
 of \eqref{phi-psi} with $\mu=\mu_2\in (0, 1/4]$ starting from $(\alpha, \phi_2'(0))$ moving in its negative direction as $x$ decreases from $0$. By the phase plane result for cases (b) and (c), due to  $(\alpha, \phi_2'(0))$ lying above $\Gamma_2^*$, this trajectory intersects the line $\phi=0$ at some finite $x<0$, a contradiction to
 $\phi_2(x)>0$ for all $x<0$. Hence \eqref{ss-3-up} has no solution for $\alpha\in (0, \alpha_1^{**})$.

 We now consider the case $\alpha=\alpha_j^{**}$, $j=1,2$. The trajectory $\Gamma_1^-$ of \eqref{phi-psi} with $\mu=\mu_1$ starting from $(\alpha_j^{**}, \Psi_1(\alpha_j^{**}))$
 in its negative direction gives rise to a function $\phi_1(x)$ satisfying \eqref{phi-eq} with $\mu=\mu_1$ for $x\leq 0$ and
 \[
 \phi_1(0)=\alpha_j^{**},\; \phi_1'(0)=\Psi_1(\alpha_j^{**}),\; \phi_1'<0,\; \phi_1(-\infty)=1.
 \]
 The trajectory $\Gamma_2^*$ starting from $(\alpha_j^{**}, \Psi_2(\alpha_j^{**}))$
 in its negative direction gives rise to a function $\phi_2(x)$ satisfying \eqref{phi-eq} with $\mu=\mu_2$ for $x\leq 0$ and
 \[
  \phi_2(0)=\alpha_j^{**},\; \phi_2'(0)=\Psi_2(\alpha_j^{**}),\;\phi_2'>0,\; \phi_2(-\infty)=0.
 \]
  The trajectory $\Gamma_3^+$ starting from $(\alpha_j^{**}, \Psi_3(\alpha_j^{**}))$
 in its positive direction gives rise to a function $\phi_3(x)$ satisfying \eqref{phi-eq} with $\mu=\mu_3$ for $x\geq 0$ and
 \[
  \phi_3(0)=\alpha_j^{**},\; \phi_3'(0)=\Psi_3(\alpha^{**}),\;\phi_3'>0,\; \phi_3(+\infty)=1.
 \]
 By the definition of $\alpha_j^{**}$ we find $\phi_3'(0)=\xi_1\phi_1'(0)+\xi_2\phi_2'(0)$. Therefore $(\Phi_{U_1}, \Phi_{U_2}, \Phi_L):=(\phi_1, \phi_2, \phi_3)$ is
 a solution of \eqref{SS-3-up} with $\alpha=\alpha_j^{**}$. The uniqueness of this solution is easily checked as before. Since $\Gamma_2^*$ is tangent to $\ell_2^+$
 at $(0,0)$,
 the behavior of $\phi_2(x)$ as $x\to-\infty$ can be precisely determined, which yields the desired behavior of $\Phi_{U_2}(x)$ as $x\to-\infty$.

 Next we consider the case $\alpha\in (\alpha_2^{**}, 1)$. The trajectory $\Gamma^+_3$ starting from the point $(\alpha, \Psi_3(\alpha))$ in its positive direction clearly generates a $\phi_3$ satisying  \eqref{phi-eq} with $\mu=\mu_3$ for $x\geq 0$ and
 \[
  \phi_3(0)=\alpha,\; \phi_3'(0)=\Psi_3(\alpha),\;\phi_3'>0,\; \phi_3(+\infty)=1.
 \]
 Moreover,
$(\alpha, \Psi_3(\alpha))$ is the only point on the line $\phi=\alpha$ such that the trajectory of \eqref{phi-psi} with $\mu=\mu_3$ passing through it
 generates a solution $\phi$ of \eqref{phi-eq} satisfying $\phi(x)\in (0, 1)$ for $x\geq 0$.

 The trajectory $\Gamma_1^-$ of \eqref{phi-psi} with $\mu=\mu_1$ starting from $(\alpha, \Psi_1(\alpha))$
 in its negative direction gives rise to a function $\phi_1(x)$ satisfying \eqref{phi-eq} with $\mu=\mu_1$ for $x\leq 0$ and
 \[
 \phi_1(0)=\alpha,\; \phi_1'(0)=\Psi_1(\alpha),\; \phi_1'<0,\; \phi_1(-\infty)=1.
 \]

 Since
 $\Psi_3(\alpha)<\xi_1\Psi_1(\alpha)+\xi_2\Psi_2(\alpha)$ and $\Psi_1(\alpha)<0$,
 \[
 \psi_2^0:=\xi_2^{-1}[\Psi_3(\alpha)-\xi_1\Psi_1(\alpha)]\in (0, \Psi_2(\alpha)),
 \]
  and hence the trajectory $\Gamma_{\psi^0_2}$ of \eqref{phi-psi} with $\mu=\mu_2$ starting from $(\alpha, \psi_2^0)$ in the negative direction approaches $(0,0)$ as $x\to-\infty$, and lies above the $\phi$-axis. It thus generates a $\phi_2(x)$ satisfying \eqref{phi-eq} with $\mu=\mu_2$ for $x\leq 0$ and
 \[
 0<\phi_2<1,\; \phi_2'>0,\; \phi_2(0)=\alpha,\; \phi_2'(0)=\xi_1\psi_2^0,\; \phi_2(-\infty)=0.
 \]
 Now $(\phi_1, \phi_2, \phi_3)$ clearly satisfies \eqref{SS-3-up} with $\alpha\in (\alpha^{**}_2, 1)$, and
 \[
 0<\phi_i<1,\;   \mbox{ for } i=1,2,3,\; \phi_1'<0, \phi_2'>0, \phi_3'>0, \phi_1(-\infty)=\phi_3(+\infty)=1,\; \phi_2(-\infty)=0.
 \]
 Conversely, if
  $(\Phi_{U_1}, \Phi_{U_2}, \Phi_L)$ is a solution of \eqref{SS-3-up} for some $\alpha\in (\alpha_2^{**}, 1)$, then using the phase plane results of subsection 5.1, it is easily seen that necessarily $(\Phi_{U_1}, \Phi_{U_2}, \Phi_L)$ coincides with $(\phi_1, \phi_2,\phi_3)$ constructed above. This proves the uniqueness conclusion.
 Since $(\alpha, \psi_2^0)$ lies below $\Gamma_2^*$, the trajectory $\Gamma_{\psi_2^0}$ is tangent to $\ell_2^-$ at $(0, 0)$, which yields
 the asymptotic behavior for $\phi_{U_2}$  when $\alpha\in (\alpha_2^{**}, 1)$.

 It remains to check the case $\beta_L\geq 2$. This time we replace $\Gamma_3^+$ in the above argument by $H_3$ and the analysis carries over  without extra difficulties.

   The proof of the lemma is now complete. $\hfill\Box$

\subsection{Solutions of \eqref{ss-1-2}}

As before, if we set
\[
\Phi_U(x):=\phi_U(\beta_U^{-1}x),\; \Phi_{L_1}(x):=\phi_{L_1}(\beta_{L_1}^{-1}x),\; \Phi_{L_2}(x):=\phi_{L_2}(\beta_{L_2}x),
\]
then \eqref{ss-1-2} is reduced to
\begin{equation}\label{SS-1-2}
\left\{
\begin{array}{ll}
- \Phi_{L_1}'' +\Phi_{L_1}'= \beta_{L_1}^{-2}(\Phi_{L_1} - \Phi^2_{L_1}),\;  0<\Phi_{L_1}<1, & x\in \mathbb{R}_{L_1} := (0, +\infty),\\
- \Phi_{L_2}'' +\Phi_{L_2}' = \beta_{L_2}^{-2}(\Phi_{L_2} - \Phi^2_{L_2}), \;  0<\Phi_{L_2}<1, & x\in \mathbb{R}_{L_2} := (0, +\infty),\\
- \Phi_U'' +\Phi_U' = \beta_{U}^{-2}(\Phi_U - \Phi^2_U),\; \;\;\;\;\; 0<\Phi_U<1, & x\in \mathbb{R}_U := (-\infty, 0),\\
 \Phi_{L_1}(0) = \Phi_{L_2}(0) = \Phi_U (0) =\alpha\in (0,1), & \\
\eta_1\Phi_{L_1}'(0) + \eta_2\Phi_{L_2}' (0) = \Phi_U'(0), &
\end{array}
\right.
\end{equation}
where
\[
\eta_1:=\frac{a_{L_1}\beta_{L_1}}{a_U\beta_U},\; \eta_2:=\frac{a_{L_2}\beta_{L_2}}{a_U\beta_U}\; \mbox{ and hence } \eta_1+\eta_2=1.
\]

We now prove a slightly weaker version of Theorem \ref{ull} by further developing the phase plane approach of the previous subsection and applying it to
 \eqref{SS-1-2}.

\begin{lem}\label{prop:1-2-weak} The following assertions hold for \eqref{ss-1-2}:

\begin{itemize}
\item[\rm(I)] If $\beta_{U} \in (0, 2)$, then \eqref{ss-1-2} has no solution.

\item[\rm(II)]  If $\beta_U, \beta_{L_1}, \beta_{L_2} \geq  2$, then   \eqref{ss-1-2} has a unique solution for every $\alpha\in (0, 1)$.

\item[\rm(III)] If $\beta_{U} \geq  2> \min\{\beta_{L_1},\beta_{L_2}\}>0 $, then there exist $\alpha^*_1, \alpha_2^* \in (0,1)$ with $\alpha^*_1\leq \alpha_2^*$
 such that \eqref{ss-1-2} has a
unique solution for each $\alpha\in \{\alpha_1^*\}\cup [\alpha_2^{*}, 1)$,  and has no solution for $\alpha
\in (0, \alpha_1^{*})$.

 \item[\rm(IV)] Whenever \eqref{ss-1-2} has a solution $(\phi_{L_1}, \phi_{L_2},\phi_U)$, we have
 \[
 \phi_{L_1}'>0,\;\phi_{L_2}'>0,\; \phi_U'>0,\; \phi_{L_1}(+\infty)=\phi_{L_2}(+\infty)=1,\; \phi_U(-\infty)=0.
 \]
 Moreover,  in case {\rm (III)}, as $x\to-\infty$,
  \begin{equation}\label{phiU}\left\{
  \begin{aligned}
  \phi_{U}(x)&=(c+o(1))e^{\frac{1}{2}(\beta_{U}+\sqrt{\beta_{U}^2-4}\,)x}\;  \mbox{ for some $c=c(\alpha)>0$ when $\alpha\in\{\alpha_1^{*},\alpha_2^*\}$,}\\
   \phi_{U}(x)&=(c+o(1))e^{\frac{1}{2}(\beta_{U}-\sqrt{\beta_{U}^2-4}\,)x}\;  \mbox{ for some $c=c(\alpha)>0$ when $\alpha\in(\alpha_2^*, 1).$}
  \end{aligned} \right.
  \end{equation}

\end{itemize}
\end{lem}

We can actually show that $\alpha_1^*=\alpha_2^*$; see Remark \ref{uniq}.

\begin{proof} We will work with the equivalent problem \eqref{SS-1-2}.
We define
\[
\mu_1:=\beta_{L_1}^{-2}, \;
\mu_2:=\beta_{L_2}^{-2},\; \mu_3:=\beta_U^{-2}.
\]

{\bf Case \rm(I).} Suppose by way of contradiction that \eqref{SS-1-2} has a solution $(\Phi_{L_1},\Phi_{L_2}, \Phi_U)$ for some $\alpha\in (0, 1)$. Since $\beta_U<2$, the same consideration as in case (i) of Theorem \ref{prop:2-bra} shows that
\[
\Gamma_{U}:=\{(\Phi_{U}(x), \Phi_{U}'(x)): x\leq 0\}
\]
is part of the trajectory $\Gamma^-_3$ of \eqref{phi-psi} with $\mu=\mu_3=\beta_{U}^{-2}$, where the subscript $3$ in $\Gamma^-_3$ indicates the special trejectory
$\Gamma^-$ of \eqref{phi-psi} with $\mu=\mu_3$. Therefore $\Phi'_{U}(0)<0$. It follows that
\[
\eta_1\Phi_{L_1}'(0)+\eta_2\Phi_{L_2}'(0)=\Phi_U'(0)<0.
\]
Therefore at least one of $\Phi_{L_1}'(0)$ and $\Phi_{L_2}'(0)$ is negative. For definiteness, we assume that $\Phi_{L_1}'(0)$ is negative.
Now $\Gamma_{L_1}:=\{(\Phi_{L_1}(x), \Phi'_{L_1}(x)): x\geq 0\}$ is the trajectory of \eqref{phi-psi} with $\mu=\mu_1=\beta_{L_1}^{-2}$ starting from $(\alpha, \Phi_{L_1}'(0))$ moving in the positive direction as $x$ increases from 0. As $(\alpha, \Phi'_{L_1}(0))$ is below the positive $\phi$-axis, by the phase plane result in subsection 5.1 for cases (a)-(c), $\Gamma_{L_1}$ must intersects the line $\phi=0$ at some finite $x>0$, a contradiction to $\Phi_{L_1}(x)>0$ for all $x>0$.
 This contradiction completes the proof for Case (I).

{\bf Case (II).}  The trajectories $H_1, H_2$ and $\Gamma_3^*$ are important in this case, where $\mu_i\in (0, 1/4]$ for $i=1,2,3$.
Let us recall that in the $\phi\psi$-plane,  $H_i$ is below the line $\ell^-_i$ for $i=1,2$,  and $\Gamma_3^*$ is above the
line $\ell_3^+$.
Let
\[
\psi=\Psi_1(\phi),\; \psi=\Psi_2(\phi),\;\psi=\Psi_3(\phi)
\]
be the equations of $H_1, H_2, \Gamma_3^*$ in the $\phi\psi$-plane, respectively.
Then the above facts imply
\[
\begin{aligned}
\Psi_i(\phi)&<\frac 12(1-\sqrt{1-4\mu_i}\,)\phi\leq \frac 12 \phi&& \mbox{ for } \phi\in (0, 1),\; i=1,2,
\\
\Psi_3(\phi)&>\frac 12(1+\sqrt{1-4\mu_3}\,)\phi\geq \frac 12 \phi &&\mbox{ for } \phi\in (0, 1).
\end{aligned}
\]
Therefore
\[
\Psi(\phi):=\eta_1\Psi_1(\phi)+\eta_2\Psi_2(\phi)<\frac 12 \phi<\Psi_3(\phi) \mbox{ for } \phi\in (0, 1).
\]

Fix $\alpha\in (0, 1)$. For $i=1,2 $ we now consider the point $(\alpha, \Psi_i(\alpha))\in H_i$. The trajectory $H_i$ from $(\alpha, \Psi_i(\alpha))$ in its positive direction
yields a solution $\phi_i(x)$ ($x\geq 0$) of \eqref{phi-eq} with $\mu=\mu_i$ satisfying
\[
\phi_i(0)=\alpha,\; \phi_i'(0)=\Psi_i(\alpha),\; \phi_i'>0,\; \phi_i(+\infty)=1, \; i=1,2.
\]
Since
\[
\psi_3^0:=\eta_1\phi_1'(0)+\eta_2\phi_2'(0)=\Psi(\alpha)<\Psi_3(\alpha),
\]
the point $(\alpha, \psi_3^0)$ lies below $\Gamma_3^*$ and is above the $\phi$-axis in the $\phi\psi$-plane. Thus by our phase plane result in subsection 5.1 for
cases (b) and (c), the trajectory of \eqref{phi-psi} with $\mu=\mu_3\in (0, 1/4]$ passing through $(\alpha, \psi_3^0)$ converges to $(0, 0)$ in its negative direction, and stays above the $\phi$-axis. This trajectory from $(\alpha, \psi_3^0)$ in the negative direction thus yields a solution $\phi_3(x)$ $(x\leq 0)$ for \eqref{phi-eq} with $\mu=\mu_3$, and it further satisfies
\[
\phi_3(0)=\alpha=\phi_1(0)=\phi_2(0),\; \phi'_3(0)=\psi_3^0=\eta_1\phi_1'(0)+\eta_2\phi_2'(0),\; \phi_3'>0,\; \phi_3(-\infty)=0.
\]
Thus $(\Phi_{L_1}, \Phi_{L_2}, \Phi_U):=(\phi_1, \phi_2, \phi_3)$ is a solution of \eqref{SS-1-2} with the above fixed $\alpha$, and it satisfies further
\[
\phi_i'>0 \mbox{ for } i=1,2,3, \;\; \phi_1(+\infty)=\phi_2(+\infty)=1,\; \phi_3(-\infty)=0.
\]
The uniqueness of this solution can be proved as before.

{\bf Case (III).}
Without loss of generality we assume $\beta_{L_1}<2$. There are two subcases:
\[
\mbox{ \rm {\bf (III-1):} $\beta_{L_2}<2$, \;\; {\bf  (III-2):}  $\beta_{L_2}\geq2$.}
\]

In case {\bf (III-1)}, the trajectories $\Gamma_1^+, \Gamma^+_2$ and $\Gamma_3^*$ are important for our analysis. Let
\[
\psi=\Psi_1(\phi),\; \psi=\Psi_2(\phi),\;\psi=\Psi_3(\phi)
\]
be the equations of $\Gamma^+_1, \Gamma^+_2, \Gamma_3^*$ in the $\phi\psi$-plane, respectively.
Then define
\[
\Psi(\phi):=\eta_1\Psi_1(\phi)+\eta_2\Psi_2(\phi).
\]
Clearly
\[
\Psi(0)>0=\Psi_3(0),\; \Psi(1)=0<\Psi_3(1).
\]
Therefore there exist $\alpha_1^{*}, \alpha_2^*\in (0, 1)$ with $\alpha_1^{*}\leq  \alpha_2^*$ such that
\begin{equation}\label{alpha*}\left\{
\begin{aligned}
\Psi(\phi)>\Psi_3(\phi) \mbox{ for  } \phi\in (0, \alpha_1^{*}),\; \Psi(\alpha_1^{*})=\Psi_3(\alpha_1^{*}),\\
\Psi(\phi)<\Psi_3(\phi) \mbox{ for  } \phi\in (\alpha_2^{*}, 1),\; \Psi(\alpha_2^{*})=\Psi_3(\alpha_2^{*}).
\end{aligned}
\right.
\end{equation}

For $i, j\in \{1,2\} $ we now consider the point $(\alpha_j^{*}, \Psi_i(\alpha_j^{*}))\in \Gamma^+_i$. The trajectory $\Gamma^+_i$ from $(\alpha_j^{*}, \Psi_i(\alpha_j^{*}))$ in its positive direction
yields a solution $\phi_i(x)$ ($x\geq 0$) of \eqref{phi-eq} with $\mu=\mu_i$ satisfying
\[
\phi_i(0)=\alpha_j^{*},\; \phi_i'(0)=\Psi_i(\alpha_j^{*}),\; \phi_i'>0,\; \phi_i(+\infty)=1, \; i=1,2.
\]
The trajectory $\Gamma_3^*$ from $(\alpha_j^{*}, \Psi_3(\alpha_j^{*}))$ in its negative direction
yields a solution $\phi_3(x)$ ($x\leq 0$) of \eqref{phi-eq} with $\mu=\mu_3$ satisfying
\[
\phi_3(0)=\alpha_j^{*},\; \phi_3'(0)=\Psi_3(\alpha_j^{*}),\; \phi_3'>0,\; \phi_3(-\infty)=0.
\]
By \eqref{alpha*}, $\phi_3'(0)=\eta_1\phi_1'(0)+\eta_2\phi'_2(0)$ and thus $(\Phi_{L_1}, \Phi_{L_2}, \Phi_U):=(\phi_1, \phi_2, \phi_3)$ is a solution of \eqref{SS-1-2} with $\alpha=\alpha_j^{*}$, and satisfies further
\[
\phi_i'>0 \mbox{ for } i=1,2,3, \;\; \phi_1(+\infty)=\phi_2(+\infty)=1,\; \phi_3(-\infty)=0.
\]

If $(\Phi_{L_1}, \Phi_{L_2}, \Phi_U)$ is any solution of \eqref{SS-1-2} with $\alpha=\alpha_j^{*}$, then for $i=1,2$, $\{(\Phi_{L_i}(x),\Phi_{L_i}'(x)): x\geq 0\}$ has to be part of $\Gamma^+_i$ as any other trejectory of \eqref{phi-psi} with $\mu=\mu_i$ intersects the lines $\phi=0$ or $\phi=1$ in the positive direction. This  implies that
$(\Phi_{L_1}, \Phi_{L_2}, \Phi_U)$ must agree with the above constructed $(\phi_1, \phi_2, \phi_3)$. We have thus proved the uniqueness.

Next we consider the case $\alpha\in (0, \alpha_1^{*})$. Suppose for contradiction that $(\Phi_{L_1}, \Phi_{L_2}, \Phi_U)$ is a solution of \eqref{SS-1-2}
for some  $\alpha\in (0, \alpha_1^{*})$. Then the same consideration as in the uniqueness proof above shows that for $i=1,2$, $\{(\Phi_{L_i}(x),\Phi_{L_i}'(x)): x\geq 0\}$ forms part of $\Gamma^+_i$. Therefore
\[
\Phi_{L_i}'(0)=\Psi_i(\alpha) \mbox{ for } i=1,2.
\]
Due to \eqref{alpha*} we obtain
\[
\Phi_U'(0)=\eta_1\Phi_{L_1}'(0)+\eta_2\Phi_{L_2}'(0)=\Psi(\alpha)>\Psi_3(\alpha).
\]
Thus the point $(\alpha, \Phi'_U(0))$ lies above $\Gamma_3^*$. Now $\Gamma_U:=\{(\Phi_U(x),\Phi_U'(x)): x\leq 0\}$ is the trajectory of \eqref{phi-psi} with $\mu=\mu_3$ starting from $(\alpha, \Phi'_U(0))$ moving in the negative direction as $x$ decreases from 0. By the phase plane result for cases (b) and (c), this trajectory intersects the line $\phi=0$ in its negative direction, that is, $\Phi_U(x)=0$ for some finite $x<0$, which is a contradiction to $\Phi_U(x)>0$ for all $x<0$. This proves the nonexistence result for $\alpha\in (0, \alpha_1^{*})$.

Suppose now $\alpha\in (\alpha_2^*, 1)$. For $i=1,2$, the trajectory $\Gamma^+_i$ from $(\alpha, \Psi_i(\alpha))$ in its positive direction
yields a solution $\phi_i(x)$ ($x\geq 0$) of \eqref{phi-eq} with $\mu=\mu_i$ satisfying
\[
\phi_i(0)=\alpha,\; \phi_i'(0)=\Psi_i(\alpha),\; \phi_i'>0,\; \phi_i(+\infty)=1, \; i=1,2.
\]
By \eqref{alpha*},
$
\Psi(\alpha)<\Psi_3(\alpha)$, and hence the point $(\alpha, \Psi(\alpha))$ lies below $\Gamma_3^*$ and above the $\phi$-axis.
Therefore the trajectory $\Gamma_\alpha$ of \eqref{phi-psi} with $\mu=\mu_3$ starting from $(\alpha, \Psi(\alpha))$ in its negative direction
yields a solution $\phi_3(x)$ ($x\leq 0$) of \eqref{phi-eq} with $\mu=\mu_3$ satisfying
\[
\phi_3(0)=\alpha,\; \phi_3'(0)=\Psi(\alpha)=\eta_1\phi_1'(0)+\eta_2\phi_2'(0),\; \phi_3'>0,\; \phi_3(-\infty)=0.
\]
Thus $(\Phi_{L_1}, \Phi_{L_2}, \Phi_U):=(\phi_1, \phi_2, \phi_3)$ is a solution of \eqref{SS-1-2} with $\alpha\in (\alpha_2^*, 1)$, and satisfies further
\[
\phi_i'>0 \mbox{ for } i=1,2,3, \;\; \phi_1(+\infty)=\phi_2(+\infty)=1,\; \phi_3(-\infty)=0.
\]

If $(\Phi_{L_1}, \Phi_{L_2}, \Phi_U)$ is any solution of \eqref{SS-1-2} with $\alpha\in (\alpha_2^*, 1)$, then for $i=1,2$, $\{(\Phi_{L_i}(x),\Phi_{L_i}'(x)): x\geq 0\}$ has to be part of $\Gamma^+_i$ as any other trejectory of \eqref{phi-psi} with $\mu=\mu_i$ intersects the lines $\phi=0$ or $\phi=1$ in the positive direction. This  implies that
$(\Phi_{L_1}, \Phi_{L_2}, \Phi_U)$ must agree with the above constructed $(\phi_1, \phi_2, \phi_3)$. We have thus proved the uniqueness.

\medskip

 We now consider case {\bf (III-2)}. In this case the trajectories $\Gamma_1^+, H_2$ and $\Gamma_3^*$ are important for our analysis. Let
\[
\psi=\Psi_1(\phi),\; \psi=\Psi_2(\phi),\;\psi=\Psi_3(\phi)
\]
be the equations of $\Gamma^+_1, H_2, \Gamma_3^*$ in the $\phi\psi$-plane, respectively.
Then define
\[
\Psi(\phi):=\eta_1\Psi_1(\phi)+\eta_2\Psi_2(\phi).
\]
We have
\[
\Psi(0)>0=\Psi_3(0),\; \Psi(1)=0<\Psi_3(1).
\]
Therefore there exists $\alpha_1^{*}, \alpha_2^*\in (0, 1)$ with $\alpha_1^{*}\leq  \alpha_2^*$ such that
\eqref{alpha*} holds. The rest of the proof is parallel to case (III-1) above, and is thus omitted.

To complete the proof, it remains to prove
\eqref{phiU}. From our proofs above, we know that $(\Phi_U(x), \Phi_U'(x))$ goes to $(0,0)$ along a trajectory $\Gamma_U$ of \eqref{phi-psi} with $\mu=\beta_U^{-2}$. When $\alpha\in\{\alpha_1^*, \alpha_2^*\}$, $\Gamma_U$ is part of $\Gamma^*$ which is tangent to $\ell^+: \psi=\frac12 (1+\sqrt{1-4\beta_U^{-2}}\,)\phi$ at $(0,0)$,  and when $\alpha\in (\alpha_2^*, 1)$, $\Gamma_U$ is below $\Gamma^*$ and is tangent to $\ell^-: \psi=\frac 12 (1-\sqrt{1-4\beta_U^{-2}}\,)\phi$ at $(0,0)$.
From these facts we obtain \eqref{phiU} by a standard calculation.

The proof of the lemma is now complete.
\end{proof}

\section{Further discussions}

We have rather completely analysed the dynamical behaviour of the Fisher-KPP equation over three sample graphs, to gain insight to the spreading and growth of a new  or invasive species  in a local river system represented by these simple graphs, which may be a part of a  bigger river network. The approach is motivated by the classical works  \cite{F} and  \cite{KPP}, and by several recent works on similar equations but over finite graphs.

The dynamics of the models over finite graphs is nice and simple, but does not capture features that describe the invasion behaviour of the species. Our work  indicates that the model
over  infinite graphs is much more difficult to analyse, but is capable of capturing several new features of the population dynamics, even in the very simple yet basic cases considered here. The long-time dynamics of the model in all the cases considered here can only be one of  three types: {\it washing out, persistence at carrying capacity, and persistence below carrying capacity}. And which case happens is determined completely by the water flow speeds of the river branches in the local river system (see Remark \ref{re1} for details), and the water flow speeds of the branches are affected by both the topological structure of the graph representing the river branches and by the cross section area of each river branche, as explained in the introduction. In a practical situation, these speeds can be easily obtained, though the threshold speed $c_*$ need to be worked out carefully due to the rescaling in the model.

Since the local river system is represented by graphs with edges of infinite length, the model is only good to describe the population dynamics during the spreading period within the local river system; once the spreading has gone beyond the local system, the population dynamics need to be modelled differently. Therefore, the phenomena revealed in this paper provide insights to the population dynamics only during the spreading phase of the species.
Our results suggest that in a complex river system, the population distribution of the species during its spreading phase can be rather varied, depending largely on the local environment, and in this research, we have focused particularly on the local river structure and the water flow speed (which are closely related). Moreover, in a local system, the water flow speeds in the upper river branches play a more important role than that of the lower river branches  in determining the local population distribution during its spreading phase, which may appear not evident, but seems reasonable, since high population density in lower rivers due to better living environment (slow water flow speed  here) is less easy to  propagate to upper rivers than the other way round.

In a sense, we have succeeded in using some simple models to reveal the effect of the structure of a local river system on the population dynamics of a new species,
although the spreading speed issue is not treated yet. But much generality is sacrificed here.
For example,  each river branch is treated like a tube with a constant cross section, without spatial variation, and the water flow speed in each river branch is thus a constant. Seasonal variation is also ignored. Nevertheless, we believe that the basic features revealed in these simple models should retain in many realistic  situations, for example, when the environment is time-periodic.

Since the very complicated set of stationary solutions are analysed through a phase plane argument, the method is difficult to extend to treat  heterogenous situations. On the other hand, since for each setting of water flow speeds, only one stationary solution is selected to attract all the solutions as time goes to infinity, to gain a complete understanding of the long-time behaviour of the model, it is not necessary to know all the stationary solutions. This could be a direction for further extension of the research here.

\end {document}